\documentclass[reqno]{amsart}
\usepackage{amsmath,amssymb,latexsym, url,graphicx,newtxtext,newtxmath,tikz}
\usepackage[margin=1in]{geometry}
\usepackage{enumitem}
\usepackage{booktabs}
\usepackage[colorlinks=true,linkcolor=blue,citecolor=blue,urlcolor=blue]{hyperref}

\newtheorem{theorem}{Theorem}[section]
\newtheorem{lemma}[theorem]{Lemma}
\newtheorem{proposition}[theorem]{Proposition}
\newtheorem{problem}[theorem]{Problem}
\newtheorem{conjecture}[theorem]{Conjecture}
\newtheorem{corollary}[theorem]{Corollary}
\theoremstyle{definition}
\newtheorem{definition}[theorem]{Definition}
\newtheorem{remark}[theorem]{Remark}
\newtheorem{question}[theorem]{Question}

\DeclareMathOperator{\SO}{SO}
\DeclareMathOperator{\tr}{tr}
\DeclareMathOperator{\dist}{dist}
\DeclareMathOperator{\diag}{diag}
\DeclareMathOperator{\Ssup}{S}
\newcommand{\R}{\mathbb{R}}
\newcommand{\Z}{\mathbb{Z}}
\newcommand{\Q}{\mathbb{Q}}
\newcommand{\Sph}{\mathbb{S}}
\newcommand{\Dom}{\mathcal{D}}
\newcommand{\Kdom}{\mathcal{D}_{\mathrm h}}
\newcommand{\Pcyl}{\mathcal{P}_{\mathrm{cyl}}}
\newcommand{\Dsame}{D_{\mathrm{same}}}
\newcommand{\Dcross}{D_{\mathrm{cross}}}

\begin{document}

\title[Insertion and expansion of cylinders]{Insertion and expansion of cylinders around
the unit ball} 
\author[I. Mati\'c]{
Ivan Mati\'c}
\address{
Baruch College, City University of New York}\email{ivan.matic@baruch.cuny.edu}
\author[R. Radoi\v ci\'c]{
Rado\v s Radoi\v ci\'c} \address{
Baruch College, City University of New York}\email{rados.radoicic@baruch.cuny.edu}
\author[A. Spomer]{
Andreas Spomer}
\address{
Universit\"at zu K\"oln} \email{andreas.spomer1996@gmail.com} 

\begin{abstract}
A unit ball is surrounded with a ring of $n$ equal infinite circular cylinders. Each cylinder touches the ball and its
two neighbors. Can one more cylinder of the same radius be added if the configuration may be rearranged? Can the
$n$ cylinders instead be enlarged, and by how much? We prove that insertion is impossible for $3\le n\le7$ and
possible for $n\ge9$. The case $n=8$ remains open. The case $n=6$ is Kuperberg's problem from 1990, which was
resolved recently. We prove that the ring of four cylinders cannot be expanded. This settles a conjecture of Ogievetsky
and Shlosman from 2018. For $n\ge5$ the largest possible expansion radius remains unknown.
\end{abstract}
\maketitle

\section{Introduction}\label{sec:intro}

\begin{figure}[ht]
\centering
\begin{tikzpicture}[scale=1, line cap=round, line join=round,
  cylbody/.style={fill=gray!6, draw=black, line width=0.5pt},
  cyltop/.style={fill=gray!22, draw=black, line width=0.5pt},
  declare function={az=25; el=52;      
    mm=7; om=sin(180/mm)/(1-sin(180/mm)); R=1+om;   
    Ht=0.85; Hb=1.15;                  
    XC(\q)=R*sin(\q-az); DD(\q)=-R*cos(\q-az)*sin(el); }]
\foreach \q in {205, 256.4286, 153.5714}{
  \filldraw[cylbody] ({XC(\q)-om}, {DD(\q)+Ht*cos(el)})
    -- ({XC(\q)-om}, {DD(\q)-Hb*cos(el)})
    arc[start angle=180, end angle=360, x radius={om}, y radius={om*sin(el)}]
    -- ({XC(\q)+om}, {DD(\q)+Ht*cos(el)})
    arc[start angle=0, end angle=180, x radius={om}, y radius={om*sin(el)}] -- cycle;
  \filldraw[cyltop] ({XC(\q)},{DD(\q)+Ht*cos(el)}) circle[x radius={om}, y radius={om*sin(el)}];
}
\shade[ball color=gray!35] (0,0) circle (1);
\draw[line width=0.5pt] (0,0) circle (1);
\foreach \q in {307.8571, 102.1429, 359.2857, 50.7143}{
  \filldraw[cylbody] ({XC(\q)-om}, {DD(\q)+Ht*cos(el)})
    -- ({XC(\q)-om}, {DD(\q)-Hb*cos(el)})
    arc[start angle=180, end angle=360, x radius={om}, y radius={om*sin(el)}]
    -- ({XC(\q)+om}, {DD(\q)+Ht*cos(el)})
    arc[start angle=0, end angle=180, x radius={om}, y radius={om*sin(el)}] -- cycle;
  \filldraw[cyltop] ({XC(\q)},{DD(\q)+Ht*cos(el)}) circle[x radius={om}, y radius={om*sin(el)}];
}
\draw[thin] ({XC(az+180+3*360/mm)},{DD(az+180+3*360/mm)+Ht*cos(el)})
  -- ({XC(az+180+3*360/mm)+om},{DD(az+180+3*360/mm)+Ht*cos(el)})
  node[midway, anchor=north, inner sep=2pt] {$\omega_{7}$};
\end{tikzpicture}
\begin{tikzpicture}[scale=0.85,
  declare function={mm=7; om=sin(180/mm)/(1-sin(180/mm)); R=1+om; a0=90; }]
\fill[gray!15] (0,0) circle (1);
\draw[line width=0.5pt] (0,0) circle (1);
\foreach \k in {0,...,6}{
  \draw[fill=gray!4, line width=0.5pt]
    ({R*cos(a0+\k*360/mm)}, {R*sin(a0+\k*360/mm)}) circle[radius={om}];
}
\draw[thin] ({R*cos(a0)},{R*sin(a0)}) -- ({(R+om)*cos(a0)}, {(R+om)*sin(a0)})
  node[midway, anchor=west, inner sep=2pt] {$\omega_{7}$};
\end{tikzpicture}

\caption{The ring of seven pairwise tangent cylinders around the unit ball (left)
and its equatorial cross-section (right).}
\label{fig:ring3d}
\end{figure}
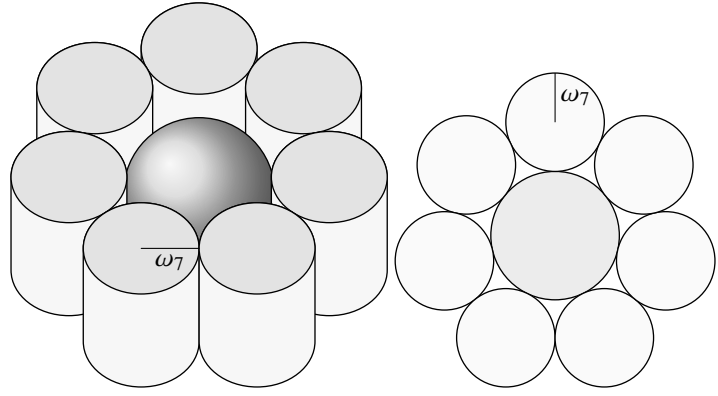

\subsection{Formulation}\label{subsection:formulation}

A \emph{cylinder} means a closed solid right circular cylinder of infinite length. Two cylinders
are \emph{non-overlapping} if their interiors are disjoint. A cylinder \emph{touches} the unit ball
$B$ if it meets $B$ but not the interior of $B$.
For every integer $n\ge3$ it is
possible for $n$ pairwise non-overlapping cylinders of radius
\begin{align}
\omega_n&=\frac{\sin\frac{\pi}{n}}{1-\sin\frac{\pi}{n}} \label{eqn:omegaNDefinition}
\end{align}
to simultaneously touch the unit ball. The canonical way is the \emph{ring configuration} of
Figure~\ref{fig:ring3d}. 
The cylinders are perpendicular to a common plane, and their axes pass through the vertices of a
regular $n$-gon. Each of them is tangent to its two neighbors and to the ball. Much of the earlier work on cylinders touching a ball concerns the following two questions, for each
fixed $n\ge3$.

\begin{question}[Insertion]\label{Q:insertion}
Can one more cylinder of radius $\omega_n$ be added, for a total of $n+1$? The cylinders may be
rearranged, but they must be non-overlapping and touch the unit ball.
\end{question}

\begin{question}[Expansion]\label{Q:expansion}
Let $\omega'>\omega_n$. Can the $n$ cylinders be enlarged to the common radius $\omega'$ and rearranged
so that they are non-overlapping and touch the unit ball? For which $\omega'$ is the answer yes?
\end{question}
 
\subsection{Cylinder-packing pairs}\label{subsection:Pcyl}
All results of this paper are statements about the following set of pairs.

\begin{definition}\label{def:Pcyl}
$\Pcyl$ is the set of pairs $(n,\omega)\in\mathbb N\times(0,\infty)$ such that no more than $n$
pairwise non-overlapping cylinders of radius $\omega$ can simultaneously touch the unit ball. We call
these \emph{cylinder-packing pairs}.
\end{definition}

The following lemma is elementary and is also a consequence of Lemma~\ref{lem:monotone} in
Section~\ref{sec:reduction}.

\begin{lemma}\label{lem:Pcylmono}
If $(n,\omega)\in\Pcyl$, then $(n',\omega')\in\Pcyl$ for every integer $n'\ge n$ and every
$\omega'\ge\omega$.
\end{lemma}

In the ring, adjacent axes are at distance $2(1+\omega_n)\sin\frac\pi n=2\omega_n$. So the ring
consists of $n$ non-overlapping cylinders of radius $\omega_n$. With Lemma~\ref{lem:Pcylmono} this
gives
\begin{align}\label{eq:ringLowerBound}
(n-1,\;\omega)\;\notin\;\Pcyl\quad\text{for every }0<\omega\le\omega_n\ \ (n\ge3).
\end{align}
In the pair language, insertion at $n$ is possible if and only if $(n,\omega_n)\notin\Pcyl$.
Expansion to the radius $\omega'$ is possible if and only if $(n-1,\omega')\notin\Pcyl$. Kuperberg's
question is whether $(6,1)\in\Pcyl$.

Two quantities are read off from $\Pcyl$. For $\omega>0$ let $\mathcal N(\omega)$ be the smallest
$n$ with $(n,\omega)\in\Pcyl$. It is the largest number of non-overlapping cylinders of radius
$\omega$ that can touch the unit ball. For $n\ge2$ let $R_n$ be the largest $\omega$ with
$(n-1,\omega)\notin\Pcyl$. It is the largest radius at which $n$ non-overlapping cylinders can touch
the unit ball. The maximum exists by Corollary~\ref{cor:attain}. Question~\ref{Q:expansion} asks for
$R_n$, and \eqref{eq:ringLowerBound} gives $R_n\ge\omega_n$.

\subsection{Results}\label{subsection:results}

\begin{theorem}[No insertion for $3\le n\le7$]\label{thm:A}
For every $n\in\{3,4,5,6,7\}$ we have $(n,\omega_n)\in\Pcyl$. Hence insertion is impossible for these
$n$.
\end{theorem}

The case $n=6$ is Kuperberg's problem, solved in \cite{MRsix,MRCode2026}. Theorem~\ref{thm:A} places
it in a uniform family.

\begin{theorem}[Insertion for $n\ge9$]\label{thm:Ge9}
For every $n\ge9$ we have $(n,\omega_n)\notin\Pcyl$. Hence insertion is possible for these $n$.
\end{theorem}

\begin{theorem}[Ogievetsky--Shlosman conjecture]\label{thm:OS}
For every $\omega>1+\sqrt2$, $(3,\omega)\in\Pcyl$. Equivalently, $R_4=1+\sqrt2$.
\end{theorem}

Theorem~\ref{thm:OS} proves the conjecture of Ogievetsky and Shlosman \cite{OSsix,OSPlatonic,OSMorse}.
We also find all configurations of four non-overlapping cylinders of radius $1+\sqrt2$ touching the
unit ball. Theorem~\ref{thm:OSequality} shows that they form a single one-parameter family. The ring of
five cylinders can be expanded (Remark~\ref{rem:five}). Hence $n=4$ is the largest $n$ for which the
ring cannot be expanded.

\subsection{Background}\label{subsec:background}
Kuperberg \cite{K} posed Question~\ref{Q:insertion} for $n=6$ in 1990 (see also
\cite[Section~2.4, Conjecture~16]{BMP}). In this case $\omega_6=1$, so the ball and the cylinders
have the same radius. Heppes and Szab\'o \cite{HS} proved that $(8,1)\in\Pcyl$. Bra\ss{} and Wenk
\cite{BW} improved this to $(7,1)\in\Pcyl$. Yardimci \cite{Yar} showed that seven cylinders of radius
$1.04965$ cannot touch the ball. Kuperberg's question was answered in \cite{MRsix,MRCode2026}, which
proves $(6,1)\in\Pcyl$.

Kuperberg \cite{K} also asked for the value of $R_n$ \cite[Section~2.4, Problem~17]{BMP}. Ogievetsky
and Shlosman conjectured that $R_4=1+\sqrt2$ \cite{OSsix,OSPlatonic,OSMorse}.

Firsching \cite{F} found six cylinders of radius $1.049659>\omega_6$ around the unit ball and gave
exact coordinates for them. His numerical search also found expansions of the rings of five and seven
cylinders. Ogievetsky and Shlosman \cite{OSsix,OSCm,OSO6,OSPlatonic} developed a general method for
deforming the ring. They found a configuration $C_{\mathfrak m}$ of six cylinders of radius
$\frac18\bigl(3+\sqrt{33}\bigr)=1.0930\ldots$ and proved that it is a local maximum \cite{OSCm}. They
showed that the ring can be expanded for every even $n\ge6$ \cite[Section~6]{OSsix}. The survey
\cite{OSunlocking} describes this circle of ideas.

Our upper bounds use the two-point bound. It is the Lov\'asz theta number \cite{Lovasz} of a Cayley
graph on the rotation group. De Laat and Vallentin \cite{dLV} set up theta numbers for such infinite
packing graphs. For packings of spherical caps the two-point bound is the linear programming bound of
Delsarte, Goethals, and Seidel \cite{DGS}. For packings on the rotation group it was first worked out
by de Oliveira Filho, Spomer, and Vallentin \cite{SOV}. They parametrize $\SO(3)$ by Euler angles and
use Bochner's theorem \cite{Bochner} to describe the functions of positive type. We use quaternions
instead. Then every hypothesis of the bound becomes a statement about integers
(Section~\ref{sec:quat}). The surveys \cite{OV,VICM} and the thesis \cite{Spomer} describe the history
of these methods.

\subsection{Methods}\label{subsec:methodNonMembership}
To prove $(n,\omega_n)\notin\Pcyl$ we construct $n+1$ cylinders of radius $\omega_n$ touching $B$.
After the rescaling of Section~\ref{sec:reduction}, this means $n+1$ lines tangent to the unit sphere
with pairwise distances at least $2\sin\frac\pi n$. For large $n$ we use the \emph{pinwheel} of
Section~\ref{sec:pinwheel}. It consists of two groups of tangent lines at latitudes $\pm\varphi$,
tilted by an angle $\delta$ and staggered in longitude. This is the family through which Ogievetsky
and Shlosman unlock the ring \cite[Section~6]{OSsix}. We fix one rational pair
$(\sin\varphi,\tan\delta)$, so that every inequality becomes a comparison of rational numbers.

\begin{theorem}\label{thm:B}
For every odd $n\ge19$ and every even $n\ge38$ we have $(n,\omega_n)\notin\Pcyl$.
\end{theorem}

The remaining $19$ values $n\in\{9,10,\dots,18\}\cup\{20,22,\dots,36\}$ are treated in
Section~\ref{sec:window} by explicit configurations with rational data.

To prove $(n,\omega_n)\in\Pcyl$ we use the two-point bound on the rotation group $\SO(3)$
\cite{SOV,VICM}. A finite list of rational data, called a \emph{certificate}, proves each
case $4\le n\le7$ of Theorem~\ref{thm:A}. The case $n=3$ follows from Theorem~\ref{thm:OS}, since
$\omega_3>1+\sqrt2$. The certificates and the configurations are checked in exact arithmetic by a
standalone verifier \cite{MRVer}. The proof of Theorem~\ref{thm:OS} is elementary and does not use a
computer.

\subsection{Organization} 
Section~\ref{sec:reduction} replaces the cylinders by lines tangent to the unit sphere.
Section~\ref{sec:OSC} contains the elementary proof of the Ogievetsky--Shlosman conjecture
(Theorem~\ref{thm:OS}). Sections~\ref{sec:pinwheel}--\ref{sec:window} construct $n+1$ cylinders of
radius $\omega_n$ for every $n\ge9$. This proves that insertion is possible for these $n$
(Theorem~\ref{thm:Ge9}). Sections~\ref{sec:DBSequence}--\ref{sec:cover} prove that insertion is
impossible for $3\le n\le7$ (Theorem~\ref{thm:A}). This part uses the two-point bound and rational
certificates. Section~\ref{sec:open} lists open problems, including the case $n=8$.
Appendix~\ref{app:data} describes the deposited data and the verifier.

\section{The configuration space and the separation parameter}\label{sec:reduction}

Fix $\omega>0$ and let $B$ be the unit ball centered at $O$. A cylinder of radius $\omega$
touches $B$ if and only if its axis is tangent to the sphere of radius $1+\omega$ centered at $O$.
Two cylinders of radius $\omega$ are non-overlapping if and only if their axes are at distance at
least $2\omega$. The homothety with center $O$ and ratio $\frac1{1+\omega}$ maps the axes to lines
tangent to the unit sphere $\Sph^2$ and multiplies all distances by $\frac1{1+\omega}$. This
proves the following lemma, which we use throughout Sections~\ref{sec:pinwheel}--\ref{sec:window}.

\begin{lemma}\label{lem:dictionary}
Let $N\ge2$ and $\omega>0$. There exist $N$ pairwise non-overlapping cylinders of radius
$\omega$ touching $B$ if and only if there exist $N$ lines tangent to $\Sph^2$
with pairwise distances at least
$
\tau(\omega)=\frac{2\omega}{1+\omega}$.
\end{lemma}

The map $\omega\mapsto\tau(\omega)$ is an increasing bijection of $(0,\infty)$ onto $(0,2)$. So for
$\omega>0$ and $\tau\in(0,2)$,
\begin{align}\label{eq:tauDef}
\tau=\frac{2\omega}{1+\omega}\quad\Longleftrightarrow\quad\omega=\frac{\tau}{2-\tau}.
\end{align}
For $\tau\in(0,2)$ we define
\begin{align}\label{eqn:NmaxFirstAppearance}
N_{\max}(\tau)=\mathcal N\left(\frac{\tau}{2-\tau}\right).
\end{align}

\begin{remark}\label{rem:ringTau}
Since $\frac{\omega_n}{1+\omega_n}=\sin\frac\pi n$, the ring of $n$ cylinders has separation
$\tau(\omega_n)=2\sin\frac\pi n$. This is the edge length of a regular $n$-gon inscribed in the unit
circle. Hence $N_{\max}(2\sin\frac\pi n)\ge n$.
\end{remark}

Let us recall the standard distance formula between lines (see, e.g., \cite[Section~2]{OSsix}).

\begin{lemma}\label{lem:linedist}
Let $\ell_i=\{\overrightarrow{p_i}+t \overrightarrow{v_i}:t\in\R\}$ for $i=1,2$. Then
\begin{align*}
\dist(\ell_1,\ell_2)=
\begin{cases}
\bigl| (\overrightarrow{p_2}-\overrightarrow{p_1})\cdot (\overrightarrow{v_1}\times \overrightarrow{v_2}) \bigr|\,/\,\|\overrightarrow{v_1}\times \overrightarrow{v_2}\|, & \text{if }\overrightarrow{v_1}\nparallel\overrightarrow{v_2},\\[3pt]
\bigl\|(\overrightarrow{p_2}-\overrightarrow{p_1})\times\overrightarrow{v_1}\bigr\|\,/\,\|\overrightarrow{v_1}\|, & \text{if }\overrightarrow{v_1}\parallel\overrightarrow{v_2}.
\end{cases}
\end{align*}
\end{lemma}

Let $\overrightarrow{e_1}$, $\overrightarrow{e_2}$, and $\overrightarrow{e_3}$ be the standard basis vectors of $\R^3$.
Let $\ell_0=\{\overrightarrow{e_3}+t\overrightarrow{e_1}:t\in\R\}$ be the line tangent to
$\Sph^2$ at $\overrightarrow{e_3}$ with direction $\overrightarrow{e_1}$. The rotation group
$\SO(3)$ acts transitively on lines tangent to $\Sph^2$, so every such line is $g\cdot\ell_0$ for
some $g\in\SO(3)$, unique up to the stabilizer of $\ell_0$.

\begin{definition}\label{def:D}
For $\tau\in(0,2)$ we define the \emph{safe zone} $\Dom(\tau)$ as the subset of $\SO(3)$ given by
\[ \Dom(\tau)=\{g\in\SO(3):\dist(\ell_0,\,g\cdot\ell_0)\ge\tau\}.\]
\end{definition}

When $\tau$ is fixed we write $\Dom$ for $\Dom(\tau)$.

\begin{definition}\label{def:nonOverlappingX}
A set $X\subseteq \SO(3)$ will be called \emph{non-overlapping} if
$g^{-1}g'\in\Dom$ whenever $g$ and $g'$ are two distinct elements of $X$.
\end{definition}

By \eqref{eqn:NmaxFirstAppearance} and Lemma~\ref{lem:dictionary}, $N_{\max}(\tau)$ is the
largest number of lines tangent to $\Sph^2$ with pairwise distances at least $\tau$. Since
rotations are isometries, $\{g_1,\dots,g_N\}\subset\SO(3)$ is non-overlapping exactly when the
lines $g_1\cdot\ell_0,\dots,g_N\cdot\ell_0$ are distinct with pairwise distances at least
$\tau$. Hence $N_{\max}(\tau)$ is also the largest cardinality of a non-overlapping subset of
$\SO(3)$.

\begin{lemma}\label{lem:monotone}
If $0<\tau'\le\tau<2$, then $\Dom(\tau)\subseteq\Dom(\tau')$, and consequently
$N_{\max}(\tau)\le N_{\max}(\tau')$.
\end{lemma}
\begin{proof}
The inclusion is immediate from Definition~\ref{def:D}. Hence a set that is non-overlapping for
$\tau$ is non-overlapping for $\tau'$, so $N_{\max}(\tau)\le N_{\max}(\tau')$.
\end{proof}

\begin{lemma}\label{lem:usc}
The function $g\mapsto\dist(\ell_0,g\cdot\ell_0)$ is upper semicontinuous on $\SO(3)$.
\end{lemma}
\begin{proof}
By definition
$\dist(\ell_0,g\cdot\ell_0)=\inf_{s,t\in\R}\bigl\|\overrightarrow{e_3}+t\overrightarrow{e_1}-g\cdot \overrightarrow{e_3}-s\,g\cdot \overrightarrow{e_1}\bigr\|$,
an infimum of a family of continuous functions of $g$, hence upper semicontinuous.
\end{proof}

\begin{corollary}\label{cor:attain}
For every integer $N\ge2$ the number
$\tau^*(N)=\sup\{\tau\in(0,2):N_{\max}(\tau)\ge N\}$
is attained, in the sense that $N_{\max}(\tau)\ge N$ holds if and only if $\tau\le\tau^*(N)$.
Consequently $R_N=\frac{\tau^*(N)}{2-\tau^*(N)}$ and $\mathcal N(\omega)\ge N$ if and only if
$\omega\le R_N$.
\end{corollary}
\begin{proof}
Let $\Phi(g_1,\dots,g_N)=\min_{i\ne j}\dist(\ell_0,g_i^{-1}g_j\cdot\ell_0)$. Each of the finitely
many entries is upper semicontinuous by Lemma~\ref{lem:usc}, and a minimum of finitely many
upper semicontinuous functions is upper semicontinuous. As $\SO(3)^N$ is compact, $\Phi$ attains
its maximum, say at $(\hat g_1,\dots,\hat g_N)$. Write $\hat\tau=\Phi(\hat g_1,\dots,\hat g_N)$.

If $\tau\le\hat\tau$ then $\{\hat g_1,\dots,\hat g_N\}$ is non-overlapping for $\tau$. Its
elements are distinct because $\tau>0$ and $\dist(\ell_0,\ell_0)=0$. So $N_{\max}(\tau)\ge N$.
If $\tau>\hat\tau$ and $X=\{g_1,\dots,g_N\}$ were non-overlapping for $\tau$, then
$\Phi(g_1,\dots,g_N)\ge\tau>\hat\tau$, contradicting maximality. Hence $\tau^*(N)=\hat\tau$ and the
stated equivalence holds. The last assertion follows by applying the increasing bijection
\eqref{eq:tauDef}.
\end{proof}

\section{Proof of the Ogievetsky--Shlosman conjecture}\label{sec:OSC}

Let $\ell$ be a line tangent to the unit sphere. Let $\overrightarrow{p}$ be the point of tangency,
let $\overrightarrow{v}$ be a unit vector parallel to $\ell$, and let
$\overrightarrow{w}=\overrightarrow{p}\times\overrightarrow{v}$. We call
$(\overrightarrow{p},\overrightarrow{v},\overrightarrow{w})$ an \emph{orthonormal system associated to}
$\ell$. The only other system associated to $\ell$ is
$(\overrightarrow{p},-\overrightarrow{v},-\overrightarrow{w})$. Passing from one to the other is called
\emph{reversing the direction} of $\ell$.

\begin{lemma}\label{lem:OSIdentity}
Let $\ell$ and $\ell'$ be lines tangent to the unit sphere. Let
$(\overrightarrow{p},\overrightarrow{v},\overrightarrow{w})$ and
$(\overrightarrow{p'},\overrightarrow{v'},\overrightarrow{w'})$ be their associated orthonormal systems.
Then
\begin{align}
2\left(\overrightarrow{p}\cdot\overrightarrow{p'}\right)^2-2\left(\overrightarrow{w}\cdot\overrightarrow{w'}\right)^2
&=\left(\overrightarrow{v}\cdot\overrightarrow{w'}-\overrightarrow{w}\cdot\overrightarrow{v'}\right)^2
+\left(\left(\overrightarrow{p'}-\overrightarrow{p}\right)\cdot\left(\overrightarrow{v}\times\overrightarrow{v'}\right)\right)^2
-2\left\|\overrightarrow{v}\times\overrightarrow{v'}\right\|^2. \label{eqn:OSIdentity}
\end{align}
\end{lemma}

\begin{proof}
Both sides of \eqref{eqn:OSIdentity} are invariant under rotations. Hence we may assume that
$\overrightarrow{p'}$, $\overrightarrow{v'}$, and $\overrightarrow{w'}$ are the coordinate vectors
$\overrightarrow{e_1}$, $\overrightarrow{e_2}$, and $\overrightarrow{e_3}$. Let
$\overrightarrow{p}=\langle p_1,p_2,p_3\rangle$, $\overrightarrow{v}=\langle v_1,v_2,v_3\rangle$, and
$\overrightarrow{w}=\langle w_1,w_2,w_3\rangle$. Then
$\overrightarrow{v}\times\overrightarrow{e_2}=\langle -v_3,0,v_1\rangle$. The second coordinate of
$\overrightarrow{p}\times\overrightarrow{v}$ is $w_2=p_3v_1-p_1v_3$. Therefore
$\left(\overrightarrow{e_1}-\overrightarrow{p}\right)\cdot\left(\overrightarrow{v}\times\overrightarrow{e_2}\right)=-(v_3+w_2)$.
The right-hand side of \eqref{eqn:OSIdentity} becomes
\[
(v_3-w_2)^2+(v_3+w_2)^2-2\left(v_1^2+v_3^2\right)=2w_2^2-2v_1^2.
\]
The coordinates of $\overrightarrow{e_1}$ in the orthonormal basis
$\{\overrightarrow{p},\overrightarrow{v},\overrightarrow{w}\}$ are $p_1$, $v_1$, and $w_1$. Hence
$p_1^2+v_1^2+w_1^2=1$. Since $\overrightarrow{w}$ is a unit vector, also $w_1^2+w_2^2+w_3^2=1$.
Subtracting the two equations gives $w_2^2-v_1^2=p_1^2-w_3^2$. This proves
\eqref{eqn:OSIdentity}, because $\overrightarrow{p}\cdot\overrightarrow{p'}=p_1$ and
$\overrightarrow{w}\cdot\overrightarrow{w'}=w_3$.
\end{proof}

\begin{lemma}\label{lem:OSLemma}
Let $\ell$ and $\ell'$ be lines tangent to the unit sphere. Let
$(\overrightarrow{p},\overrightarrow{v},\overrightarrow{w})$ and
$(\overrightarrow{p'},\overrightarrow{v'},\overrightarrow{w'})$ be their associated orthonormal systems.
If $\dist(\ell,\ell')\ge\sqrt2$, then $\overrightarrow{p}\cdot\overrightarrow{p'}\le0$ and
$\left|\overrightarrow{w}\cdot\overrightarrow{w'}\right|\le-\overrightarrow{p}\cdot\overrightarrow{p'}$.
The last inequality is an equality if and only if the following two conditions hold.
\begin{enumerate}[label=(\roman*)]
\item $\overrightarrow{v}\cdot\overrightarrow{w'}=\overrightarrow{w}\cdot\overrightarrow{v'}$.
\item $\ell\|\ell'$ or $\dist(\ell,\ell')=\sqrt2$.
\end{enumerate}
\end{lemma}

\begin{proof}
The points $\overrightarrow{p}$ and $\overrightarrow{p'}$ lie on $\ell$ and $\ell'$. Hence
$\left\|\overrightarrow{p}-\overrightarrow{p'}\right\|\ge\dist(\ell,\ell')\ge\sqrt2$. Since
$\left\|\overrightarrow{p}-\overrightarrow{p'}\right\|^2=2-2\,\overrightarrow{p}\cdot\overrightarrow{p'}$, this
means $\overrightarrow{p}\cdot\overrightarrow{p'}\le0$. If $\ell$ and $\ell'$ are not parallel, then
Lemma~\ref{lem:linedist} gives
$\left(\left(\overrightarrow{p'}-\overrightarrow{p}\right)\cdot\left(\overrightarrow{v}\times\overrightarrow{v'}\right)\right)^2
=\dist^2(\ell,\ell')\left\|\overrightarrow{v}\times\overrightarrow{v'}\right\|^2\ge2\left\|\overrightarrow{v}\times\overrightarrow{v'}\right\|^2$.
If they are parallel, then $\overrightarrow{v}\times\overrightarrow{v'}=0$. In both cases the last two
terms on the right-hand side of \eqref{eqn:OSIdentity} have a non-negative sum. The sum is zero if
and only if (ii) holds. The first term is zero if and only if (i) holds. Lemma~\ref{lem:OSIdentity} now
gives $\left(\overrightarrow{w}\cdot\overrightarrow{w'}\right)^2\le\left(\overrightarrow{p}\cdot\overrightarrow{p'}\right)^2$,
with equality if and only if (i) and (ii) hold. It remains to recall that
$\overrightarrow{p}\cdot\overrightarrow{p'}\le0$.
\end{proof}

For $\beta\in[0,\pi)$ let $\overrightarrow{u}=\langle\cos\beta,\sin\beta,0\rangle$ and
$\overrightarrow{u'}=\langle-\sin\beta,\cos\beta,0\rangle$. Let $\mathcal C_\beta$ be the configuration of
the four tangent lines
\begin{align}
\tfrac{1}{\sqrt2}\left(\pm\overrightarrow{e_1}+\overrightarrow{e_3}\right)+\R\,\overrightarrow{e_2}
\quad\text{and}\quad
\tfrac{1}{\sqrt2}\left(\pm\overrightarrow{u}-\overrightarrow{e_3}\right)+\R\,\overrightarrow{u'}.
\label{eqn:OSfamily}
\end{align}
It consists of two parallel pairs. The pairs lie on opposite sides of the plane $z=0$ and their
directions make the angle $\beta$.

\begin{lemma}\label{lem:Cbeta}
For $\beta\in(0,\pi)$ all six pairwise distances in $\mathcal C_\beta$ are equal to $\sqrt2$. The
configuration $\mathcal C_0$ is the ring of four lines.
\end{lemma}

\begin{proof}
The lines of the first pair are parallel. The difference of their base points is
$\sqrt2\,\overrightarrow{e_1}$, which is orthogonal to $\overrightarrow{e_2}$. Hence their distance is
$\sqrt2$. The same holds for the second pair. Now take one line from each pair. Then
$\overrightarrow{e_2}\times\overrightarrow{u'}=\sin\beta\,\overrightarrow{e_3}$. The difference of the
base points has $\overrightarrow{e_3}$-coordinate $-\sqrt2$. By Lemma~\ref{lem:linedist} the distance
is $\sqrt2$ when $\sin\beta\neq0$. For $\beta=0$ all four lines are parallel to
$\overrightarrow{e_2}$, and their points of tangency form a square.
\end{proof}

\begin{theorem}\label{thm:OSequality}
Let $\ell_1$, $\ell_2$, $\ell_3$, $\ell_4$ be lines tangent to the unit sphere. Their pairwise
distances are at least $\sqrt2$ if and only if, after a rotation, they form $\mathcal C_\beta$ for
some $\beta\in[0,\pi)$.
\end{theorem}

\begin{proof}
The configurations $\mathcal C_\beta$ have the required distances by Lemma~\ref{lem:Cbeta}.
Conversely, assume that the pairwise distances are at least $\sqrt2$.
For each $i$ let
$\left(\overrightarrow{p_i},\overrightarrow{v_i},\overrightarrow{w_i}\right)$ be an associated
orthonormal system. Let $P$ and $W$ be the Gram matrices with entries
$P_{ij}=\overrightarrow{p_i}\cdot\overrightarrow{p_j}$ and $W_{ij}=\overrightarrow{w_i}\cdot\overrightarrow{w_j}$.
Four vectors in $\R^3$ are linearly dependent. Hence there is a non-zero vector
$\overrightarrow{a}\in\R^4$ with $\sum_ia_i\overrightarrow{w_i}=0$. Reversing the direction of $\ell_i$
does not change the line, and it replaces $\overrightarrow{w_i}$ by $-\overrightarrow{w_i}$. We
reverse every line with $a_i<0$, and we replace each such $a_i$ by $-a_i$. The relation
$\sum_ia_i\overrightarrow{w_i}=0$ still holds, and now $a_i\ge0$ for all $i$. Lemma~\ref{lem:OSLemma}
gives
\begin{align}
0&=\left\|\sum_{i}a_i\overrightarrow{w_i}\right\|^2
=\sum_{i}a_i^2+\sum_{i\neq j}W_{ij}a_ia_j
\ge\sum_{i}a_i^2-\sum_{i\neq j}\left|W_{ij}\right|a_ia_j\nonumber\\
&\ge\sum_{i}a_i^2+\sum_{i\neq j}P_{ij}a_ia_j
=\left\|\sum_{i}a_i\overrightarrow{p_i}\right\|^2\ge0.
\label{eqn:OSchain}
\end{align}
Hence all inequalities in \eqref{eqn:OSchain} are equalities.

Let us prove the following claim:
\begin{align} & \text{The matrices }W\text{ and }P\text{ are equal, or become equal after reversing one more line.} \label{eqn:ClaimInOSProof} \end{align}
The equalities in \eqref{eqn:OSchain} give $\sum_ia_i\overrightarrow{p_i}=0$. They also give
$W_{ij}=-\left|W_{ij}\right|=P_{ij}$ whenever $a_i>0$ and $a_j>0$. Now let $a_k=0$. Then
$\sum_ia_iP_{ik}=\overrightarrow{p_k}\cdot\sum_ia_i\overrightarrow{p_i}=\overrightarrow{p_k}\cdot0=0$, and all
terms $a_iP_{ik}$ are non-positive. Hence $P_{ik}=0$ whenever $a_i>0$. Lemma~\ref{lem:OSLemma} gives
$W_{ik}=0=P_{ik}$. The diagonal entries of $W$ and $P$ are all equal to $1$. So far we know that
$W_{ij}=P_{ij}$ for every pair $i\neq j$ in which at least one of $a_i$ and $a_j$ is positive. Since
$\overrightarrow{a}\neq0$ and $\overrightarrow{w_i}\neq0$, at least two of the $a_i$ are positive. Hence
there is at most one pair $(i,j)$ with $a_i=a_j=0$. If there is none, the claim
\eqref{eqn:ClaimInOSProof} is proved. Otherwise we may assume that $a_1,a_2>0$ and $a_3=a_4=0$.
Then $a_1\overrightarrow{p_1}=-a_2\overrightarrow{p_2}$. Taking lengths gives $a_1=a_2$, and hence
$\overrightarrow{p_2}=-\overrightarrow{p_1}$. In the same way $\sum_ia_i\overrightarrow{w_i}=0$ gives
$\overrightarrow{w_2}=-\overrightarrow{w_1}$. Therefore
$\overrightarrow{v_2}=\overrightarrow{w_2}\times\overrightarrow{p_2}=\overrightarrow{w_1}\times\overrightarrow{p_1}=\overrightarrow{v_1}$.
Let $k\in\{3,4\}$. We showed above that $W_{1k}=P_{1k}=0$ and $W_{2k}=P_{2k}=0$. So the inequality of
Lemma~\ref{lem:OSLemma} is an equality for the pairs $(1,k)$ and $(2,k)$, and condition (i) gives
\begin{align*}
\overrightarrow{v_1}\cdot\overrightarrow{w_k}=\overrightarrow{w_1}\cdot\overrightarrow{v_k}
=-\overrightarrow{w_2}\cdot\overrightarrow{v_k}=-\overrightarrow{v_2}\cdot\overrightarrow{w_k}
=-\overrightarrow{v_1}\cdot\overrightarrow{w_k}.
\end{align*}
Hence $\overrightarrow{v_1}\cdot\overrightarrow{w_k}=\overrightarrow{w_1}\cdot\overrightarrow{v_k}=0$. The
vector $\overrightarrow{w_1}$ is orthogonal to $\overrightarrow{v_k}$ and, since $W_{1k}=0$, also to
$\overrightarrow{w_k}$. Therefore $\overrightarrow{p_k}=\pm\overrightarrow{w_1}$. The vector
$\overrightarrow{w_k}$ is orthogonal to $\overrightarrow{v_1}$ and $\overrightarrow{w_1}$, so
$\overrightarrow{w_k}=\pm\overrightarrow{p_1}$. Then
$\overrightarrow{v_k}=\overrightarrow{w_k}\times\overrightarrow{p_k}=\pm\overrightarrow{p_1}\times\overrightarrow{w_1}=\pm\overrightarrow{v_1}$.
If $\overrightarrow{p_3}=\overrightarrow{p_4}$, then $\ell_3$ and $\ell_4$ have the same point of tangency
and parallel directions, so they coincide. Hence $P_{34}=-1$ and $W_{34}=\pm1$. Reversing $\ell_4$ if
necessary gives $W_{34}=P_{34}$. This proves the claim \eqref{eqn:ClaimInOSProof}.

From now on $W=P$. The arrays of vectors $(\overrightarrow{p_i})_{i=1}^4$ and $(\overrightarrow{w_i})_{i=1}^4$ have the
same Gram matrix. Hence the map $\sum_ix_i\overrightarrow{p_i}\mapsto\sum_ix_i\overrightarrow{w_i}$ is
well defined and preserves lengths. It extends to an orthogonal matrix $Q$ with
$Q\overrightarrow{p_i}=\overrightarrow{w_i}$ for all $i$. If $\det Q=-1$, we reverse the directions of all four lines and
replace $Q$ by $-Q$. This does not change $W$. So $Q$ is a rotation by an angle $\varphi\in[0,\pi]$ about a unit vector, which
we may take to be $\overrightarrow{e_3}$. 
The matrix representation of $Q$ is \begin{align*} Q&=\left[\begin{array}{ccc}\cos\varphi& -\sin\varphi& 0\\ \sin\varphi&\cos\varphi&0\\0&0&1\end{array}\right].
\end{align*}
In particular $\sin\varphi\ge0$. For every unit vector $\overrightarrow{x}$ it is straightforward to verify that the quadratic form associated with $Q$ is
\begin{align*}
\overrightarrow{x}\cdot Q\overrightarrow{x}&=\cos\varphi+(1-\cos\varphi)\left(\overrightarrow{x}\cdot\overrightarrow{e_3}\right)^2.
\end{align*}

Since $\overrightarrow{p_i}\cdot Q\overrightarrow{p_i}=\overrightarrow{p_i}\cdot\overrightarrow{w_i}=0$,
we get $\varphi\neq0$, $\cos\varphi\le0$, and
\begin{align*}
\overrightarrow{p_i}=\left\langle r\cos\alpha_i,\ r\sin\alpha_i,\ \varepsilon_ih\right\rangle,
\quad h=\sqrt{\tfrac{-\cos\varphi}{1-\cos\varphi}},\quad r=\sqrt{\tfrac{1}{1-\cos\varphi}},\quad \varepsilon_i\in\{1,-1\}.
\end{align*}
Since $W=P$ and $P_{ij}\leq 0$, all pairs satisfy $\left|W_{ij}\right|=-P_{ij}$. By Lemma~\ref{lem:OSLemma},
$\overrightarrow{v_i}\cdot\overrightarrow{w_j}=\overrightarrow{w_i}\cdot\overrightarrow{v_j}$ for all
$i\neq j$. Since $\overrightarrow{v_i}=\overrightarrow{w_i}\times\overrightarrow{p_i}$ and
$\overrightarrow{v_j}=\overrightarrow{w_j}\times\overrightarrow{p_j}$, we have
\begin{align*}
\overrightarrow{v_i}\cdot\overrightarrow{w_j}-\overrightarrow{w_i}\cdot\overrightarrow{v_j}
=\det\left(\overrightarrow{w_i},\overrightarrow{p_i},\overrightarrow{w_j}\right)
-\det\left(\overrightarrow{w_i},\overrightarrow{w_j},\overrightarrow{p_j}\right)
=-\det\left(\overrightarrow{w_i},\overrightarrow{w_j},\overrightarrow{p_i}+\overrightarrow{p_j}\right).
\end{align*}
Hence the equation $\overrightarrow{v_i}\cdot\overrightarrow{w_j}=\overrightarrow{w_i}\cdot\overrightarrow{v_j}$
is equivalent to
$\det\left(\overrightarrow{w_i},\overrightarrow{w_j},\overrightarrow{p_i}+\overrightarrow{p_j}\right)=0$.
Let $\delta=\alpha_i-\alpha_j$. Since
$\overrightarrow{w_i}=Q\overrightarrow{p_i}=\left\langle r\cos(\alpha_i+\varphi),\ r\sin(\alpha_i+\varphi),\ \varepsilon_ih\right\rangle$,
expanding the determinant gives
\begin{align*}
\det\left(\overrightarrow{w_i},\overrightarrow{w_j},\overrightarrow{p_i}+\overrightarrow{p_j}\right)
&=
\begin{vmatrix}
r\cos(\alpha_i+\varphi) & r\cos(\alpha_j+\varphi) & r(\cos\alpha_i+\cos\alpha_j)\\
r\sin(\alpha_i+\varphi) & r\sin(\alpha_j+\varphi) & r(\sin\alpha_i+\sin\alpha_j)\\
\varepsilon_ih & \varepsilon_jh & (\varepsilon_i+\varepsilon_j)h
\end{vmatrix}\\
&=r^2h\cdot\begin{cases}
-2\varepsilon_i(1-\cos\varphi)\sin\delta, & \text{if }\varepsilon_i=\varepsilon_j,\\
-2\varepsilon_i\sin\varphi\,(1+\cos\delta), & \text{if }\varepsilon_i\neq\varepsilon_j.
\end{cases}
\end{align*}
Note that $r>0$ and $1-\cos\varphi>0$.
If $\cos\varphi=0$, then $h=0$. All $\overrightarrow{p_i}$ and $\overrightarrow{w_i}$ are orthogonal to
$\overrightarrow{e_3}$, so all lines are parallel to $\overrightarrow{e_3}$. The four points
$\overrightarrow{p_i}$ lie on a great circle, and their pairwise angles are at least $\frac\pi2$. Hence
consecutive points are at angle exactly $\frac\pi2$. This is the ring $\mathcal C_0$.

If $-1<\cos\varphi<0$, then $h>0$ and $\sin\varphi>0$. Take two lines at the same height. The
determinant vanishes, so $\sin\delta=0$. Also $\delta\not\equiv0\pmod{2\pi}$, since otherwise the two
lines coincide. So $\delta\equiv\pi\pmod{2\pi}$. Take two lines at different heights. The determinant
vanishes, so $\cos\delta=-1$. Hence every two of the four angles $\alpha_i$ differ by $\pi$ modulo
$2\pi$. This is impossible for three angles.

If $\cos\varphi=-1$, then $h=r=\frac1{\sqrt2}$ and $\sin\varphi=0$. Two lines at the same height again
satisfy $\delta\equiv\pi\pmod{2\pi}$. Hence each of the two circles contains exactly two of the points
$\overrightarrow{p_i}$, and they are antipodal on that circle. After a rotation about
$\overrightarrow{e_3}$, the points are $\frac1{\sqrt2}\left(\pm\overrightarrow{e_1}+\overrightarrow{e_3}\right)$
and $\frac1{\sqrt2}\left(\pm\overrightarrow{u}-\overrightarrow{e_3}\right)$. Here $Q$ is the rotation by
$\pi$ about $\overrightarrow{e_3}$. For $\overrightarrow{x}=\langle x_1,x_2,x_3\rangle$ we get
$Q\overrightarrow{x}\times\overrightarrow{x}=2x_3\left\langle-x_2,x_1,0\right\rangle$. So the directions
are $\pm\overrightarrow{e_2}$ and $\pm\overrightarrow{u'}$. This is $\mathcal C_\beta$.
\end{proof}

\begin{proof}[Proof of Theorem~\ref{thm:OS}]
Suppose that four non-overlapping cylinders of radius $\omega'>1+\sqrt2$ touch the unit ball. By
Lemma~\ref{lem:dictionary} there are four tangent lines with pairwise distances at least
$\tau(\omega')>\sqrt2$. By Theorem~\ref{thm:OSequality} they form a configuration $\mathcal C_\beta$.
By Lemma~\ref{lem:Cbeta} two of them are at distance exactly $\sqrt2$. This is a contradiction.
\end{proof}

\begin{remark}
The homothety of Lemma~\ref{lem:dictionary} maps configurations of cylinders bijectively to
configurations of tangent lines. Hence Theorem~\ref{thm:OSequality} proves the conjecture of
Ogievetsky and Shlosman \cite{OSsix,OSPlatonic} and \cite[Section~3.1]{OSMorse} in full. They conjectured that the four-cylinder configurations of radius
$1+\sqrt2$ form a circle modulo rotations, and that no configuration of larger radius exists. The
family $\mathcal C_\beta$ is the motion found in \cite[Section~6.3]{OSsix}. One parallel pair stays
fixed and the other pair rotates about the common perpendicular $\overrightarrow{e_3}$. As $\beta\to\pi$, the configuration
$\mathcal C_\beta$ tends to $\mathcal C_0$. Measure the angle from the direction of the first
pair to the direction of the second pair about the axis that joins the pairs. This angle is $\beta$,
and it does not change under rotations. Hence distinct $\beta\in[0,\pi)$ give configurations that are not rotations of
each other, and the family is a circle.
\end{remark}

\begin{remark}\label{rem:five}
The ring of five cylinders can be expanded. A numerical configuration with this property was found
by Firsching \cite{F}. The standalone verifier \cite{MRVer} confirms it in exact arithmetic for a
rational configuration of five tangent lines (Appendix~\ref{app:data}).
\end{remark}

\section{The pinwheel and its distance formulas}\label{sec:pinwheel}

\begin{definition}[Pinwheel]\label{def:pinwheel}
Let $N=2k\ge6$ be even and set $\gamma=\frac{2\pi}k=\frac{4\pi}N$. For a latitude
$\mu\in(-\frac\pi2,\frac\pi2)$ and longitude $\lambda$, let
\[
x(\mu,\lambda)=(\cos \mu\cos\lambda,\ \cos \mu\sin\lambda,\ \sin \mu)\in\Sph^2,
\]
and let $\overrightarrow{u}=\langle-\sin \mu\cos\lambda,\,-\sin \mu\sin\lambda,\,\cos \mu\rangle$ and
$\overrightarrow{t}=\langle-\sin\lambda,\,\cos\lambda,\,0\rangle$ be the unit tangent vectors at $x(\mu,\lambda)$ pointing north
and east. Given parameters $\varphi\in(0,\frac\pi2)$, $\delta\in(0,\frac\pi2)$ and
$\psi\in\R$, the pinwheel $P(N;\varphi,\delta,\psi)$ is the family of $N$ lines
tangent to $\Sph^2$ consisting of the \emph{up group}
\begin{align*}
\ell_j^{+}:\quad\text{tangent point }x(\varphi,\ j\gamma),\quad
\text{direction } \overrightarrow{v}=\cos\delta\,\overrightarrow{u}+\sin\delta\,\overrightarrow{t},\quad j=0,\dots,k-1,
\end{align*}
and the \emph{down group}
\begin{align*}
\ell_j^{-}:\quad\text{tangent point }x(-\varphi,\ j\gamma+\psi),\quad
\text{direction } \overrightarrow{v}=\cos\delta\,\overrightarrow{u}+\sin\delta\,\overrightarrow{t},\quad j=0,\dots,k-1.
\end{align*}
\end{definition}

Up to a rotation about the vertical axis and a reflection, $P(N;\varphi,\delta,\psi)$ is the
configuration $C_{N/2}(\varphi,\varkappa,\delta)$ of Ogievetsky and Shlosman \cite[Section~6]{OSsix}
with $\psi=\frac\gamma2+2\varkappa$. We index by the number of lines rather than by
the number of lines in one group.

\begin{figure}[t]
\centering
\colorlet{clrN}{green!40!black}
\colorlet{clrS}{blue!50}
\begin{tikzpicture}[scale=3.2, line cap=round, line join=round, >=stealth, dot/.style={circle, fill, inner sep=1.05pt}]
\draw[line width=0.5pt] (0,0) circle (1);
\draw[thin] plot[smooth] coordinates {(-1,0) (-0.9994,-0.0109) (-0.9975,-0.0218) (-0.9944,-0.0327) (-0.99,-0.0435) (-0.9845,-0.0543) (-0.9777,-0.065) (-0.9696,-0.0756) (-0.9604,-0.0861) (-0.95,-0.0965) (-0.9383,-0.1068) (-0.9256,-0.117) (-0.9116,-0.127) (-0.8965,-0.1369) (-0.8804,-0.1466) (-0.8631,-0.1561) (-0.8447,-0.1654) (-0.8253,-0.1745) (-0.8048,-0.1834) (-0.7834,-0.1921) (-0.761,-0.2005) (-0.7376,-0.2087) (-0.7133,-0.2166) (-0.6881,-0.2242) (-0.6621,-0.2316) (-0.6352,-0.2387) (-0.6076,-0.2454) (-0.5792,-0.2519) (-0.5501,-0.2581) (-0.5202,-0.2639) (-0.4898,-0.2694) (-0.4587,-0.2746) (-0.4271,-0.2794) (-0.3949,-0.2839) (-0.3622,-0.288) (-0.3291,-0.2918) (-0.2956,-0.2952) (-0.2617,-0.2983) (-0.2274,-0.3009) (-0.1929,-0.3032) (-0.1582,-0.3051) (-0.1232,-0.3067) (-0.0881,-0.3078) (-0.0529,-0.3086) (-0.0176,-0.309) (0.0176,-0.309) (0.0529,-0.3086) (0.0881,-0.3078) (0.1232,-0.3067) (0.1582,-0.3051) (0.1929,-0.3032) (0.2274,-0.3009) (0.2617,-0.2983) (0.2956,-0.2952) (0.3291,-0.2918) (0.3622,-0.288) (0.3949,-0.2839) (0.4271,-0.2794) (0.4587,-0.2746) (0.4898,-0.2694) (0.5202,-0.2639) (0.5501,-0.2581) (0.5792,-0.2519) (0.6076,-0.2454) (0.6352,-0.2387) (0.6621,-0.2316) (0.6881,-0.2242) (0.7133,-0.2166) (0.7376,-0.2087) (0.761,-0.2005) (0.7834,-0.1921) (0.8048,-0.1834) (0.8253,-0.1745) (0.8447,-0.1654) (0.8631,-0.1561) (0.8804,-0.1466) (0.8965,-0.1369) (0.9116,-0.127) (0.9256,-0.117) (0.9383,-0.1068) (0.95,-0.0965) (0.9604,-0.0861) (0.9696,-0.0756) (0.9777,-0.065) (0.9845,-0.0543) (0.99,-0.0435) (0.9944,-0.0327) (0.9975,-0.0218) (0.9994,-0.0109) (1,0)};
\draw[very thin, dashed] plot[smooth] coordinates {(1,0) (0.9994,0.0109) (0.9975,0.0218) (0.9944,0.0327) (0.99,0.0435) (0.9845,0.0543) (0.9777,0.065) (0.9696,0.0756) (0.9604,0.0861) (0.95,0.0965) (0.9383,0.1068) (0.9256,0.117) (0.9116,0.127) (0.8965,0.1369) (0.8804,0.1466) (0.8631,0.1561) (0.8447,0.1654) (0.8253,0.1745) (0.8048,0.1834) (0.7834,0.1921) (0.761,0.2005) (0.7376,0.2087) (0.7133,0.2166) (0.6881,0.2242) (0.6621,0.2316) (0.6352,0.2387) (0.6076,0.2454) (0.5792,0.2519) (0.5501,0.2581) (0.5202,0.2639) (0.4898,0.2694) (0.4587,0.2746) (0.4271,0.2794) (0.3949,0.2839) (0.3622,0.288) (0.3291,0.2918) (0.2956,0.2952) (0.2617,0.2983) (0.2274,0.3009) (0.1929,0.3032) (0.1582,0.3051) (0.1232,0.3067) (0.0881,0.3078) (0.0529,0.3086) (0.0176,0.309) (-0.0176,0.309) (-0.0529,0.3086) (-0.0881,0.3078) (-0.1232,0.3067) (-0.1582,0.3051) (-0.1929,0.3032) (-0.2274,0.3009) (-0.2617,0.2983) (-0.2956,0.2952) (-0.3291,0.2918) (-0.3622,0.288) (-0.3949,0.2839) (-0.4271,0.2794) (-0.4587,0.2746) (-0.4898,0.2694) (-0.5202,0.2639) (-0.5501,0.2581) (-0.5792,0.2519) (-0.6076,0.2454) (-0.6352,0.2387) (-0.6621,0.2316) (-0.6881,0.2242) (-0.7133,0.2166) (-0.7376,0.2087) (-0.761,0.2005) (-0.7834,0.1921) (-0.8048,0.1834) (-0.8253,0.1745) (-0.8447,0.1654) (-0.8631,0.1561) (-0.8804,0.1466) (-0.8965,0.1369) (-0.9116,0.127) (-0.9256,0.117) (-0.9383,0.1068) (-0.95,0.0965) (-0.9604,0.0861) (-0.9696,0.0756) (-0.9777,0.065) (-0.9845,0.0543) (-0.99,0.0435) (-0.9944,0.0327) (-0.9975,0.0218) (-0.9994,0.0109) (-1,0)};
\draw[line width=0.7pt, clrN] plot[smooth] coordinates {(-0.7622,0.6473) (-0.7698,0.6376) (-0.776,0.6278) (-0.781,0.6179) (-0.7846,0.608) (-0.787,0.598) (-0.788,0.588) (-0.7876,0.578) (-0.786,0.568) (-0.783,0.5581) (-0.7787,0.5482) (-0.7731,0.5383) (-0.7661,0.5286) (-0.7579,0.5189) (-0.7484,0.5093) (-0.7377,0.4999) (-0.7257,0.4906) (-0.7125,0.4815) (-0.698,0.4725) (-0.6824,0.4638) (-0.6657,0.4552) (-0.6478,0.4469) (-0.6288,0.4388) (-0.6088,0.4309) (-0.5878,0.4233) (-0.5657,0.416) (-0.5427,0.409) (-0.5188,0.4022) (-0.494,0.3958) (-0.4684,0.3897) (-0.442,0.3839) (-0.4148,0.3785) (-0.387,0.3734) (-0.3585,0.3687) (-0.3294,0.3643) (-0.2997,0.3603) (-0.2695,0.3567) (-0.2389,0.3535) (-0.2079,0.3506) (-0.1765,0.3482) (-0.1448,0.3462) (-0.1129,0.3445) (-0.0808,0.3433) (-0.0485,0.3425) (-0.0162,0.3421) (0.0162,0.3421) (0.0485,0.3425) (0.0808,0.3433) (0.1129,0.3445) (0.1448,0.3462) (0.1765,0.3482) (0.2079,0.3506) (0.2389,0.3535) (0.2695,0.3567) (0.2997,0.3603) (0.3294,0.3643) (0.3585,0.3687) (0.387,0.3734) (0.4148,0.3785) (0.442,0.3839) (0.4684,0.3897) (0.494,0.3958) (0.5188,0.4022) (0.5427,0.409) (0.5657,0.416) (0.5878,0.4233) (0.6088,0.4309) (0.6288,0.4388) (0.6478,0.4469) (0.6657,0.4552) (0.6824,0.4638) (0.698,0.4725) (0.7125,0.4815) (0.7257,0.4906) (0.7377,0.4999) (0.7484,0.5093) (0.7579,0.5189) (0.7661,0.5286) (0.7731,0.5383) (0.7787,0.5482) (0.783,0.5581) (0.786,0.568) (0.7876,0.578) (0.788,0.588) (0.787,0.598) (0.7846,0.608) (0.781,0.6179) (0.776,0.6278) (0.7698,0.6376) (0.7622,0.6473)};
\draw[very thin, dashed, clrN] plot[smooth] coordinates {(0.7622,0.6473) (0.756,0.6543) (0.7491,0.6611) (0.7415,0.6679) (0.7333,0.6747) (0.7245,0.6813) (0.715,0.6879) (0.7049,0.6944) (0.6942,0.7008) (0.6829,0.707) (0.671,0.7132) (0.6585,0.7193) (0.6454,0.7252) (0.6318,0.7311) (0.6176,0.7368) (0.6029,0.7423) (0.5876,0.7478) (0.5719,0.7531) (0.5556,0.7582) (0.5389,0.7632) (0.5217,0.768) (0.504,0.7727) (0.4859,0.7772) (0.4674,0.7816) (0.4484,0.7858) (0.4291,0.7898) (0.4094,0.7936) (0.3894,0.7972) (0.369,0.8007) (0.3482,0.804) (0.3272,0.8071) (0.3059,0.8099) (0.2843,0.8126) (0.2625,0.8151) (0.2404,0.8174) (0.2182,0.8195) (0.1957,0.8214) (0.1731,0.8231) (0.1503,0.8246) (0.1274,0.8258) (0.1044,0.8269) (0.0813,0.8277) (0.0581,0.8284) (0.0349,0.8288) (0.0116,0.829) (-0.0116,0.829) (-0.0349,0.8288) (-0.0581,0.8284) (-0.0813,0.8277) (-0.1044,0.8269) (-0.1274,0.8258) (-0.1503,0.8246) (-0.1731,0.8231) (-0.1957,0.8214) (-0.2182,0.8195) (-0.2404,0.8174) (-0.2625,0.8151) (-0.2843,0.8126) (-0.3059,0.8099) (-0.3272,0.8071) (-0.3482,0.804) (-0.369,0.8007) (-0.3894,0.7972) (-0.4094,0.7936) (-0.4291,0.7898) (-0.4484,0.7858) (-0.4674,0.7816) (-0.4859,0.7772) (-0.504,0.7727) (-0.5217,0.768) (-0.5389,0.7632) (-0.5556,0.7582) (-0.5719,0.7531) (-0.5876,0.7478) (-0.6029,0.7423) (-0.6176,0.7368) (-0.6318,0.7311) (-0.6454,0.7252) (-0.6585,0.7193) (-0.671,0.7132) (-0.6829,0.707) (-0.6942,0.7008) (-0.7049,0.6944) (-0.715,0.6879) (-0.7245,0.6813) (-0.7333,0.6747) (-0.7415,0.6679) (-0.7491,0.6611) (-0.756,0.6543) (-0.7622,0.6473)};
\draw[line width=0.7pt, clrS] plot[smooth] coordinates {(-0.7622,-0.6473) (-0.756,-0.6543) (-0.7491,-0.6611) (-0.7415,-0.6679) (-0.7333,-0.6747) (-0.7245,-0.6813) (-0.715,-0.6879) (-0.7049,-0.6944) (-0.6942,-0.7008) (-0.6829,-0.707) (-0.671,-0.7132) (-0.6585,-0.7193) (-0.6454,-0.7252) (-0.6318,-0.7311) (-0.6176,-0.7368) (-0.6029,-0.7423) (-0.5876,-0.7478) (-0.5719,-0.7531) (-0.5556,-0.7582) (-0.5389,-0.7632) (-0.5217,-0.768) (-0.504,-0.7727) (-0.4859,-0.7772) (-0.4674,-0.7816) (-0.4484,-0.7858) (-0.4291,-0.7898) (-0.4094,-0.7936) (-0.3894,-0.7972) (-0.369,-0.8007) (-0.3482,-0.804) (-0.3272,-0.8071) (-0.3059,-0.8099) (-0.2843,-0.8126) (-0.2625,-0.8151) (-0.2404,-0.8174) (-0.2182,-0.8195) (-0.1957,-0.8214) (-0.1731,-0.8231) (-0.1503,-0.8246) (-0.1274,-0.8258) (-0.1044,-0.8269) (-0.0813,-0.8277) (-0.0581,-0.8284) (-0.0349,-0.8288) (-0.0116,-0.829) (0.0116,-0.829) (0.0349,-0.8288) (0.0581,-0.8284) (0.0813,-0.8277) (0.1044,-0.8269) (0.1274,-0.8258) (0.1503,-0.8246) (0.1731,-0.8231) (0.1957,-0.8214) (0.2182,-0.8195) (0.2404,-0.8174) (0.2625,-0.8151) (0.2843,-0.8126) (0.3059,-0.8099) (0.3272,-0.8071) (0.3482,-0.804) (0.369,-0.8007) (0.3894,-0.7972) (0.4094,-0.7936) (0.4291,-0.7898) (0.4484,-0.7858) (0.4674,-0.7816) (0.4859,-0.7772) (0.504,-0.7727) (0.5217,-0.768) (0.5389,-0.7632) (0.5556,-0.7582) (0.5719,-0.7531) (0.5876,-0.7478) (0.6029,-0.7423) (0.6176,-0.7368) (0.6318,-0.7311) (0.6454,-0.7252) (0.6585,-0.7193) (0.671,-0.7132) (0.6829,-0.707) (0.6942,-0.7008) (0.7049,-0.6944) (0.715,-0.6879) (0.7245,-0.6813) (0.7333,-0.6747) (0.7415,-0.6679) (0.7491,-0.6611) (0.756,-0.6543) (0.7622,-0.6473)};
\draw[very thin, dashed, clrS] plot[smooth] coordinates {(0.7622,-0.6473) (0.7698,-0.6376) (0.776,-0.6278) (0.781,-0.6179) (0.7846,-0.608) (0.787,-0.598) (0.788,-0.588) (0.7876,-0.578) (0.786,-0.568) (0.783,-0.5581) (0.7787,-0.5482) (0.7731,-0.5383) (0.7661,-0.5286) (0.7579,-0.5189) (0.7484,-0.5093) (0.7377,-0.4999) (0.7257,-0.4906) (0.7125,-0.4815) (0.698,-0.4725) (0.6824,-0.4638) (0.6657,-0.4552) (0.6478,-0.4469) (0.6288,-0.4388) (0.6088,-0.4309) (0.5878,-0.4233) (0.5657,-0.416) (0.5427,-0.409) (0.5188,-0.4022) (0.494,-0.3958) (0.4684,-0.3897) (0.442,-0.3839) (0.4148,-0.3785) (0.387,-0.3734) (0.3585,-0.3687) (0.3294,-0.3643) (0.2997,-0.3603) (0.2695,-0.3567) (0.2389,-0.3535) (0.2079,-0.3506) (0.1765,-0.3482) (0.1448,-0.3462) (0.1129,-0.3445) (0.0808,-0.3433) (0.0485,-0.3425) (0.0162,-0.3421) (-0.0162,-0.3421) (-0.0485,-0.3425) (-0.0808,-0.3433) (-0.1129,-0.3445) (-0.1448,-0.3462) (-0.1765,-0.3482) (-0.2079,-0.3506) (-0.2389,-0.3535) (-0.2695,-0.3567) (-0.2997,-0.3603) (-0.3294,-0.3643) (-0.3585,-0.3687) (-0.387,-0.3734) (-0.4148,-0.3785) (-0.442,-0.3839) (-0.4684,-0.3897) (-0.494,-0.3958) (-0.5188,-0.4022) (-0.5427,-0.409) (-0.5657,-0.416) (-0.5878,-0.4233) (-0.6088,-0.4309) (-0.6288,-0.4388) (-0.6478,-0.4469) (-0.6657,-0.4552) (-0.6824,-0.4638) (-0.698,-0.4725) (-0.7125,-0.4815) (-0.7257,-0.4906) (-0.7377,-0.4999) (-0.7484,-0.5093) (-0.7579,-0.5189) (-0.7661,-0.5286) (-0.7731,-0.5383) (-0.7787,-0.5482) (-0.783,-0.5581) (-0.786,-0.568) (-0.7876,-0.578) (-0.788,-0.588) (-0.787,-0.598) (-0.7846,-0.608) (-0.781,-0.6179) (-0.776,-0.6278) (-0.7698,-0.6376) (-0.7622,-0.6473)};
\draw[->, thin] (0,0) -- (0,1.2364);
\node[anchor=base west, inner sep=1pt] at (-0.02,1.2864) {$z$};
\draw[line width=0.35pt, clrN] plot[smooth] coordinates {(-0.3794,-0.9252) (-0.4061,-0.9135) (-0.4322,-0.9004) (-0.4576,-0.886) (-0.4824,-0.8703) (-0.5064,-0.8533) (-0.5297,-0.8351) (-0.5522,-0.8156) (-0.574,-0.795) (-0.5948,-0.7731) (-0.6148,-0.7502) (-0.6339,-0.7261) (-0.6521,-0.7009) (-0.6693,-0.6748) (-0.6855,-0.6476) (-0.7007,-0.6195) (-0.7149,-0.5904) (-0.728,-0.5605) (-0.7401,-0.5298) (-0.751,-0.4983) (-0.7609,-0.4661) (-0.7696,-0.4331) (-0.7773,-0.3996) (-0.7837,-0.3654) (-0.789,-0.3308) (-0.7932,-0.2956) (-0.7962,-0.26) (-0.798,-0.224) (-0.7986,-0.1877) (-0.7981,-0.1511) (-0.7964,-0.1143) (-0.7935,-0.0773) (-0.7895,-0.0402) (-0.7843,-0.003) (-0.7779,0.0341) (-0.7704,0.0713) (-0.7618,0.1083) (-0.752,0.1451) (-0.7411,0.1817) (-0.7292,0.2181) (-0.7162,0.2542) (-0.7021,0.2898) (-0.687,0.3251) (-0.6708,0.3598) (-0.6537,0.3941) (-0.6357,0.4277) (-0.6166,0.4608) (-0.5967,0.4931) (-0.5759,0.5248) (-0.5543,0.5556) (-0.5318,0.5856) (-0.5086,0.6148) (-0.4846,0.6431) (-0.4599,0.6704) (-0.4345,0.6968) (-0.4085,0.7221) (-0.3819,0.7463) (-0.3547,0.7695) (-0.327,0.7915) (-0.2989,0.8124) (-0.2702,0.832) (-0.2412,0.8505) (-0.2119,0.8676) (-0.1822,0.8836) (-0.1522,0.8982) (-0.122,0.9115) (-0.0917,0.9234) (-0.0612,0.934) (-0.0306,0.9432) (0,0.9511)};
\draw[line width=1pt, clrN] (-0.8514,0.2619) -- (-0.4073,0.616);
\draw[->, line width=0.65pt, clrN] (-0.6293,0.439) -- (-0.5015,0.6636);
\node[dot, clrN] at (-0.6293,0.439) {};
\node[anchor=base west, inner sep=1pt, clrN] at (-0.3673,0.666) {$\ell_1^{+}$};
\draw[line width=0.35pt, clrN] plot[smooth] coordinates {(-0.0489,-0.9988) (-0.0549,-0.9977) (-0.0609,-0.9949) (-0.0667,-0.9904) (-0.0724,-0.9843) (-0.0781,-0.9765) (-0.0835,-0.9671) (-0.0889,-0.956) (-0.0941,-0.9434) (-0.0991,-0.9292) (-0.104,-0.9134) (-0.1087,-0.8961) (-0.1132,-0.8773) (-0.1175,-0.8571) (-0.1216,-0.8354) (-0.1256,-0.8123) (-0.1293,-0.7878) (-0.1328,-0.762) (-0.136,-0.735) (-0.1391,-0.7067) (-0.1419,-0.6772) (-0.1445,-0.6466) (-0.1468,-0.6149) (-0.1489,-0.5822) (-0.1507,-0.5485) (-0.1523,-0.5139) (-0.1537,-0.4784) (-0.1547,-0.4421) (-0.1555,-0.4051) (-0.1561,-0.3674) (-0.1564,-0.3291) (-0.1564,-0.2903) (-0.1562,-0.2509) (-0.1557,-0.2112) (-0.1549,-0.171) (-0.1539,-0.1306) (-0.1527,-0.09) (-0.1511,-0.0492) (-0.1493,-0.0084) (-0.1473,0.0325) (-0.145,0.0733) (-0.1425,0.114) (-0.1398,0.1545) (-0.1368,0.1947) (-0.1336,0.2347) (-0.1301,0.2742) (-0.1264,0.3133) (-0.1226,0.3518) (-0.1185,0.3897) (-0.1142,0.427) (-0.1097,0.4636) (-0.1051,0.4994) (-0.1003,0.5344) (-0.0953,0.5685) (-0.0901,0.6016) (-0.0848,0.6337) (-0.0794,0.6648) (-0.0738,0.6947) (-0.0681,0.7235) (-0.0622,0.7511) (-0.0563,0.7774) (-0.0503,0.8024) (-0.0442,0.8261) (-0.038,0.8483) (-0.0318,0.8692) (-0.0255,0.8886) (-0.0192,0.9065) (-0.0128,0.9229) (-0.0064,0.9378) (0,0.9511)};
\draw[line width=1pt, clrN] (-0.3089,0.1171) -- (0.0623,0.573);
\draw[->, line width=0.65pt, clrN] (-0.1233,0.345) -- (-0.0982,0.5887);
\node[dot, clrN] at (-0.1233,0.345) {};
\node[anchor=base west, inner sep=1pt, clrN] at (0.1023,0.613) {$\ell_2^{+}$};
\draw[line width=0.35pt, clrN] plot[smooth] coordinates {(0.2115,-0.9774) (0.2327,-0.972) (0.2535,-0.965) (0.2739,-0.9565) (0.2938,-0.9465) (0.3133,-0.9349) (0.3323,-0.9218) (0.3508,-0.9073) (0.3687,-0.8913) (0.386,-0.8739) (0.4026,-0.855) (0.4187,-0.8348) (0.434,-0.8133) (0.4487,-0.7904) (0.4627,-0.7663) (0.4759,-0.741) (0.4883,-0.7144) (0.5,-0.6868) (0.5108,-0.658) (0.5209,-0.6281) (0.5301,-0.5973) (0.5384,-0.5655) (0.5459,-0.5328) (0.5525,-0.4992) (0.5582,-0.4648) (0.5631,-0.4297) (0.567,-0.3939) (0.57,-0.3575) (0.5721,-0.3205) (0.5733,-0.283) (0.5736,-0.245) (0.5729,-0.2066) (0.5713,-0.1679) (0.5688,-0.129) (0.5654,-0.0898) (0.5611,-0.0505) (0.5559,-0.0111) (0.5498,0.0284) (0.5428,0.0677) (0.535,0.107) (0.5262,0.1461) (0.5167,0.185) (0.5063,0.2235) (0.4951,0.2617) (0.4831,0.2995) (0.4703,0.3368) (0.4568,0.3736) (0.4425,0.4097) (0.4276,0.4452) (0.4119,0.48) (0.3956,0.5141) (0.3787,0.5473) (0.3611,0.5796) (0.343,0.611) (0.3243,0.6414) (0.3051,0.6707) (0.2854,0.6991) (0.2652,0.7262) (0.2446,0.7522) (0.2237,0.7771) (0.2023,0.8006) (0.1807,0.8229) (0.1587,0.8439) (0.1365,0.8635) (0.1141,0.8817) (0.0915,0.8985) (0.0688,0.9139) (0.0459,0.9277) (0.023,0.9402) (0,0.9511)};
\draw[line width=1pt, clrN] (0.407,0.1293) -- (0.497,0.6428);
\draw[->, line width=0.65pt, clrN] (0.452,0.3861) -- (0.3602,0.6214);
\node[dot, clrN] at (0.452,0.3861) {};
\node[anchor=base west, inner sep=1pt, clrN] at (0.517,0.6928) {$\ell_3^{+}$};
\draw[line width=0.35pt, clrS] plot[smooth] coordinates {(-0.5886,-0.8085) (-0.6133,-0.7897) (-0.6373,-0.7699) (-0.6604,-0.7491) (-0.6827,-0.7275) (-0.7042,-0.7049) (-0.7248,-0.6814) (-0.7445,-0.657) (-0.7632,-0.6318) (-0.781,-0.6059) (-0.7978,-0.5791) (-0.8135,-0.5517) (-0.8283,-0.5235) (-0.842,-0.4947) (-0.8547,-0.4653) (-0.8663,-0.4352) (-0.8768,-0.4047) (-0.8862,-0.3736) (-0.8944,-0.342) (-0.9016,-0.31) (-0.9076,-0.2777) (-0.9125,-0.2449) (-0.9162,-0.2119) (-0.9188,-0.1786) (-0.9202,-0.1451) (-0.9205,-0.1114) (-0.9196,-0.0776) (-0.9175,-0.0436) (-0.9143,-0.0096) (-0.9099,0.0244) (-0.9044,0.0584) (-0.8978,0.0923) (-0.89,0.1261) (-0.8811,0.1597) (-0.8711,0.1931) (-0.86,0.2263) (-0.8478,0.2592) (-0.8346,0.2918) (-0.8203,0.324) (-0.8049,0.3558) (-0.7886,0.3871) (-0.7713,0.418) (-0.7529,0.4483) (-0.7337,0.4781) (-0.7135,0.5073) (-0.6924,0.5358) (-0.6705,0.5637) (-0.6477,0.5908) (-0.6241,0.6172) (-0.5997,0.6429) (-0.5745,0.6677) (-0.5487,0.6917) (-0.5221,0.7148) (-0.4949,0.737) (-0.4671,0.7583) (-0.4386,0.7786) (-0.4096,0.798) (-0.3801,0.8163) (-0.3502,0.8336) (-0.3198,0.8499) (-0.2889,0.8651) (-0.2578,0.8792) (-0.2263,0.8922) (-0.1945,0.9041) (-0.1624,0.9148) (-0.1302,0.9244) (-0.0978,0.9328) (-0.0653,0.9401) (-0.0326,0.9462) (0,0.9511)};
\draw[line width=1pt, clrS] (-0.6466,-0.8041) -- (-0.8041,-0.5573);
\draw[->, line width=0.65pt, clrS] (-0.7254,-0.6807) -- (-0.8727,-0.5051);
\node[dot, clrS] at (-0.7254,-0.6807) {};
\node[anchor=base west, inner sep=1pt, clrS] at (-0.7566,-0.8841) {$\ell_1^{-}$};
\draw[line width=0.35pt, clrS] plot[smooth] coordinates {(-0.1301,-0.9915) (-0.1449,-0.9887) (-0.1595,-0.9843) (-0.1739,-0.9783) (-0.1879,-0.9706) (-0.2017,-0.9614) (-0.2151,-0.9506) (-0.2282,-0.9382) (-0.2409,-0.9242) (-0.2532,-0.9087) (-0.265,-0.8918) (-0.2765,-0.8733) (-0.2875,-0.8534) (-0.298,-0.8322) (-0.308,-0.8095) (-0.3175,-0.7855) (-0.3265,-0.7602) (-0.3349,-0.7337) (-0.3428,-0.7059) (-0.3501,-0.677) (-0.3569,-0.6469) (-0.3631,-0.6158) (-0.3686,-0.5837) (-0.3736,-0.5506) (-0.3779,-0.5167) (-0.3816,-0.4818) (-0.3847,-0.4462) (-0.3872,-0.4098) (-0.389,-0.3728) (-0.3902,-0.3351) (-0.3907,-0.2969) (-0.3906,-0.2582) (-0.3898,-0.219) (-0.3884,-0.1796) (-0.3864,-0.1398) (-0.3837,-0.0997) (-0.3804,-0.0596) (-0.3765,-0.0193) (-0.3719,0.021) (-0.3667,0.0613) (-0.3609,0.1015) (-0.3546,0.1415) (-0.3476,0.1813) (-0.3401,0.2207) (-0.332,0.2599) (-0.3233,0.2985) (-0.3142,0.3367) (-0.3045,0.3744) (-0.2943,0.4114) (-0.2836,0.4477) (-0.2725,0.4833) (-0.2609,0.5181) (-0.2488,0.5521) (-0.2364,0.5851) (-0.2236,0.6172) (-0.2104,0.6483) (-0.1968,0.6783) (-0.183,0.7071) (-0.1688,0.7348) (-0.1544,0.7613) (-0.1397,0.7866) (-0.1247,0.8105) (-0.1096,0.8331) (-0.0943,0.8543) (-0.0788,0.8742) (-0.0632,0.8925) (-0.0475,0.9094) (-0.0317,0.9248) (-0.0159,0.9387) (0,0.9511)};
\draw[line width=1pt, clrS] (-0.3978,-0.9345) -- (-0.218,-0.6849);
\draw[->, line width=0.65pt, clrS] (-0.3079,-0.8097) -- (-0.3704,-0.6604);
\node[dot, clrS] at (-0.3079,-0.8097) {};
\node[anchor=base west, inner sep=1pt, clrS] at (-0.3978,-1.0645) {$\ell_2^{-}$};
\draw[line width=0.35pt, clrS] plot[smooth] coordinates {(0.1178,-0.993) (0.1315,-0.9906) (0.1449,-0.9865) (0.1581,-0.9808) (0.1711,-0.9734) (0.1837,-0.9645) (0.1961,-0.9539) (0.2081,-0.9418) (0.2198,-0.9281) (0.2311,-0.9129) (0.2421,-0.8961) (0.2526,-0.8779) (0.2628,-0.8582) (0.2725,-0.8372) (0.2817,-0.8147) (0.2905,-0.7909) (0.2988,-0.7657) (0.3066,-0.7393) (0.3139,-0.7117) (0.3206,-0.6829) (0.3269,-0.653) (0.3326,-0.622) (0.3377,-0.5899) (0.3423,-0.5569) (0.3464,-0.523) (0.3498,-0.4882) (0.3527,-0.4526) (0.355,-0.4162) (0.3567,-0.3792) (0.3578,-0.3415) (0.3583,-0.3033) (0.3583,-0.2645) (0.3576,-0.2254) (0.3563,-0.1858) (0.3545,-0.146) (0.3521,-0.1059) (0.3491,-0.0656) (0.3455,-0.0252) (0.3413,0.0152) (0.3366,0.0556) (0.3313,0.0959) (0.3255,0.1361) (0.3191,0.176) (0.3122,0.2156) (0.3048,0.2549) (0.2969,0.2937) (0.2885,0.3321) (0.2796,0.3699) (0.2702,0.4071) (0.2604,0.4437) (0.2502,0.4795) (0.2396,0.5145) (0.2285,0.5486) (0.2171,0.5819) (0.2053,0.6142) (0.1932,0.6454) (0.1808,0.6756) (0.1681,0.7047) (0.1551,0.7326) (0.1418,0.7593) (0.1283,0.7848) (0.1146,0.8089) (0.1007,0.8317) (0.0866,0.8532) (0.0724,0.8732) (0.0581,0.8918) (0.0436,0.9089) (0.0291,0.9245) (0.0146,0.9385) (0,0.9511)};
\draw[line width=1pt, clrS] (0.0743,-0.9748) -- (0.4904,-0.6509);
\draw[->, line width=0.65pt, clrS] (0.2824,-0.8129) -- (0.3398,-0.6642);
\node[dot, clrS] at (0.2824,-0.8129) {};
\node[anchor=base west, inner sep=1pt, clrS] at (0.0443,-1.1048) {$\ell_3^{-}$};
\draw[thin, clrN] plot[smooth] coordinates {(-0.5457,0.5859) (-0.5437,0.585) (-0.5417,0.5841) (-0.5398,0.5831) (-0.5379,0.5821) (-0.536,0.581) (-0.5342,0.5798) (-0.5324,0.5786) (-0.5306,0.5774) (-0.5289,0.5761) (-0.5272,0.5747) (-0.5256,0.5733) (-0.524,0.5718) (-0.5224,0.5703) (-0.5209,0.5687) (-0.5194,0.5671) (-0.518,0.5654) (-0.5166,0.5637) (-0.5152,0.5619) (-0.5139,0.5601) (-0.5127,0.5582) (-0.5115,0.5563) (-0.5103,0.5543) (-0.5092,0.5523) (-0.5081,0.5502) (-0.5071,0.5481) (-0.5061,0.546) (-0.5052,0.5438) (-0.5043,0.5416) (-0.5035,0.5393)};
\node[anchor=base west, inner sep=1pt, clrN] at (-0.6515,0.6736) {$\overrightarrow{u_1}^{\,+}$};
\draw[thin, clrS] plot[smooth] coordinates {(0.3199,-0.7157) (0.323,-0.7153) (0.3262,-0.715) (0.3293,-0.7147) (0.3324,-0.7145) (0.3355,-0.7143) (0.3385,-0.7141) (0.3415,-0.714) (0.3445,-0.7139) (0.3475,-0.7139) (0.3505,-0.7138) (0.3534,-0.7139) (0.3563,-0.714) (0.3592,-0.7141) (0.362,-0.7142) (0.3648,-0.7144) (0.3676,-0.7146) (0.3704,-0.7149) (0.3731,-0.7152) (0.3757,-0.7156) (0.3784,-0.7159) (0.381,-0.7164) (0.3835,-0.7168) (0.3861,-0.7173) (0.3885,-0.7179) (0.391,-0.7184) (0.3934,-0.719) (0.3957,-0.7197) (0.398,-0.7204) (0.4003,-0.7211)};
\node[anchor=base west, inner sep=1pt, clrS] at (0.1798,-0.6842) {$\overrightarrow{u_3}^{\,-}$};
\end{tikzpicture}
\caption{A few of the lines of a pinwheel configuration. Three tangent points lie
on the circle of latitude $\varphi$ (dark green) and three on the circle of
latitude $-\varphi$ (light blue). Each meridian is drawn in the colour of the
group of its tangent point. The two groups are staggered in longitude by the
free offset $\psi$. At each tangent point a short piece of the corresponding
tangent line is drawn. The arrow at each tangent point is the north-pointing
unit tangent vector, i.e.\ the direction of increasing latitude. Note that it
points towards the north pole on the southern circle as well. Every line is
obtained by rotating this vector eastward by the same angle $\delta$ within the
tangent plane. The small arcs at $\ell_1^{+}$ and $\ell_3^{-}$ mark this angle.}
\label{fig:pinwheel}
\end{figure}
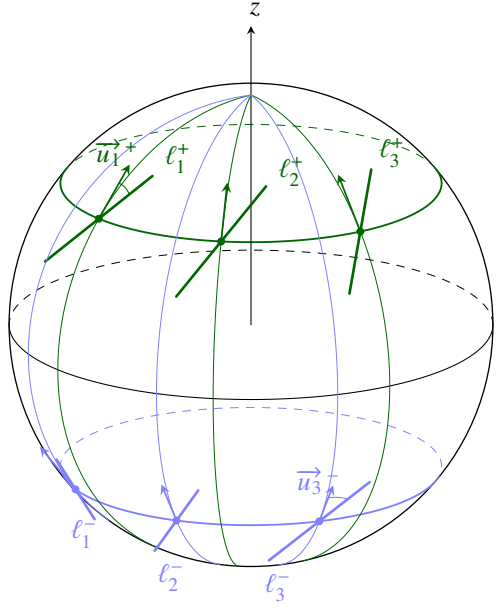

\subsection{The two distance formulas} 
By the rotational symmetry about the $z$-axis it suffices to compute the distance between a
tangent line at latitude $\varphi$ and longitude $0$ and a tangent line at latitude
$\pm\varphi$ and longitude $\Delta$.

Let $\varphi\in(0,\frac\pi2)$ and $\Delta\in[0,2\pi)$. Let $P_1$ be the point of the unit sphere
with latitude $\varphi$ and longitude $0$, and $P_2^{\pm}$ the point with latitude $\pm\varphi$
and longitude $\Delta$, so that
\begin{align*}
\overrightarrow{p_1}&=\langle \cos\varphi,0,\sin\varphi\rangle,\\
\overrightarrow{p_2}^{\pm}&=\langle \cos\Delta\cos\varphi,\sin\Delta\cos\varphi,\pm\sin\varphi\rangle.
\end{align*}

Let $\overrightarrow{u_1}$ and $\overrightarrow{u_2}^{+}$ be the unit tangent vectors at $P_1$ and $P_2^+$ pointing north, namely
\begin{align*}
\overrightarrow{u_1} &=\langle -\sin \varphi,0,\cos\varphi\rangle;\\
\overrightarrow{u_2}^{+}&=\langle -\cos\Delta\sin \varphi,-\sin\Delta\sin \varphi,\cos\varphi\rangle.
\end{align*}
Let $\overrightarrow{t_2}=\langle -\sin\Delta,\cos\Delta,0\rangle$ be the unit tangent vector at $P_2^+$ pointing east. Then $\{\overrightarrow{t_2},\overrightarrow{u_2}^+\}$ is an orthonormal basis of the tangent plane at $P_2^+$. For $\delta\in(0,\frac\pi2)$ let
\begin{align}
\overrightarrow{v_2}^+&=\cos\delta\overrightarrow{u_2}^++\sin\delta\overrightarrow{t_2}\nonumber \\&=
\langle 
-\sin\delta\sin\Delta
-\cos\delta\cos\Delta\sin\varphi,
\sin\delta\cos\Delta-\cos\delta\sin\Delta\sin\varphi,\cos\delta\cos\varphi
\rangle. \label{eqn:vectorV2P}
\end{align}

Setting $\Delta=0$ in \eqref{eqn:vectorV2P} gives the corresponding vector at $P_1$,
\begin{align}
\overrightarrow{v_1}&=\langle  -\sin\varphi\cos\delta,
\sin\delta,
\cos\varphi\cos\delta
\rangle,
\label{eqn:vectorV1}
\end{align}
and replacing $\varphi$ by $-\varphi$ in \eqref{eqn:vectorV2P} gives the corresponding vector at $P_2^-$,
\begin{align} 
\overrightarrow{v_2}^-&=\langle -\sin\delta\sin\Delta+\sin\varphi\cos\delta\cos\Delta,
\sin\delta\cos\Delta+\sin\varphi\cos\delta\sin\Delta,
\cos\varphi\cos\delta
\rangle.
\label{eqn:vectorV2M}
\end{align}

Let $\ell^+(0)$, $\ell^+(\Delta)$ and $\ell^-(\Delta)$ be the lines through $P_1$, $P_2^+$ and $P_2^-$ with directions $\overrightarrow{v_1}$, $\overrightarrow{v_2}^+$ and $\overrightarrow{v_2}^-$.
The argument records the longitude.
The lines of Definition~\ref{def:pinwheel} are $\ell_j^+=\ell^+(j\gamma)$ and $\ell_j^-=\ell^-(j\gamma+\psi)$.

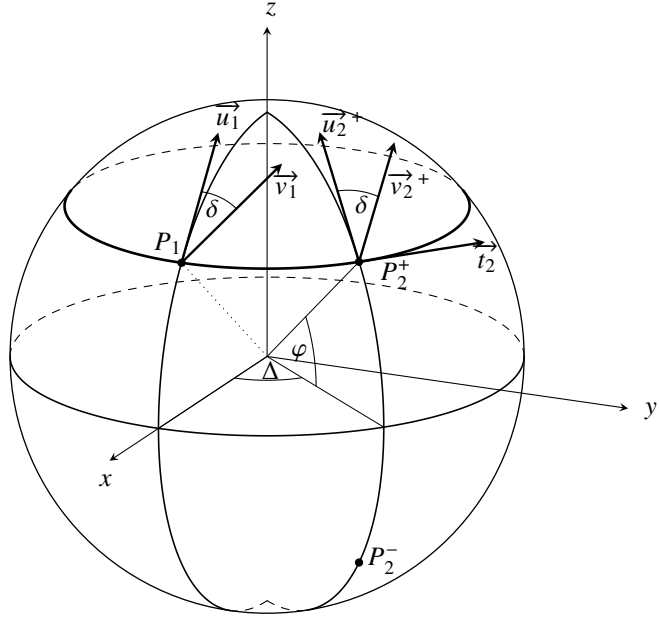
\begin{figure}[t]
\centering
\begin{tikzpicture}[scale=3.4, line cap=round, line join=round, >={stealth}, dot/.style={circle, fill, inner sep=1.1pt}]
\draw[line width=0.5pt] (0,0) circle (1);
\draw[semithick] plot[smooth] coordinates {(-1,0) (-0.9994,-0.0109) (-0.9975,-0.0218) (-0.9944,-0.0327) (-0.99,-0.0435) (-0.9845,-0.0543) (-0.9777,-0.065) (-0.9696,-0.0756) (-0.9604,-0.0861) (-0.95,-0.0965) (-0.9383,-0.1068) (-0.9256,-0.117) (-0.9116,-0.127) (-0.8965,-0.1369) (-0.8804,-0.1466) (-0.8631,-0.1561) (-0.8447,-0.1654) (-0.8253,-0.1745) (-0.8048,-0.1834) (-0.7834,-0.1921) (-0.761,-0.2005) (-0.7376,-0.2087) (-0.7133,-0.2166) (-0.6881,-0.2242) (-0.6621,-0.2316) (-0.6352,-0.2387) (-0.6076,-0.2454) (-0.5792,-0.2519) (-0.5501,-0.2581) (-0.5202,-0.2639) (-0.4898,-0.2694) (-0.4587,-0.2746) (-0.4271,-0.2794) (-0.3949,-0.2839) (-0.3622,-0.288) (-0.3291,-0.2918) (-0.2956,-0.2952) (-0.2617,-0.2983) (-0.2274,-0.3009) (-0.1929,-0.3032) (-0.1582,-0.3051) (-0.1232,-0.3067) (-0.0881,-0.3078) (-0.0529,-0.3086) (-0.0176,-0.309) (0.0176,-0.309) (0.0529,-0.3086) (0.0881,-0.3078) (0.1232,-0.3067) (0.1582,-0.3051) (0.1929,-0.3032) (0.2274,-0.3009) (0.2617,-0.2983) (0.2956,-0.2952) (0.3291,-0.2918) (0.3622,-0.288) (0.3949,-0.2839) (0.4271,-0.2794) (0.4587,-0.2746) (0.4898,-0.2694) (0.5202,-0.2639) (0.5501,-0.2581) (0.5792,-0.2519) (0.6076,-0.2454) (0.6352,-0.2387) (0.6621,-0.2316) (0.6881,-0.2242) (0.7133,-0.2166) (0.7376,-0.2087) (0.761,-0.2005) (0.7834,-0.1921) (0.8048,-0.1834) (0.8253,-0.1745) (0.8447,-0.1654) (0.8631,-0.1561) (0.8804,-0.1466) (0.8965,-0.1369) (0.9116,-0.127) (0.9256,-0.117) (0.9383,-0.1068) (0.95,-0.0965) (0.9604,-0.0861) (0.9696,-0.0756) (0.9777,-0.065) (0.9845,-0.0543) (0.99,-0.0435) (0.9944,-0.0327) (0.9975,-0.0218) (0.9994,-0.0109) (1,0)};
\draw[thin, dashed] plot[smooth] coordinates {(1,0) (0.9994,0.0109) (0.9975,0.0218) (0.9944,0.0327) (0.99,0.0435) (0.9845,0.0543) (0.9777,0.065) (0.9696,0.0756) (0.9604,0.0861) (0.95,0.0965) (0.9383,0.1068) (0.9256,0.117) (0.9116,0.127) (0.8965,0.1369) (0.8804,0.1466) (0.8631,0.1561) (0.8447,0.1654) (0.8253,0.1745) (0.8048,0.1834) (0.7834,0.1921) (0.761,0.2005) (0.7376,0.2087) (0.7133,0.2166) (0.6881,0.2242) (0.6621,0.2316) (0.6352,0.2387) (0.6076,0.2454) (0.5792,0.2519) (0.5501,0.2581) (0.5202,0.2639) (0.4898,0.2694) (0.4587,0.2746) (0.4271,0.2794) (0.3949,0.2839) (0.3622,0.288) (0.3291,0.2918) (0.2956,0.2952) (0.2617,0.2983) (0.2274,0.3009) (0.1929,0.3032) (0.1582,0.3051) (0.1232,0.3067) (0.0881,0.3078) (0.0529,0.3086) (0.0176,0.309) (-0.0176,0.309) (-0.0529,0.3086) (-0.0881,0.3078) (-0.1232,0.3067) (-0.1582,0.3051) (-0.1929,0.3032) (-0.2274,0.3009) (-0.2617,0.2983) (-0.2956,0.2952) (-0.3291,0.2918) (-0.3622,0.288) (-0.3949,0.2839) (-0.4271,0.2794) (-0.4587,0.2746) (-0.4898,0.2694) (-0.5202,0.2639) (-0.5501,0.2581) (-0.5792,0.2519) (-0.6076,0.2454) (-0.6352,0.2387) (-0.6621,0.2316) (-0.6881,0.2242) (-0.7133,0.2166) (-0.7376,0.2087) (-0.761,0.2005) (-0.7834,0.1921) (-0.8048,0.1834) (-0.8253,0.1745) (-0.8447,0.1654) (-0.8631,0.1561) (-0.8804,0.1466) (-0.8965,0.1369) (-0.9116,0.127) (-0.9256,0.117) (-0.9383,0.1068) (-0.95,0.0965) (-0.9604,0.0861) (-0.9696,0.0756) (-0.9777,0.065) (-0.9845,0.0543) (-0.99,0.0435) (-0.9944,0.0327) (-0.9975,0.0218) (-0.9994,0.0109) (-1,0)};
\draw[line width=1pt] plot[smooth] coordinates {(-0.7622,0.6473) (-0.7698,0.6376) (-0.776,0.6278) (-0.781,0.6179) (-0.7846,0.608) (-0.787,0.598) (-0.788,0.588) (-0.7876,0.578) (-0.786,0.568) (-0.783,0.5581) (-0.7787,0.5482) (-0.7731,0.5383) (-0.7661,0.5286) (-0.7579,0.5189) (-0.7484,0.5093) (-0.7377,0.4999) (-0.7257,0.4906) (-0.7125,0.4815) (-0.698,0.4725) (-0.6824,0.4638) (-0.6657,0.4552) (-0.6478,0.4469) (-0.6288,0.4388) (-0.6088,0.4309) (-0.5878,0.4233) (-0.5657,0.416) (-0.5427,0.409) (-0.5188,0.4022) (-0.494,0.3958) (-0.4684,0.3897) (-0.442,0.3839) (-0.4148,0.3785) (-0.387,0.3734) (-0.3585,0.3687) (-0.3294,0.3643) (-0.2997,0.3603) (-0.2695,0.3567) (-0.2389,0.3535) (-0.2079,0.3506) (-0.1765,0.3482) (-0.1448,0.3462) (-0.1129,0.3445) (-0.0808,0.3433) (-0.0485,0.3425) (-0.0162,0.3421) (0.0162,0.3421) (0.0485,0.3425) (0.0808,0.3433) (0.1129,0.3445) (0.1448,0.3462) (0.1765,0.3482) (0.2079,0.3506) (0.2389,0.3535) (0.2695,0.3567) (0.2997,0.3603) (0.3294,0.3643) (0.3585,0.3687) (0.387,0.3734) (0.4148,0.3785) (0.442,0.3839) (0.4684,0.3897) (0.494,0.3958) (0.5188,0.4022) (0.5427,0.409) (0.5657,0.416) (0.5878,0.4233) (0.6088,0.4309) (0.6288,0.4388) (0.6478,0.4469) (0.6657,0.4552) (0.6824,0.4638) (0.698,0.4725) (0.7125,0.4815) (0.7257,0.4906) (0.7377,0.4999) (0.7484,0.5093) (0.7579,0.5189) (0.7661,0.5286) (0.7731,0.5383) (0.7787,0.5482) (0.783,0.5581) (0.786,0.568) (0.7876,0.578) (0.788,0.588) (0.787,0.598) (0.7846,0.608) (0.781,0.6179) (0.776,0.6278) (0.7698,0.6376) (0.7622,0.6473)};
\draw[thin, dashed] plot[smooth] coordinates {(0.7622,0.6473) (0.756,0.6543) (0.7491,0.6611) (0.7415,0.6679) (0.7333,0.6747) (0.7245,0.6813) (0.715,0.6879) (0.7049,0.6944) (0.6942,0.7008) (0.6829,0.707) (0.671,0.7132) (0.6585,0.7193) (0.6454,0.7252) (0.6318,0.7311) (0.6176,0.7368) (0.6029,0.7423) (0.5876,0.7478) (0.5719,0.7531) (0.5556,0.7582) (0.5389,0.7632) (0.5217,0.768) (0.504,0.7727) (0.4859,0.7772) (0.4674,0.7816) (0.4484,0.7858) (0.4291,0.7898) (0.4094,0.7936) (0.3894,0.7972) (0.369,0.8007) (0.3482,0.804) (0.3272,0.8071) (0.3059,0.8099) (0.2843,0.8126) (0.2625,0.8151) (0.2404,0.8174) (0.2182,0.8195) (0.1957,0.8214) (0.1731,0.8231) (0.1503,0.8246) (0.1274,0.8258) (0.1044,0.8269) (0.0813,0.8277) (0.0581,0.8284) (0.0349,0.8288) (0.0116,0.829) (-0.0116,0.829) (-0.0349,0.8288) (-0.0581,0.8284) (-0.0813,0.8277) (-0.1044,0.8269) (-0.1274,0.8258) (-0.1503,0.8246) (-0.1731,0.8231) (-0.1957,0.8214) (-0.2182,0.8195) (-0.2404,0.8174) (-0.2625,0.8151) (-0.2843,0.8126) (-0.3059,0.8099) (-0.3272,0.8071) (-0.3482,0.804) (-0.369,0.8007) (-0.3894,0.7972) (-0.4094,0.7936) (-0.4291,0.7898) (-0.4484,0.7858) (-0.4674,0.7816) (-0.4859,0.7772) (-0.504,0.7727) (-0.5217,0.768) (-0.5389,0.7632) (-0.5556,0.7582) (-0.5719,0.7531) (-0.5876,0.7478) (-0.6029,0.7423) (-0.6176,0.7368) (-0.6318,0.7311) (-0.6454,0.7252) (-0.6585,0.7193) (-0.671,0.7132) (-0.6829,0.707) (-0.6942,0.7008) (-0.7049,0.6944) (-0.715,0.6879) (-0.7245,0.6813) (-0.7333,0.6747) (-0.7415,0.6679) (-0.7491,0.6611) (-0.756,0.6543) (-0.7622,0.6473)};
\draw[semithick] plot[smooth] coordinates {(-0.1426,-0.9898) (-0.1551,-0.9875) (-0.1673,-0.9842) (-0.1794,-0.98) (-0.1914,-0.9748) (-0.2031,-0.9686) (-0.2147,-0.9615) (-0.226,-0.9534) (-0.2371,-0.9444) (-0.248,-0.9344) (-0.2586,-0.9236) (-0.269,-0.9118) (-0.2791,-0.8991) (-0.2889,-0.8855) (-0.2985,-0.8711) (-0.3078,-0.8558) (-0.3167,-0.8396) (-0.3253,-0.8226) (-0.3337,-0.8048) (-0.3416,-0.7862) (-0.3493,-0.7669) (-0.3566,-0.7467) (-0.3635,-0.7259) (-0.3701,-0.7043) (-0.3764,-0.682) (-0.3822,-0.6591) (-0.3877,-0.6355) (-0.3928,-0.6112) (-0.3975,-0.5864) (-0.4018,-0.561) (-0.4057,-0.535) (-0.4092,-0.5085) (-0.4124,-0.4815) (-0.4151,-0.4541) (-0.4174,-0.4261) (-0.4192,-0.3978) (-0.4207,-0.3691) (-0.4218,-0.34) (-0.4224,-0.3105) (-0.4226,-0.2808) (-0.4224,-0.2508) (-0.4218,-0.2205) (-0.4208,-0.19) (-0.4193,-0.1593) (-0.4175,-0.1285) (-0.4152,-0.0975) (-0.4125,-0.0665) (-0.4094,-0.0354) (-0.4059,-0.0042) (-0.402,0.0269) (-0.3977,0.0581) (-0.393,0.0891) (-0.388,0.1201) (-0.3825,0.151) (-0.3767,0.1817) (-0.3704,0.2123) (-0.3639,0.2426) (-0.3569,0.2727) (-0.3497,0.3025) (-0.342,0.332) (-0.334,0.3612) (-0.3257,0.39) (-0.3171,0.4185) (-0.3082,0.4465) (-0.2989,0.4741) (-0.2894,0.5013) (-0.2796,0.5279) (-0.2695,0.554) (-0.2591,0.5796) (-0.2485,0.6046) (-0.2376,0.629) (-0.2265,0.6527) (-0.2152,0.6759) (-0.2037,0.6983) (-0.1919,0.7201) (-0.18,0.7412) (-0.1679,0.7615) (-0.1556,0.7811) (-0.1432,0.7999) (-0.1307,0.8179) (-0.118,0.8351) (-0.1052,0.8515) (-0.0922,0.867) (-0.0792,0.8817) (-0.0661,0.8955) (-0.053,0.9084) (-0.0398,0.9205) (-0.0265,0.9316) (-0.0133,0.9418) (0,0.9511)};
\draw[thin, dashed] plot[smooth] coordinates {(0,-0.9511) (-0.005,-0.9543) (-0.01,-0.9574) (-0.015,-0.9604) (-0.0201,-0.9633) (-0.0251,-0.966) (-0.0301,-0.9686) (-0.0351,-0.971) (-0.0401,-0.9733) (-0.0451,-0.9755) (-0.05,-0.9775) (-0.055,-0.9794) (-0.06,-0.9812) (-0.065,-0.9828) (-0.0699,-0.9843) (-0.0749,-0.9856) (-0.0798,-0.9868) (-0.0847,-0.9879) (-0.0896,-0.9888) (-0.0945,-0.9896) (-0.0994,-0.9902) (-0.1043,-0.9908) (-0.1091,-0.9911) (-0.114,-0.9913) (-0.1188,-0.9914) (-0.1236,-0.9914) (-0.1284,-0.9912) (-0.1331,-0.9909) (-0.1379,-0.9904) (-0.1426,-0.9898)};
\draw[semithick] plot[smooth] coordinates {(0.1555,-0.9878) (0.1688,-0.9852) (0.182,-0.9817) (0.1949,-0.9771) (0.2077,-0.9717) (0.2203,-0.9652) (0.2326,-0.9578) (0.2447,-0.9495) (0.2566,-0.9402) (0.2682,-0.93) (0.2796,-0.9189) (0.2906,-0.9069) (0.3014,-0.894) (0.3119,-0.8802) (0.3221,-0.8655) (0.332,-0.85) (0.3416,-0.8337) (0.3508,-0.8165) (0.3596,-0.7986) (0.3681,-0.7798) (0.3763,-0.7603) (0.3841,-0.74) (0.3915,-0.719) (0.3985,-0.6973) (0.4051,-0.6749) (0.4114,-0.6519) (0.4172,-0.6282) (0.4226,-0.6039) (0.4276,-0.579) (0.4322,-0.5535) (0.4363,-0.5275) (0.44,-0.5009) (0.4433,-0.4739) (0.4462,-0.4464) (0.4486,-0.4184) (0.4506,-0.3901) (0.4521,-0.3614) (0.4532,-0.3323) (0.4538,-0.3029) (0.454,-0.2731) (0.4537,-0.2432) (0.453,-0.2129) (0.4519,-0.1825) (0.4503,-0.1519) (0.4483,-0.1211) (0.4458,-0.0902) (0.4429,-0.0593) (0.4395,-0.0283) (0.4357,0.0028) (0.4315,0.0339) (0.4269,0.0649) (0.4218,0.0958) (0.4163,0.1267) (0.4105,0.1574) (0.4042,0.188) (0.3975,0.2184) (0.3904,0.2486) (0.383,0.2785) (0.3751,0.3082) (0.3669,0.3376) (0.3584,0.3666) (0.3494,0.3952) (0.3402,0.4235) (0.3306,0.4514) (0.3206,0.4788) (0.3104,0.5058) (0.2999,0.5322) (0.289,0.5581) (0.2779,0.5835) (0.2665,0.6083) (0.2549,0.6325) (0.2429,0.6561) (0.2308,0.679) (0.2184,0.7013) (0.2058,0.7229) (0.193,0.7437) (0.1801,0.7639) (0.1669,0.7832) (0.1536,0.8019) (0.1401,0.8197) (0.1265,0.8367) (0.1127,0.8529) (0.0989,0.8682) (0.0849,0.8827) (0.0709,0.8964) (0.0568,0.9091) (0.0427,0.921) (0.0285,0.9319) (0.0142,0.942) (0,0.9511)};
\draw[thin, dashed] plot[smooth] coordinates {(0,-0.9511) (0.0055,-0.9543) (0.0109,-0.9574) (0.0164,-0.9604) (0.0219,-0.9632) (0.0274,-0.9659) (0.0328,-0.9685) (0.0383,-0.9709) (0.0437,-0.9732) (0.0492,-0.9753) (0.0546,-0.9773) (0.06,-0.9791) (0.0655,-0.9808) (0.0709,-0.9824) (0.0763,-0.9838) (0.0817,-0.9851) (0.087,-0.9862) (0.0924,-0.9872) (0.0978,-0.988) (0.1031,-0.9887) (0.1084,-0.9893) (0.1137,-0.9897) (0.119,-0.99) (0.1243,-0.9901) (0.1296,-0.9901) (0.1348,-0.9899) (0.14,-0.9896) (0.1452,-0.9892) (0.1504,-0.9886) (0.1555,-0.9878)};
\draw[->, thin] (0,0) -- (-0.6128,-0.4061);
\node[anchor=base west, inner sep=1pt] at (-0.6628,-0.5061) {$x$};
\draw[->, thin] (0,0) -- (1.4048,-0.2024);
\node[anchor=base west, inner sep=1pt] at (1.4548,-0.2224) {$y$};
\draw[->, thin] (0,0) -- (0,1.2839);
\node[anchor=base west, inner sep=1pt] at (-0.02,1.3339) {$z$};
\draw[very thin] (0,0) -- (-0.4226,-0.2801);
\draw[very thin] (0,0) -- (0.454,-0.2753);
\draw[very thin] (0,0) -- (0.3577,0.3686);
\draw[very thin, dotted] (0,0) -- (-0.333,0.3648);
\draw[thin] plot[smooth] coordinates {(-0.1268,-0.084) (-0.1204,-0.0849) (-0.114,-0.0858) (-0.1075,-0.0865) (-0.101,-0.0873) (-0.0944,-0.088) (-0.0877,-0.0887) (-0.081,-0.0893) (-0.0743,-0.0898) (-0.0675,-0.0903) (-0.0607,-0.0908) (-0.0538,-0.0912) (-0.0469,-0.0916) (-0.04,-0.0919) (-0.0331,-0.0921) (-0.0261,-0.0924) (-0.0192,-0.0925) (-0.0122,-0.0926) (-0.0052,-0.0927) (0.0017,-0.0927) (0.0087,-0.0927) (0.0157,-0.0926) (0.0227,-0.0924) (0.0296,-0.0923) (0.0366,-0.092) (0.0435,-0.0917) (0.0504,-0.0914) (0.0572,-0.091) (0.0641,-0.0906) (0.0709,-0.0901) (0.0776,-0.0895) (0.0844,-0.089) (0.091,-0.0883) (0.0977,-0.0877) (0.1042,-0.0869) (0.1108,-0.0862) (0.1172,-0.0853) (0.1236,-0.0845) (0.1299,-0.0836) (0.1362,-0.0826)};
\node[anchor=base west, inner sep=1pt] at (-0.0323,-0.0859) {$\Delta$};
\draw[thin] plot[smooth] coordinates {(0.1907,-0.1156) (0.1906,-0.1088) (0.1906,-0.102) (0.1904,-0.0951) (0.1902,-0.0882) (0.19,-0.0813) (0.1897,-0.0744) (0.1893,-0.0674) (0.1889,-0.0604) (0.1884,-0.0534) (0.1879,-0.0464) (0.1873,-0.0393) (0.1867,-0.0323) (0.186,-0.0252) (0.1853,-0.0182) (0.1845,-0.0111) (0.1837,-0.004) (0.1828,0.003) (0.1818,0.0101) (0.1808,0.0172) (0.1798,0.0242) (0.1786,0.0313) (0.1775,0.0383) (0.1763,0.0454) (0.175,0.0524) (0.1737,0.0594) (0.1723,0.0664) (0.1709,0.0734) (0.1695,0.0803) (0.168,0.0872) (0.1664,0.0941) (0.1648,0.101) (0.1631,0.1079) (0.1614,0.1147) (0.1597,0.1215) (0.1579,0.1282) (0.156,0.1349) (0.1542,0.1416) (0.1522,0.1482) (0.1503,0.1548)};
\node[anchor=base west, inner sep=1pt] at (0.0803,0.0007) {$\varphi$};
\draw[->, line width=0.9pt] (-0.333,0.3648) -- (-0.1899,0.8719);
\node[anchor=base west, inner sep=1pt] at (-0.2099,0.9019) {$\overrightarrow{u_1}$};
\draw[->, line width=0.9pt] (-0.333,0.3648) -- (0.0585,0.7509);
\node[anchor=base west, inner sep=1pt] at (0.0185,0.6309) {$\overrightarrow{v_1}$};
\draw[thin] plot[smooth] coordinates {(-0.255,0.6414) (-0.2496,0.6406) (-0.2442,0.6396) (-0.2389,0.6386) (-0.2336,0.6374) (-0.2284,0.6361) (-0.2232,0.6348) (-0.218,0.6333) (-0.2129,0.6317) (-0.2079,0.63) (-0.2029,0.6282) (-0.1979,0.6263) (-0.193,0.6243) (-0.1881,0.6222) (-0.1833,0.62) (-0.1786,0.6177) (-0.1739,0.6153) (-0.1693,0.6128) (-0.1648,0.6102) (-0.1603,0.6075) (-0.1559,0.6047) (-0.1515,0.6018) (-0.1472,0.5988) (-0.143,0.5958) (-0.1389,0.5926) (-0.1349,0.5893) (-0.1309,0.586) (-0.127,0.5826) (-0.1232,0.579) (-0.1195,0.5754)};
\node[anchor=base west, inner sep=1pt] at (-0.2503,0.5436) {$\delta$};
\draw[->, line width=0.9pt] (0.3577,0.3686) -- (0.204,0.874);
\node[anchor=base west, inner sep=1pt] at (0.204,0.874) {$\overrightarrow{u_2}^{\,+}$};
\draw[->, line width=0.9pt] (0.3577,0.3686) -- (0.8478,0.4457);
\node[anchor=base west, inner sep=1pt] at (0.7978,0.3457) {$\overrightarrow{t_2}$};
\draw[->, line width=0.9pt] (0.3577,0.3686) -- (0.4957,0.8345);
\node[anchor=base west, inner sep=1pt] at (0.4657,0.6345) {$\overrightarrow{v_2}^{\,+}$};
\draw[thin] plot[smooth] coordinates {(0.2739,0.6442) (0.2792,0.645) (0.2846,0.6457) (0.29,0.6463) (0.2954,0.6467) (0.3008,0.6471) (0.3063,0.6473) (0.3117,0.6474) (0.3172,0.6474) (0.3228,0.6473) (0.3283,0.6471) (0.3338,0.6468) (0.3394,0.6464) (0.3449,0.6459) (0.3505,0.6452) (0.3561,0.6445) (0.3616,0.6436) (0.3672,0.6426) (0.3727,0.6415) (0.3783,0.6404) (0.3838,0.6391) (0.3894,0.6377) (0.3949,0.6362) (0.4004,0.6345) (0.4059,0.6328) (0.4114,0.631) (0.4168,0.6291) (0.4222,0.6271) (0.4276,0.6249) (0.433,0.6227)};
\node[anchor=base west, inner sep=1pt] at (0.3303,0.5592) {$\delta$};
\node[dot] at (-0.333,0.3648) {};
\node[anchor=base west, inner sep=1pt] at (-0.453,0.4148) {$P_1$};
\node[dot] at (0.3577,0.3686) {};
\node[anchor=base west, inner sep=1pt] at (0.4277,0.2986) {$P_2^{+}$};
\node[dot] at (0.3577,-0.8025) {};
\node[anchor=base west, inner sep=1pt] at (0.3777,-0.8125) {$P_2^{-}$};
\end{tikzpicture}
\caption{The points $P_1$, $P_2^{\pm}$ and the tangent frames from
Lemma~\ref{lem:distances}. The thick circle is the latitude circle at latitude
$\varphi$ containing $P_1$ and $P_2^{+}$. The meridians through $P_1$ and through
$P_2^{\pm}$ and the equator are drawn with hidden parts dashed. The vector $\overrightarrow{u_1}$ lies
along the meridian through $P_1$, as $\overrightarrow{u_2}^{\,+}$ does along the meridian through
$P_2^{+}$. The vectors
$\overrightarrow{u_1}$ and $\overrightarrow{u_2}^{\,+}$ point in the direction of increasing latitude,
$\overrightarrow{t_2}$ points eastward, and $\overrightarrow{v_1}$, $\overrightarrow{v_2}^{\,+}$ are obtained by
rotating $\overrightarrow{u_1}$, $\overrightarrow{u_2}^{\,+}$ by the angle $\delta$ within the respective
tangent planes, as in \eqref{eqn:vectorV2P}. The vector $\overrightarrow{v_2}^{\,-}$ at
$P_2^{-}$ (not drawn) is obtained from \eqref{eqn:vectorV2M}.}
\label{fig:lemmaDistances}
\end{figure}

 \begin{lemma}
 \label{lem:distances} Assume that $\varphi\in (0,\frac{\pi}{2})$, $\Delta\in [0,2\pi)$, 
 and $\delta\in (0,\frac{\pi}{2})$. The squares of the distances between $\ell^+(0)$ and $\ell^{\pm}(\Delta)$ satisfy
 \begin{align}
 \dist^2(\ell^+(0),\ell^+(\Delta))&= 
 \frac{4\cos^4\varphi\tan^2\delta \sin^2\frac{\Delta}2 }{
\left(\tan^2\delta+ \sin^2\varphi\right)
\cdot 
\left[ 1+\tan^2\delta -\sin^2\frac{\Delta}{2} \left(\tan^2\delta+ \sin^2\varphi\right)\right]
};
\label{eq:lemmaDistl1l2Plus}\\
 \dist^2(\ell^+(0),\ell^-(\Delta))&= 
 \frac{4\left(\sin\varphi\tan\delta +\tan\frac{\Delta}{2}\right)^2}{
 \cos^2\varphi(1+\tan^2\delta) + \left(    \sin\varphi\tan\delta+ \tan\frac{\Delta}{2}
\right)^2}\quad\text{for }\Delta\neq\pi,
 \label{eq:lemmaDistl1l2Minus}
 \end{align}
 while $\dist(\ell^+(0),\ell^-(\Delta))=2$ for $\Delta=\pi$.
 \end{lemma}
\begin{proof}
If $\overrightarrow{v_1}$ and $\overrightarrow{v_2}^{\pm}$ are not parallel, then
\begin{align}
\dist(\ell^+(0),\ell^{\pm}(\Delta))&=\frac{\left|(\overrightarrow{p_2}^{\pm}-\overrightarrow{p_1})\cdot \left(\overrightarrow{v_1}\times\overrightarrow{v_2}^{\pm}\right)\right|}
{\left\|\overrightarrow{v_1}\times\overrightarrow{v_2}^{\pm}\right\|}. \label{eqn:distanceFormula}
\end{align}
Write $\overrightarrow{v_1}\times\overrightarrow{v_2}^{\pm}=\langle N_x^{\pm},N_y^{\pm},N_z^{\pm}\rangle$.
From \eqref{eqn:vectorV2P}--\eqref{eqn:vectorV2M},
\begin{align*}
N_x^{\pm}&=\left(\sin\delta(1-\cos\Delta)\pm\cos\delta\sin\varphi\sin\Delta\right)\cos\delta\cos\varphi,\\
N_y^{\pm}&=\left(\cos\delta\sin\varphi(1\mp \cos\Delta)-\sin\delta\sin\Delta \right)\cos\delta\cos\varphi,\\
N_z^{\pm}&=(-1\pm1)\sin\delta\cos\delta\sin\varphi\cos\Delta+\sin\Delta\left(\sin^2\delta\pm\cos^2\delta\sin^2\varphi\right).
\end{align*}

\emph{The line $\ell^+(\Delta)$.} Here $\overrightarrow{p_2}^+-\overrightarrow{p_1}=\cos\varphi\langle\cos\Delta-1,\sin\Delta,0\rangle$, and expanding,
\begin{align*}
\left(\overrightarrow{p_2}^+-\overrightarrow{p_1}\right)\cdot\left(\overrightarrow{v_1}\times\overrightarrow{v_2}^+\right)&=
-2\cos^2\varphi\sin\delta\cos\delta(1-\cos\Delta),\\
\left\|\overrightarrow{v_1}\times\overrightarrow{v_2}^+\right\|^2&=
\left(\sin^2\delta+\cos^2\delta\sin^2\varphi\right)\\
&\quad\cdot\left(2(1-\cos\Delta)\cos^2\delta\cos^2\varphi+\sin^2\Delta\left(\sin^2\delta+\cos^2\delta\sin^2\varphi\right)\right).
\end{align*}
For $\Delta\in(0,2\pi)$ the norm is positive, since its second factor is at least
$2(1-\cos\Delta)\cos^2\delta\cos^2\varphi>0$, so \eqref{eqn:distanceFormula} applies. For
$\Delta=0$ the lines coincide and both sides of \eqref{eq:lemmaDistl1l2Plus} vanish. Substituting,
dividing numerator and denominator by $\cos^2\delta$, and using
$1-\cos^2\Delta=(1-\cos\Delta)(1+\cos\Delta)$, $1-\cos\Delta=2\sin^2\frac\Delta2$ and
$\cos^2\varphi+\tan^2\delta+\sin^2\varphi=1+\tan^2\delta$ gives \eqref{eq:lemmaDistl1l2Plus}. The
bracket in its denominator is at least
$1+\tan^2\delta-(\tan^2\delta+\sin^2\varphi)=\cos^2\varphi>0$.

\emph{The line $\ell^-(\Delta)$.} Here
$\overrightarrow{p_2}^--\overrightarrow{p_1}=\langle\cos\varphi(\cos\Delta-1),\cos\varphi\sin\Delta,-2\sin\varphi\rangle$.
With the abbreviations
\begin{align}
A&=\sin\varphi\cos\delta\cos\tfrac{\Delta}{2}-\sin\delta\sin\tfrac{\Delta}{2},\nonumber\\
B&=\sin\delta\sin\varphi\cos\tfrac{\Delta}{2}+\cos\delta\sin\tfrac{\Delta}{2},\label{eqn:defA}\\
C&=\sin\varphi\cos\delta\sin\tfrac{\Delta}{2}+\sin\delta\cos\tfrac{\Delta}{2},\nonumber
\end{align}
a direct expansion, using the half-angle formulas, gives
\begin{align}\label{eqn:NormMinusFinal}
\left(\overrightarrow{p_2}^--\overrightarrow{p_1}\right)\cdot\left(\overrightarrow{v_1}\times\overrightarrow{v_2}^-\right)=4AB,\quad
N_z^-=-2AC,\quad (N_x^-)^2+(N_y^-)^2=4\cos^2\delta\cos^2\varphi\,A^2,
\end{align}
so $\|\overrightarrow{v_1}\times\overrightarrow{v_2}^-\|^2=4A^2\left[\cos^2\delta\cos^2\varphi+C^2\right]$.
The mixed terms in $A^2+C^2$ cancel, whence
\begin{align}\label{eqn:identityAC}
A^2+C^2=\sin^2\varphi\cos^2\delta+\sin^2\delta,\quad\text{i.e.}\quad \cos^2\delta\cos^2\varphi+A^2+C^2=1.
\end{align}
In particular $\cos^2\delta\cos^2\varphi+C^2\ge\cos^2\delta\cos^2\varphi>0$, and the two direction
vectors are parallel if and only if $A=0$.

If $A\ne0$, then \eqref{eqn:distanceFormula} and \eqref{eqn:NormMinusFinal} give
\begin{align}\label{eqn:distMinusMaster}
\dist^2(\ell^+(0),\ell^-(\Delta))=\frac{4B^2}{\cos^2\delta\cos^2\varphi+C^2}.
\end{align}
For $\Delta=\pi$ (where $A=-\sin\delta\ne0$) we have $B=\cos\delta$, $C=\sin\varphi\cos\delta$, and
\eqref{eqn:distMinusMaster} equals $4$, proving the last claim. For $\Delta\ne\pi$, dividing
numerator and denominator of \eqref{eqn:distMinusMaster} by $\cos^2\delta\cos^2\frac\Delta2>0$ and
using $\tan^2\delta=\cos^2\varphi\tan^2\delta+\sin^2\varphi\tan^2\delta$ turns the denominator into
$\cos^2\varphi(1+\tan^2\delta)+(\sin\varphi\tan\delta+\tan\frac\Delta2)^2$, which is
\eqref{eq:lemmaDistl1l2Minus}.

If $A=0$, the parallel unit vectors $\overrightarrow{v_1},\overrightarrow{v_2}^-$ have the same
positive third coordinate $\cos\delta\cos\varphi$, so they are equal. Moreover
$(\overrightarrow{p_2}^--\overrightarrow{p_1})\cdot\overrightarrow{v_1}=-2\cos\varphi\cos\frac\Delta2\,A=0$,
so the distance between the two parallel lines is
$\|\overrightarrow{p_2}^--\overrightarrow{p_1}\|$, whose square
$2\cos^2\varphi(1-\cos\Delta)+4\sin^2\varphi=4(A^2+B^2)=4B^2$ agrees with
\eqref{eqn:distMinusMaster}, since \eqref{eqn:identityAC} with $A=0$ gives
$\cos^2\delta\cos^2\varphi+C^2=1$. Finally $A=0$ forces
$\tan\frac\Delta2=\sin\varphi\cot\delta>0$, so $\Delta\ne\pi$ and the reduction to
\eqref{eq:lemmaDistl1l2Minus} applies verbatim.
\end{proof}

\subsection{The formulas in $S$ and $T$}
Throughout the pinwheel construction we abbreviate
\[
S=\sin\varphi\in(0,1),\quad T=\tan\delta>0,\quad c_0=(1-S^2)(1+T^2).
\]
By Lemma~\ref{lem:distances}, every distance in the pinwheel is an algebraic function of $S$ and
$T$. In Section~\ref{sec:proofB} we choose $S$ and $T$ rational, so that the inequalities we need
become comparisons of rational numbers.

For $\Delta\in\R$ let $\Dsame(\Delta)=\dist(\ell^+(0),\ell^+(\Delta))$ and
$\Dcross(\Delta)=\dist(\ell^+(0),\ell^-(\Delta))$. With
$\cos^4\varphi=(1-S^2)^2$, \eqref{eq:lemmaDistl1l2Plus} reads
\begin{align}\label{eq:Dsame}
\Dsame(\Delta)^2
=\frac{4\sin^2\frac\Delta2\,(1-S^2)^2T^2}
{(S^2+T^2)\bigl(1+T^2-\sin^2\frac\Delta2\,(S^2+T^2)\bigr)},
\end{align}
and the bracket in the denominator is at least $1-S^2>0$, as shown in the proof of
Lemma~\ref{lem:distances}. For $\Delta\ne\pi$ write $w=\tan\frac\Delta2+ST$. The identity
$\cos^2\varphi(1+\tan^2\delta)=c_0$ turns \eqref{eq:lemmaDistl1l2Minus} into
\begin{align}\label{eq:Dcross}
\Dcross(\Delta)=\frac{2|w|}{\sqrt{w^2+c_0}},\quad \Delta\ne\pi.
\end{align}
By Lemma~\ref{lem:distances}, $\Dcross(\pi)=2$. This is also the limit of \eqref{eq:Dcross} as
$|w|\to\infty$. Both formulas are
$2\pi$-periodic in $\Delta$. For \eqref{eq:Dcross} we take $\Delta\in(-\pi,\pi]$, on which
$\Delta\mapsto\tan\frac\Delta2$ is an increasing bijection of $(-\pi,\pi)$ onto $\R$.

The two formulas cover all pairs of pinwheel lines. A rotation about the $z$-axis carries two up
lines with longitudes $\lambda_1,\lambda_2$ to the pair $(\ell^+(0),\ell^+(\Delta))$ with
$\Delta=\lambda_2-\lambda_1$, and carries an up line and a down line to the pair $(\ell^+(0),\ell^-(\Delta))$ with
$\Delta$ equal to the longitude of the down line minus that of the up line. The reflection
$(x,y,z)\mapsto(x,y,-z)$ carries two down lines to two lines of up type with the same longitudes
and tilt $-\delta$, and \eqref{eq:Dsame} depends on $\delta$ only through $T^2$. Hence two lines
of the same group with longitudes $\lambda_1,\lambda_2$ are at distance
$\Dsame(\lambda_2-\lambda_1)$, and an up line and a down line whose longitudes differ by $\Delta$
are at distance $\Dcross(\Delta)$.

\begin{lemma}\label{lem:mono}
As a function of $x=\sin^2\frac\Delta2\in(0,1]$, the right-hand side of \eqref{eq:Dsame} is
strictly increasing. Consequently, within either group of $P(N;\varphi,\delta,\psi)$ the
minimum distance is attained by adjacent lines and equals $\Dsame(\gamma)$.
\end{lemma}

\begin{proof}
The right-hand side is $\frac{Ax}{B-Cx}$ with $A=4(1-S^2)^2T^2>0$, $B=(S^2+T^2)(1+T^2)$,
$C=(S^2+T^2)^2$, and $B-Cx\ge(S^2+T^2)(1-S^2)>0$ on $[0,1]$ by the bound recorded after \eqref{eq:Dsame}. Such a
function is strictly increasing. Within a group the longitude differences are
$\Delta=j\gamma$, $1\le j\le k-1$, and $\tfrac{j\gamma}2\in[\tfrac\gamma2,\pi-\tfrac\gamma2]$,
so $\sin^2\frac{j\gamma}2\ge\sin^2\frac\gamma2$, with equality for $j\in\{1,k-1\}$.
\end{proof}

\subsection{Criterion in $S$ and $T$}\label{subsection:arc}

By Definition~\ref{def:pinwheel}, the longitude of $\ell_i^-$ minus the longitude of $\ell_j^+$ is
\[
\Delta_{ij}=\psi+(i-j)\gamma\pmod{2\pi},
\]
so the cross-group longitude differences form the set
\begin{align}\label{eq:lattice}
\Bigl\{\psi+j\gamma \bmod 2\pi:\
j=0,1,\dots,k-1\Bigr\},
\end{align}
an arithmetic progression of step $\gamma$ on the circle whose offset $\psi$ we are free to
choose.

\begin{lemma}\label{lem:arc}
Let $\tau^*=2\sin\theta$ with $\theta\in(0,\frac\pi2)$, and set
\begin{align}\label{eq:ustar}
u^*=\sqrt{c_0}\,\frac{\tau^*}{\sqrt{4-\tau^{*2}}}=\sqrt{c_0}\,\tan\theta .
\end{align}
Let
\[
J=\Bigl(2\arctan(-ST-u^*),\ 2\arctan(-ST+u^*)\Bigr).
\]
Then for $\Delta\in(-\pi,\pi]$ we have $\Dcross(\Delta)<\tau^*$ if and only if $\Delta\in J$. The
open interval $J$ has length
\[
|J|=2\bigl[\arctan(ST+u^*)-\arctan(ST-u^*)\bigr].
\]
Let
\begin{align}\label{eq:kappa}
\psi^*=\tfrac\gamma2-\arctan(ST-u^*)-\arctan(ST+u^*).
\end{align}
If $|J|\le\gamma$ and $\psi=\psi^*$, then the lattice \eqref{eq:lattice} is disjoint from $J$.
Hence every cross-group distance in $P(N;\varphi,\delta,\psi^*)$ is at least $\tau^*$.
\end{lemma}

\begin{proof}
The function $t\mapsto\frac{2t}{\sqrt{t^2+c_0}}$ is strictly increasing on $[0,\infty)$, and
$u^*$ of \eqref{eq:ustar} is the value at which it equals $\tau^*$. Hence
$\Dcross(\Delta)<\tau^*$ if and only if $|w|<u^*$. Now $|w|<u^*$ means $\tan\frac\Delta2\in(-ST-u^*,\,-ST+u^*)$, and as
$\Delta\mapsto\tan\frac\Delta2$ is an increasing bijection from $(-\pi,\pi)$ onto $\R$ (and
$\Dcross(\pi)=2\ge\tau^*$), the failure set is exactly the interval $J$. By the oddness of
$\arctan$,
\[
|J|=2\bigl[\arctan(u^*-ST)+\arctan(u^*+ST)\bigr]
=2\bigl[\arctan(ST+u^*)-\arctan(ST-u^*)\bigr].
\]
The interval $J$ is centered at
$c=\arctan(-ST-u^*)+\arctan(-ST+u^*)=\psi^*-\frac\gamma2$. With $\psi=\psi^*$ the lattice
\eqref{eq:lattice} is $\{c+\frac\gamma2+j\gamma\}$, whose distance (on the circle) from the
center $c$ is at least $\frac\gamma2\ge\frac{|J|}2$. Since $J$ is open, no lattice point lies
in $J$.
\end{proof}

Fix an even $N\ge6$, parameters $S\in(0,1)$, $T>0$, and a target $\tau^*=2\sin\theta$ with
$\theta\in(0,\frac\pi2)$. By Lemma~\ref{lem:mono}, all same-group distances are at least
$\tau^*$ if and only if $\Dsame(\frac{4\pi}N)^2\ge\tau^{*2}$. Multiplied out, this
\emph{same-group inequality} reads
\begin{align}
\sin^2\tfrac{2\pi}N\,(1-S^2)^2T^2
&\ge \sin^2\theta\Bigl[(S^2+T^2)(1+T^2)-\sin^2\tfrac{2\pi}N\,(S^2+T^2)^2\Bigr].
\tag{same-group}\label{eq:same}
\end{align}
Lemma~\ref{lem:arc} shows that all cross-group distances are at least
$\tau^*$ once the offset is chosen by \eqref{eq:kappa}, provided the interval $J$ on which
$\Dcross<\tau^*$ is no longer than $\frac{4\pi}N$. This is the \emph{cross-group inequality}
\begin{align}
\arctan\bigl(ST+\sqrt{c_0}\tan\theta\bigr)-\arctan\bigl(ST-\sqrt{c_0}\tan\theta\bigr)
&\le \frac{2\pi}N .
\tag{cross-group}\label{eq:cross}
\end{align}

\begin{theorem}\label{thm:criterion}
Let $N=2k\ge6$ be even, $S\in(0,1)$, $T>0$, and $\tau^*=2\sin\theta$ with
$\theta\in(0,\frac\pi2)$. If \eqref{eq:same} and \eqref{eq:cross} hold, then the pinwheel
$P(N;\varphi,\delta,\psi^*)$ with $\sin\varphi=S$, $\tan\delta=T$ and $\psi^*$ as in
\eqref{eq:kappa} consists of $N$ lines tangent to $\Sph^2$ with all pairwise distances at least
$\tau^*$. Consequently, by Lemma~\ref{lem:dictionary}, at least $N$ pairwise non-overlapping
cylinders of radius $\frac{\sin\theta}{1-\sin\theta}$ touch the unit ball.
\end{theorem}

\begin{proof}
Same-group pairs: by Lemma~\ref{lem:mono} the minimum is $\Dsame(\gamma)$ with
$\gamma=\frac{4\pi}N$, and \eqref{eq:same} is exactly $\Dsame(\gamma)^2\ge\tau^{*2}$ after
multiplying out the (positive) denominator of \eqref{eq:Dsame} at
$\sin^2\frac\gamma2=\sin^2\frac{2\pi}N$. Cross-group pairs: \eqref{eq:cross} is $|J|\le\gamma$,
so Lemma~\ref{lem:arc} applies. The cylinder statement follows since
$\tau^*=2\sin\theta$ corresponds under \eqref{eq:tauDef} to
$\omega=\frac{\sin\theta}{1-\sin\theta}$.
\end{proof}

\begin{remark}\label{rem:exact}
Taking $\theta=\frac\pi{N-1}$, the cylinders in Theorem~\ref{thm:criterion} have radius exactly
$\omega_{N-1}$, and $N=(N-1)+1$ of them touch $B$. Taking $\theta=\frac\pi{N-2}$ gives $N=(N-2)+2$
cylinders of radius exactly $\omega_{N-2}$. These are the two cases used in Section~\ref{sec:proofB}.
\end{remark}

\section{The pinwheel for large $n$: proof of Theorem~\ref{thm:B}}\label{sec:proofB}

\subsection{The frozen parameters}\label{subsec:frozen}
From now on
$S=\frac9{20}$, $T=\frac{31}{64}$,
so that all derived constants are rational:
\begin{align*}
&ST=\frac{279}{1280},\quad
S^2T^2=\frac{77841}{1638400},\quad
c_0=(1-S^2)(1+T^2)=\frac{1613183}{1638400}.
\end{align*}
We will need
\begin{align}\label{eq:mu2120}
&2\Phi>\tfrac{21}{20}\quad\text{and}\quad \tfrac{21}{20}\sqrt{c_0}<1+S^2T^2, \\
&\text{where }\Phi\text{ is defined by }\Phi^2=\frac{(1-S^2)^2T^2}{(S^2+T^2)(1+T^2)}=\frac{1564677136}{5658909425}.\nonumber
\end{align} 
After squaring, both inequalities in \eqref{eq:mu2120} are comparisons of rational numbers,
checked by the verifier (Table~\ref{tab:checks}).

\subsection{Elementary estimates}
Appendix~\ref{app:decimal} collects, with one-line proofs, the only analytic inequalities we
use:
\begin{enumerate}[label=(E\arabic*),leftmargin=3em]
\item\label{E:sin} $x-\frac{x^3}6\le\sin x\le x$ for $x\ge0$;
\item\label{E:tan} $\tan y\le y+y^3$ for $y\in(0,1]$;
\item\label{E:atan} $\arctan p-\arctan q=\arctan\frac{p-q}{1+pq}$ whenever $pq>-1$, and
$0\le\arctan z\le z$ for $z\ge0$;
\item\label{E:mono} $z\mapsto\dfrac{z}{1+S^2T^2-z^2}$ is strictly increasing on
$[0,\sqrt{1+S^2T^2})$.
\end{enumerate}

\subsection{The two tail lemmas}\label{subsection:twoTailLemmas}
Fix $s\in\{1,2\}$ and put $\theta=\theta(N)=\frac\pi{N-s}$ in
\eqref{eq:same}--\eqref{eq:cross}.

\begin{lemma}\label{lem:tailstar2}
Let $N\ge2s+4$ be even. If
\begin{align}\label{eq:tailstar2}
2\Bigl(1-\frac sN\Bigr)\Bigl(1-\frac{2\pi^2}{3N^2}\Bigr)\;\ge\;\frac1\Phi,
\end{align}
then \eqref{eq:same} holds with $\theta=\frac\pi{N-s}$. The left-hand side of
\eqref{eq:tailstar2} is strictly increasing in $N$, and \eqref{eq:tailstar2} holds for $s=1$,
$N=32$ and for $s=2$, $N=48$. Hence \eqref{eq:same} holds for all even $N\ge32$
(respectively $N\ge48$).
\end{lemma}

\begin{proof}
Discarding the subtracted term on the right of \eqref{eq:same} shows that
$\sin\frac{2\pi}N\,\Phi\ge\sin\theta$ is sufficient. By \ref{E:sin},
$\sin\frac{2\pi}N\ge\frac{2\pi}N\bigl(1-\frac{2\pi^2}{3N^2}\bigr)$ and
$\sin\theta\le\theta=\frac\pi{N-s}$, so \eqref{eq:tailstar2} suffices. The left side of
\eqref{eq:tailstar2} is a product of two positive factors, each strictly increasing in $N$. 
For the two threshold cases, \eqref{eq:mu2120} gives $\frac{1}{\Phi} <\frac{40}{21}<1.905$. Using $\pi^2<10$, the left-hand side of \eqref{eq:tailstar2} is at least $\frac{31}{16}\cdot\frac{763}{768}>1.92$ at $(s,N)=(1,32)$ and at least $\frac{23}{12}\cdot\frac{1723}{1728}>1.91$ at $(s,N)=(2,48)$.
\end{proof}

\begin{lemma}\label{lem:tailarc}
Let $N\ge2s+4$ be even, $\theta=\frac\pi{N-s}\le1$, and set
$\bar u=\sqrt{c_0}\,(\theta+\theta^3)$. If $\bar u^2<1+S^2T^2$ and
\begin{align}\label{eq:tailarc}
N\,\bar u\;\le\;\pi\bigl(1+S^2T^2-\bar u^2\bigr),
\end{align}
then \eqref{eq:cross} holds with $\theta=\frac\pi{N-s}$. The left-hand side of
\eqref{eq:tailarc} is strictly decreasing in $N$ and the right-hand side is increasing, so
\eqref{eq:tailarc}, once true, remains true for all larger $N$. It holds for $s=1$, $N=32$ and
for $s=2$, $N=48$, hence for all even $N\ge32$ (respectively $N\ge48$).
\end{lemma}

\begin{proof}
By \ref{E:tan}, $u^*=\sqrt{c_0}\tan\theta\le\bar u$, so $u^{*2}\le\bar u^2<1+S^2T^2$. By
\ref{E:atan} (the product $(ST+u^*)(ST-u^*)=S^2T^2-u^{*2}>-1$),
\begin{align*}
\arctan(ST+u^*)-\arctan(ST-u^*)&=\arctan\frac{2u^*}{1+S^2T^2-u^{*2}}
\le\frac{2u^*}{1+S^2T^2-u^{*2}}\\
&\le\frac{2\bar u}{1+S^2T^2-\bar u^2}\le\frac{2\pi}N,
\end{align*}
using \ref{E:mono} and \eqref{eq:tailarc}. For the monotonicity,
$N\bar u=\pi\sqrt{c_0}\,\frac N{N-s}\,\bigl(1+\theta^2\bigr)$ is a product of positive strictly
decreasing factors of $N$, while $\bar u$ itself decreases, so $\pi(1+S^2T^2-\bar u^2)$ increases.
The two threshold cases are in Table~\ref{tab:checks}.
\end{proof}

\subsection{The computer-checked inequalities}\label{subsec:checks}
Table~\ref{tab:checks} lists every numerical inequality that the pinwheel construction relies on.
In each of them, the ingredients $\pi$, $\sin$, $\tan$, $\arctan$ and $\sqrt{c_0}$ are replaced by rational bounds. Each bound is chosen on the side that makes the resulting inequality imply the original one. What remains is a comparison of two explicit rational numbers.
The standalone verifier \cite{MRVer}, invoked as
\texttt{cylie -i constructions}, performs all of them in exact arithmetic, with $512$-bit
overflow-checked integers and no floating point. The same run performs the checks
\ref{W:sphere}--\ref{W:sep} of Section~\ref{sec:window} and ends with the line
\texttt{ALL CONSTRUCTIONS VERIFIED}.

\begin{table}[t]
\begin{center}
\begin{tabular}{lll}
inequality & checked at & used in\\
\hline
\eqref{eq:same}, \eqref{eq:cross}, $\theta=\frac\pi{N-1}$ & even $N$, $20\le N\le30$ & Theorem~\ref{thm:B}\\
\eqref{eq:same}, \eqref{eq:cross}, $\theta=\frac\pi{N-2}$ & even $N$, $40\le N\le46$ & Theorem~\ref{thm:B}\\
\eqref{eq:tailarc} & $(s,N)=(1,32)$ and $(2,48)$ & Lemma~\ref{lem:tailarc}\\ 
\eqref{eq:mu2120} & independent of $N$ & Lemma~\ref{lem:tailstar2}\\
\end{tabular}
\end{center}
\caption{The numerical inequalities of the pinwheel construction, all checked by the verifier.}
\label{tab:checks}
\end{table}

\subsection{Proof of Theorem \ref{thm:B}}
Let $n\ge19$ be odd and $N=n+1$, so $N$ is even, $N\ge20$, and $\theta=\frac\pi n=\frac\pi{N-1}$.
If $20\le N\le30$, inequalities \eqref{eq:same} and \eqref{eq:cross} hold by Table~\ref{tab:checks}. 
If $N\ge32$, the inequalities hold by Lemmas~\ref{lem:tailstar2} and~\ref{lem:tailarc}. By
Theorem~\ref{thm:criterion} and Remark~\ref{rem:exact}, $N=n+1$ pairwise non-overlapping
cylinders of radius $\omega_n$ touch $B$, so $\mathcal N(\omega_n)\ge n+1$.

Let $n\ge38$ be even and $N=n+2$, so $N$ is even, $N\ge40$, and
$\theta=\frac\pi n=\frac\pi{N-2}$. If $40\le N\le46$, use Table~\ref{tab:checks}. If $N\ge48$,
use Lemmas~\ref{lem:tailstar2} and~\ref{lem:tailarc}. As before, $\mathcal N(\omega_n)\ge n+2$.
\qed

\section{The window: explicit configurations for $9\le n\le36$}\label{sec:window}
\begin{figure}[t]
\begin{tikzpicture}[scale=2.0, line cap=round, line join=round, dot/.style={circle, fill, inner sep=1.0pt}, odot/.style={circle, draw, inner sep=0.8pt, fill=white}]
\draw[line width=0.5pt] (0,0) circle (1);
\draw[thin] plot[smooth] coordinates {(-0.5608,-0.5474) (-0.4712,-0.5832) (-0.3765,-0.6125) (-0.2776,-0.6352) (-0.1756,-0.6509) (-0.0717,-0.6595) (0.0329,-0.6608) (0.1372,-0.6549) (0.2400,-0.6419) (0.3401,-0.6218) (0.4366,-0.5949) (0.5282,-0.5614) (0.6141,-0.5218) (0.6932,-0.4765) (0.7648,-0.4260) (0.8279,-0.3708) (0.8820,-0.3116) (0.9264,-0.2489) (0.9607,-0.1835) (0.9845,-0.1161) (0.9974,-0.0474) (1.0000,-0.0013)};
\draw[very thin, dashed] plot[smooth] coordinates {(0.9995,0.0218) (0.9905,0.0907) (0.9708,0.1587) (0.9404,0.2249) (0.8997,0.2887) (0.8491,0.3493) (0.7892,0.4060) (0.7207,0.4584) (0.6443,0.5057) (0.5608,0.5474) (0.4712,0.5832) (0.3765,0.6125) (0.2776,0.6352) (0.1756,0.6509) (0.0717,0.6595) (-0.0329,0.6608) (-0.1372,0.6549) (-0.2400,0.6419) (-0.3401,0.6218) (-0.4366,0.5949) (-0.5282,0.5614) (-0.6141,0.5218) (-0.6932,0.4765) (-0.7648,0.4260) (-0.8279,0.3708) (-0.8820,0.3116) (-0.9264,0.2489) (-0.9607,0.1835) (-0.9845,0.1161) (-0.9974,0.0474) (-1.0000,0.0013)};
\draw[thin] plot[smooth] coordinates {(-0.9995,-0.0218) (-0.9905,-0.0907) (-0.9708,-0.1587) (-0.9404,-0.2249) (-0.8997,-0.2887) (-0.8491,-0.3493) (-0.7892,-0.4060) (-0.7207,-0.4584) (-0.6443,-0.5057) (-0.5608,-0.5474) (-0.5608,-0.5474)};
\draw[thin] (-0.5608,-0.5474) -- (-0.6618,-0.6459);
\draw[thin] (0.5608,0.5474) -- (0.6618,0.6459);
\draw[very thin, dotted] (0.5608,0.5474) -- (-0.5608,-0.5474);
\draw[thin] (0.8279,-0.3708) -- (0.9769,-0.4376);
\draw[thin] (-0.8279,0.3708) -- (-0.9769,0.4376);
\draw[very thin, dotted] (-0.8279,0.3708) -- (0.8279,-0.3708);
\draw[thin] (0.0000,0.7502) -- (0.0000,0.8853);
\draw[thin] (0.0000,-0.7502) -- (0.0000,-0.8853);
\draw[very thin, dotted] (0.0000,-0.7502) -- (0.0000,0.7502);
\draw[very thin, dashed, red!80!black] (0.0191,-0.4232) -- (0.6207,-0.7610);
\draw[line width=1.0pt, red!80!black] (0.6384,-0.7709) -- (0.7269,-0.8206);
\node[odot, red!80!black] at (0.3730,-0.6219) {};
\draw[line width=1.0pt, red!80!black] (-0.5524,0.9910) -- (0.4714,0.9019);
\node[dot, red!80!black] at (-0.0405,0.9465) {};
\draw[line width=1.0pt, blue!80!black] (1.0980,0.0922) -- (0.5579,-0.8338);
\node[dot, blue!80!black] at (0.8279,-0.3708) {};
\draw[line width=1.0pt, green!55!black] (-0.5569,-0.9671) -- (0.0722,-0.5114);
\node[dot, green!55!black] at (-0.2423,-0.7392) {};
\draw[line width=1.0pt, green!55!black] (0.5537,0.9685) -- (0.6653,0.7526);
\draw[very thin, dashed, green!55!black] (0.6777,0.7286) -- (0.9878,0.1288);
\draw[line width=1.0pt, green!55!black] (1.0003,0.1048) -- (1.0499,0.0088);
\node[odot, green!55!black] at (0.8018,0.4887) {};
\draw[line width=1.0pt, red!80!black] (-0.0190,-1.1045) -- (-0.1252,-1.0015);
\draw[very thin, dashed, red!80!black] (-0.1429,-0.9844) -- (-0.7268,-0.4180);
\node[odot, red!80!black] at (-0.3729,-0.7613) {};
\draw[line width=1.0pt, red!80!black] (0.5523,0.2476) -- (-0.4715,0.6257);
\node[dot, red!80!black] at (0.0404,0.4367) {};
\draw[line width=1.0pt, blue!80!black] (-1.0980,0.3066) -- (-0.9495,0.3419);
\draw[very thin, dashed, blue!80!black] (-0.9360,0.3451) -- (-0.5579,0.4350);
\node[odot, blue!80!black] at (-0.8279,0.3708) {};
\draw[very thin, dashed, green!55!black] (0.5569,0.1156) -- (-0.0723,0.9454);
\node[odot, green!55!black] at (0.2423,0.5305) {};
\draw[line width=1.0pt, green!55!black] (-0.5537,-0.1170) -- (-1.0498,-0.4428);
\node[dot, green!55!black] at (-0.8018,-0.2799) {};
\end{tikzpicture}
\hfill
\begin{tikzpicture}[scale=2.0, line cap=round, line join=round, dot/.style={circle, fill, inner sep=1.0pt}, odot/.style={circle, draw, inner sep=0.8pt, fill=white}]
\draw[line width=0.5pt] (0,0) circle (1);
\draw[thin] plot[smooth] coordinates {(-0.5608,-0.5474) (-0.4712,-0.5832) (-0.3765,-0.6125) (-0.2776,-0.6352) (-0.1756,-0.6509) (-0.0717,-0.6595) (0.0329,-0.6608) (0.1372,-0.6549) (0.2400,-0.6419) (0.3401,-0.6218) (0.4366,-0.5949) (0.5282,-0.5614) (0.6141,-0.5218) (0.6932,-0.4765) (0.7648,-0.4260) (0.8279,-0.3708) (0.8820,-0.3116) (0.9264,-0.2489) (0.9607,-0.1835) (0.9845,-0.1161) (0.9974,-0.0474) (1.0000,-0.0013)};
\draw[very thin, dashed] plot[smooth] coordinates {(0.9995,0.0218) (0.9905,0.0907) (0.9708,0.1587) (0.9404,0.2249) (0.8997,0.2887) (0.8491,0.3493) (0.7892,0.4060) (0.7207,0.4584) (0.6443,0.5057) (0.5608,0.5474) (0.4712,0.5832) (0.3765,0.6125) (0.2776,0.6352) (0.1756,0.6509) (0.0717,0.6595) (-0.0329,0.6608) (-0.1372,0.6549) (-0.2400,0.6419) (-0.3401,0.6218) (-0.4366,0.5949) (-0.5282,0.5614) (-0.6141,0.5218) (-0.6932,0.4765) (-0.7648,0.4260) (-0.8279,0.3708) (-0.8820,0.3116) (-0.9264,0.2489) (-0.9607,0.1835) (-0.9845,0.1161) (-0.9974,0.0474) (-1.0000,0.0013)};
\draw[thin] plot[smooth] coordinates {(-0.9995,-0.0218) (-0.9905,-0.0907) (-0.9708,-0.1587) (-0.9404,-0.2249) (-0.8997,-0.2887) (-0.8491,-0.3493) (-0.7892,-0.4060) (-0.7207,-0.4584) (-0.6443,-0.5057) (-0.5608,-0.5474) (-0.5608,-0.5474)};
\draw[thin] (-0.5608,-0.5474) -- (-0.6618,-0.6459);
\draw[thin] (0.5608,0.5474) -- (0.6618,0.6459);
\draw[very thin, dotted] (0.5608,0.5474) -- (-0.5608,-0.5474);
\draw[line width=1.0pt, black] (0.9382,0.5716) -- (0.8850,0.4795);
\draw[very thin, dashed, black] (0.8744,0.4610) -- (0.5130,-0.1657);
\node[odot, black] at (0.7256,0.2030) {};
\draw[very thin, dashed, black] (0.2183,0.5572) -- (0.3565,0.9257);
\draw[line width=1.0pt, black] (0.3612,0.9384) -- (0.4089,1.0654);
\node[odot, black] at (0.3136,0.8113) {};
\draw[line width=1.0pt, black] (0.5228,-0.8215) -- (1.0489,-0.3685);
\node[dot, black] at (0.7858,-0.5950) {};
\draw[line width=1.0pt, black] (-1.1130,0.2454) -- (-0.8590,0.5161);
\draw[very thin, dashed, black] (-0.8420,0.5341) -- (-0.5880,0.8047);
\draw[line width=1.0pt, black] (-0.5710,0.8228) -- (-0.4355,0.9671);
\node[odot, black] at (-0.7743,0.6063) {};
\draw[line width=1.0pt, black] (-0.1392,0.2895) -- (0.8147,-0.1607);
\node[dot, black] at (0.3378,0.0644) {};
\draw[line width=1.0pt, black] (-0.5336,-0.9462) -- (-0.5852,-0.8239);
\draw[very thin, dashed, black] (-0.5955,-0.7994) -- (-0.9463,0.0322);
\node[odot, black] at (-0.7399,-0.4570) {};
\draw[line width=1.0pt, black] (-0.0749,-1.0856) -- (0.2195,-0.9764);
\draw[very thin, dashed, black] (0.2358,-0.9703) -- (0.4648,-0.8854);
\draw[line width=1.0pt, black] (0.4811,-0.8793) -- (0.5793,-0.8429);
\node[odot, black] at (0.2522,-0.9643) {};
\draw[line width=1.0pt, black] (-0.8785,0.1552) -- (-0.3795,1.0378);
\node[dot, black] at (-0.6290,0.5965) {};
\draw[line width=1.0pt, black] (0.1575,0.8159) -- (1.0381,0.3126);
\node[dot, black] at (0.5978,0.5642) {};
\draw[very thin, dashed, black] (-0.1025,-0.7622) -- (-0.4415,0.2697);
\node[odot, black] at (-0.2720,-0.2463) {};
\draw[line width=1.0pt, black] (-0.9417,-0.6030) -- (-0.1800,-0.4918);
\node[dot, black] at (-0.5609,-0.5474) {}; 
\end{tikzpicture}
\hfill
\begin{tikzpicture}[scale=2.0, line cap=round, line join=round, dot/.style={circle, fill, inner sep=1.0pt}, odot/.style={circle, draw, inner sep=0.8pt, fill=white}]
\draw[line width=0.5pt] (0,0) circle (1);
\draw[thin] plot[smooth] coordinates {(-0.7071,-0.4082) (-0.6293,-0.4487) (-0.5446,-0.4842) (-0.4540,-0.5144) (-0.3584,-0.5390) (-0.2588,-0.5577) (-0.1564,-0.5702) (-0.0523,-0.5766) (0.0523,-0.5766) (0.1564,-0.5702) (0.2588,-0.5577) (0.3584,-0.5390) (0.4540,-0.5144) (0.5446,-0.4842) (0.6293,-0.4487) (0.7071,-0.4082) (0.7771,-0.3633) (0.8387,-0.3144) (0.8910,-0.2621) (0.9336,-0.2069) (0.9659,-0.1494) (0.9877,-0.0903) (0.9986,-0.0302) (0.9998,-0.0101)};
\draw[very thin, dashed] plot[smooth] coordinates {(0.9998,0.0101) (0.9925,0.0704) (0.9744,0.1299) (0.9455,0.1880) (0.9063,0.2440) (0.8572,0.2974) (0.7986,0.3475) (0.7314,0.3938) (0.6561,0.4357) (0.5736,0.4729) (0.4848,0.5050) (0.3907,0.5315) (0.2924,0.5521) (0.1908,0.5667) (0.0872,0.5752) (-0.0175,0.5773) (-0.1219,0.5730) (-0.2250,0.5626) (-0.3256,0.5459) (-0.4226,0.5233) (-0.5150,0.4949) (-0.6018,0.4611) (-0.6820,0.4222) (-0.7547,0.3788) (-0.8192,0.3312) (-0.8746,0.2799) (-0.9205,0.2256) (-0.9563,0.1688) (-0.9816,0.1102) (-0.9962,0.0503) (-0.9998,0.0101)};
\draw[thin] plot[smooth] coordinates {(-0.9998,-0.0101) (-0.9925,-0.0704) (-0.9744,-0.1299) (-0.9455,-0.1880) (-0.9063,-0.2440) (-0.8572,-0.2974) (-0.7986,-0.3475) (-0.7314,-0.3938) (-0.7071,-0.4082)};
\draw[thin] (-0.7071,-0.4082) -- (-0.8344,-0.4817);
\draw[thin] (0.7071,0.4082) -- (0.8344,0.4817);
\draw[very thin, dotted] (0.7071,0.4082) -- (-0.7071,-0.4082);
\draw[thin] (0.7071,-0.4082) -- (0.8344,-0.4817);
\draw[thin] (-0.7071,0.4082) -- (-0.8344,0.4817);
\draw[very thin, dotted] (-0.7071,0.4082) -- (0.7071,-0.4082);
\draw[thin] (0.0000,0.8165) -- (0.0000,0.9635);
\draw[thin] (0.0000,-0.8165) -- (0.0000,-0.9635);
\draw[very thin, dotted] (0.0000,-0.8165) -- (0.0000,0.8165);
\draw[line width=1.0pt, black] (-0.4472,-0.9655) -- (-0.4472,-0.0673);
\node[dot, black] at (-0.4472,-0.5164) {};
\draw[line width=1.0pt, black] (-0.8944,0.1909) -- (-0.8944,-0.7073);
\node[dot, black] at (-0.8944,-0.2582) {};
\draw[line width=1.0pt, black] (0.8944,0.7073) -- (0.8944,0.4603);
\draw[very thin, dashed, black] (0.8944,0.4378) -- (0.8944,-0.1909);
\node[odot, black] at (0.8944,0.2582) {};
\draw[very thin, dashed, black] (0.4472,0.0673) -- (0.4472,0.8757);
\draw[line width=1.0pt, black] (0.4472,0.8981) -- (0.4472,0.9655);
\node[odot, black] at (0.4472,0.5164) {};
\draw[line width=1.0pt, black] (1.0597,0.0954) -- (0.2819,-0.3536);
\node[dot, black] at (0.6708,-0.1291) {};
\draw[line width=1.0pt, black] (-0.6125,0.8700) -- (0.1653,0.4210);
\node[dot, black] at (-0.2236,0.6455) {};
\draw[line width=1.0pt, black] (0.1653,-1.1282) -- (-0.0486,-1.0047);
\draw[very thin, dashed, black] (-0.0680,-0.9935) -- (-0.6125,-0.6792);
\node[odot, black] at (-0.2236,-0.9037) {};
\draw[line width=1.0pt, black] (0.2819,-0.8700) -- (1.0597,-0.4210);
\node[dot, black] at (0.6708,-0.6455) {};
\draw[line width=1.0pt, black] (0.6125,0.6792) -- (-0.1653,1.1282);
\node[dot, black] at (0.2236,0.9037) {};
\draw[very thin, dashed, black] (-0.2819,0.3536) -- (-0.9819,-0.0505);
\draw[line width=1.0pt, black] (-1.0014,-0.0618) -- (-1.0597,-0.0954);
\node[odot, black] at (-0.6708,0.1291) {};
\draw[very thin, dashed, black] (-0.1653,-0.4210) -- (0.5347,-0.8251);
\draw[line width=1.0pt, black] (0.5542,-0.8363) -- (0.6125,-0.8700);
\node[odot, black] at (0.2236,-0.6455) {};
\draw[line width=1.0pt, black] (-1.0597,0.4210) -- (-0.8458,0.5445);
\draw[very thin, dashed, black] (-0.8264,0.5557) -- (-0.2819,0.8700);
\node[odot, black] at (-0.6708,0.6455) {}; 
\end{tikzpicture}
\caption{The configurations that show $(9,\omega_9)$, $(10,\omega_{10})$ and
$(11,\omega_{11})\notin\Pcyl$.}
\label{fig:n9n10n11}
\end{figure}
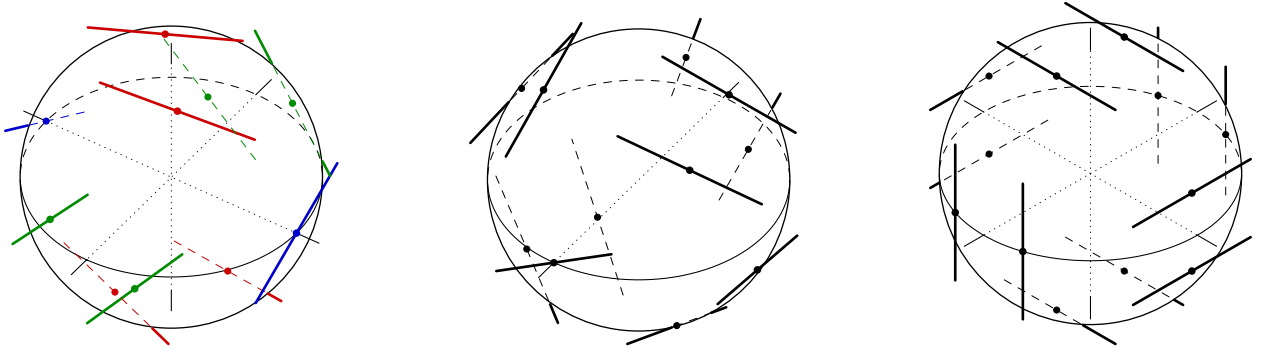

For each $n$ in the window left open by Theorems~\ref{thm:A} and~\ref{thm:B}, except $n=8$,
we give $n+1$ tangent lines with rational data. Their pairwise distances are computed exactly and
compared with a rational upper bound for $2\sin\frac\pi n$. The pinwheel is not used in this
section.

\begin{theorem}\label{thm:window}
Let $W=\{9,10,\dots,18\}\cup\{20,22,\dots,36\}$. For every $n\in W$,
\[
\mathcal N(\omega_n)\ \ge\ n+1,\quad\text{i.e.}\quad (n,\omega_n)\notin\Pcyl .
\]
\end{theorem}

\begin{proof}[Proof of Theorem \ref{thm:Ge9}]
Theorem~\ref{thm:window} covers $W$, and Theorem~\ref{thm:B} covers every odd $n\ge19$ and
every even $n\ge38$. The union is $\{n\ge9\}$.
\end{proof}

Together with Theorem~\ref{thm:A}, the
insertion question (Question~\ref{Q:insertion}) is decided for every $n\ge3$ except $n=8$: the ring
of $n$ is count-optimal at its own radius for $3\le n\le7$ and is not for any $n\ge9$.

\subsection{Rational tangent lines}

The witnesses are lines tangent to $\Sph^2$ at rational points, with integer directions.
Normalizing the direction is unnecessary, because the distance formula of
Lemma~\ref{lem:linedist} is homogeneous of degree zero in each direction vector.

\begin{definition}\label{def:ratline}
A \emph{rational tangent line} is a pair $(P,\overrightarrow{V})$ of integer vectors $P=(A,B,C,E)\in\Z^4$ and
$\overrightarrow{V}=(V_0,V_1,V_2)\in\Z^3\setminus\{0\}$ with
\begin{align}\label{eq:W12}
\text{(W1)}\quad A^2+B^2+C^2=E^2,\ \ E>0,\quad
\text{(W2)}\quad AV_0+BV_1+CV_2=0,\ \ \overrightarrow{V}\ne0 .
\end{align}
It denotes the line $\ell(P,\overrightarrow{V})=\{\overrightarrow{p}+t\overrightarrow{V}:t\in\R\}$ with $\overrightarrow{p}=\frac1E(A,B,C)$, which by (W1) and
(W2) is tangent to the unit sphere at $\overrightarrow{p}$.
\end{definition}

The following result follows from Lemma~\ref{lem:linedist}. In the parallel case we also use the
identity
$\|\overrightarrow{x}\times\overrightarrow{y}\|^2=\|\overrightarrow{x}\|^2\|\overrightarrow{y}\|^2-(\overrightarrow{x}\cdot\overrightarrow{y})^2$
for $\overrightarrow{x},\overrightarrow{y}\in\R^3$.
\begin{lemma}\label{lem:ratdist}
Let $\ell=\ell(P,\overrightarrow{V})$ and $\ell'=\ell(P',\overrightarrow{V'})$ be rational tangent lines with tangency points
$\overrightarrow{p},\overrightarrow{p'}$. Define $\overrightarrow{c}=\overrightarrow{V}\times \overrightarrow{V'}$ and $\overrightarrow{\Delta}=\overrightarrow{p'}-\overrightarrow{p}$. Then $\dist(\ell,\ell')^2$ is the rational
number
\begin{align}
\frac{(\overrightarrow{c}\cdot\overrightarrow{\Delta})^2}{\|\overrightarrow{c}\|^2}\quad\text{if }\overrightarrow{c}\ne0,\quad
\Bigl\|\overrightarrow{\Delta}-\frac{\overrightarrow{\Delta}\cdot \overrightarrow{V}}{\|\overrightarrow{V}\|^2}\,\overrightarrow{V}\Bigr\|^{2}\quad\text{if }\overrightarrow{c}=0 . \label{eqn:ratdist}
\end{align}
\end{lemma}

The threshold $2\sin\frac\pi n$ is irrational. We replace it by a slightly larger rational number.

\begin{lemma}\label{lem:sinbound}
The inequality $4\sin^2\frac\pi n \le r_n$ holds for every integer $n\geq 1$, where the rational numbers $Y_n$ and $r_n$ are defined as
$Y_n=\frac{710}{113\,n}$  and  
$r_n=Y_n^{2}-\frac{Y_n^{4}}{12}+\frac{Y_n^{6}}{360}$. 
\end{lemma}

\begin{proof}
For $y\ge0$,
\begin{align}\label{eq:cos6}
\cos y\ \ge\ 1-\frac{y^2}2+\frac{y^4}{24}-\frac{y^6}{720}.
\end{align}
Indeed, the difference of the two sides vanishes at $0$, its derivative
$-\sin y+y-\frac{y^3}6+\frac{y^5}{120}$ vanishes at $0$, the derivative of that,
$-\cos y+1-\frac{y^2}2+\frac{y^4}{24}$, vanishes at $0$, and its derivative
$\sin y-y+\frac{y^3}6$ is nonnegative by \ref{E:sin}. Integrating three times from $0$ gives
\eqref{eq:cos6}. Hence $4\sin^2\frac\pi n=2\bigl(1-\cos\frac{2\pi}n\bigr)\le\rho(2\pi/n)$ with
$\rho(y)=y^2-\frac{y^4}{12}+\frac{y^6}{360}$. Now
$\rho'(y)=\frac{y}{60}\bigl(y^4-20y^2+120\bigr)>0$ for $y>0$ (the quadratic in $y^2$ has
negative discriminant), so $\rho$ is increasing on $[0,\infty)$, and since $\pi<\frac{355}{113}$
we have $\frac{2\pi}n<Y_n$ and $\rho(2\pi/n)\le\rho(Y_n)=r_n$.
\end{proof}

\subsection{The deposited configurations and their verification}

For each $N\in\{10,11,12,\dots,21\}\cup\{23,25,\dots,37\}$ a file
\[
\texttt{cyl\_window\_N}\{N\}\texttt{.txt}
\]
contains a configuration that, for $n=N-1$, proves $(n,\omega_n)\not\in \Pcyl$. The file contains $N$ rows of seven integers
$A\ B\ C\ E\ V_0\ V_1\ V_2$. Each row corresponds to one rational tangent line. 
All integers have at most $13$ decimal digits. The verifier (Section~\ref{subsec:checks})
performs the following checks:
\begin{enumerate}[leftmargin=2.6em,label=(W\arabic*)]
\item\label{W:sphere} for every row, $A^2+B^2+C^2=E^2$ and $E>0$;
\item\label{W:tangent} for every row, $AV_0+BV_1+CV_2=0$ and $(V_0,V_1,V_2)\ne(0,0,0)$;
\item\label{W:sep} for every pair of rows, the rational squared distance of
Lemma~\ref{lem:ratdist} is at least $r_n$.
\end{enumerate}

\begin{proof}[Proof of Theorem~\ref{thm:window}]
Let $n\in W$ and $N=n+1$. The deposited file with $N$ rows passes \ref{W:sphere}--\ref{W:sep}.
By \ref{W:sphere} and \ref{W:tangent} the rows are $N$ lines tangent to $\Sph^2$. By
Lemma~\ref{lem:ratdist} and \ref{W:sep} their pairwise squared distances are at least $r_n$, and
by Lemma~\ref{lem:sinbound} $r_n\ge4\sin^2\frac\pi n=\tau(\omega_n)^2$
(Remark~\ref{rem:ringTau}). So $N$ tangent lines have pairwise distances at least
$\tau(\omega_n)$, and Lemma~\ref{lem:dictionary} gives $N=n+1$ non-overlapping cylinders of
radius $\omega_n$ touching $B$.
\end{proof}

\section{The two-point bound}\label{sec:DBSequence}

The bound we use is the two-point bound for packing problems on a compact group. For the rotation
group it was first worked out by de Oliveira Filho, Spomer, and Vallentin \cite{SOV}, in the
framework of de Laat and Vallentin \cite{dLV}. We state and prove it in the form we need.

Throughout this section, $\tau\in(0,2)$ is fixed, $L\in\mathbb N$, and, for $\ell=1,\dots,L$, we fix
\begin{enumerate}[leftmargin=2.2em,label=(\roman*)]
\item a symmetric positive-definite matrix $D_\ell\in\R^{n_\ell\times n_\ell}$, and
\item a map $T^\ell:\SO(3)\to\R^{n_\ell\times n_\ell}$ satisfying, for all $g_1,g_2,g\in\SO(3)$,
\[
T^\ell(g_1g_2)=T^\ell(g_1)\,T^\ell(g_2)
\quad\text{and}\quad
T^\ell(g)^\top D_\ell\, T^\ell(g)=D_\ell .
\]
\end{enumerate}

\begin{lemma}\label{lem:Torth}
Under (i)--(ii), for every $g\in\SO(3)$: $T^\ell(e)=I$ and
\[
T^\ell(g^{-1})\;=\;T^\ell(g)^{-1}\;=\;D_\ell^{-1}\,T^\ell(g)^\top D_\ell .
\]
\end{lemma}
\begin{proof}
Since $D_\ell$ is invertible, the identity $T^\ell(g)^\top D_\ell\,T^\ell(g)=D_\ell$ gives
$\bigl(D_\ell^{-1}T^\ell(g)^\top D_\ell\bigr)\,T^\ell(g)=I$. So $T^\ell(g)$ is invertible and
$T^\ell(g)^{-1}=D_\ell^{-1}T^\ell(g)^\top D_\ell$. Now $T^\ell(e)=T^\ell(e\cdot e)=T^\ell(e)^2$ with
$T^\ell(e)$ invertible forces $T^\ell(e)=I$. Finally $T^\ell(g)\,T^\ell(g^{-1})=T^\ell(e)=I$ gives
$T^\ell(g^{-1})=T^\ell(g)^{-1}$.
\end{proof}

\begin{definition}\label{def:DB}
For a sequence $\overrightarrow C=(C_\ell)_{\ell=1}^{L}$ of symmetric positive-semidefinite matrices
$C_\ell\in\R^{n_\ell\times n_\ell}$, we define its supremum over the safe zone as
\begin{align*}
\Ssup\bigl(\overrightarrow C\bigr)\;=\;\sup_{g\in\Dom(\tau)}\;\sum_{\ell=1}^{L}
\tr\!\Bigl(D_\ell^{-1}C_\ell\,T^\ell(g)\Bigr).
\end{align*}
\end{definition}

\begin{remark}\label{rem:entrywise}
When $D_\ell=\diag(w^\ell_1,\dots,w^\ell_{n_\ell})$ is diagonal and $C_\ell$ is symmetric,
\[
\tr\!\Bigl(D_\ell^{-1}C_\ell\,T^\ell(g)\Bigr)
=\sum_{i,j}\frac{(C_\ell)_{ij}}{w^\ell_j}\;T^\ell(g)_{ij},
\]
the sum of the entries of the Schur product of $C_\ell$ and $T^\ell(g)$, with the $j$-th column
weighted by $1/w^\ell_j$. In particular
$\tr(D_\ell^{-1}C_\ell\,T^\ell(e))=\tr(D_\ell^{-1}C_\ell)=\sum_i (C_\ell)_{ii}/w^\ell_i$.
\end{remark}

\begin{theorem}\label{thm:dbBound}
Let $\overrightarrow C=(C_\ell)_{\ell=1}^{L}$ be a sequence of symmetric positive-semidefinite
matrices $C_\ell\in\R^{n_\ell\times n_\ell}$. If $\Ssup(\overrightarrow C)\le-1$, then
\begin{align}\label{eq:dbBound}
N_{\max}(\tau)\;\le\;1+\left\lfloor\sum_{\ell=1}^{L}\tr\bigl(D_\ell^{-1}C_\ell\bigr)\right\rfloor .
\end{align}
\end{theorem}
\begin{proof}
Let $g_1,\dots,g_N$ be distinct elements of a non-overlapping set
(Definition~\ref{def:nonOverlappingX}), and put $S_\ell=\sum_{i=1}^{N}T^\ell(g_i)$. By
Lemma~\ref{lem:Torth},
\[
T^\ell\bigl(g_i^{-1}g_j\bigr)=T^\ell(g_i)^{-1}T^\ell(g_j)
=D_\ell^{-1}\,T^\ell(g_i)^\top D_\ell\,T^\ell(g_j),
\]
so that, summing over all ordered pairs and using linearity and cyclicity of the trace,
\begin{align*}
\sum_{i=1}^{N}\sum_{j=1}^{N}\tr\!\Bigl(D_\ell^{-1}C_\ell\,T^\ell\bigl(g_i^{-1}g_j\bigr)\Bigr)
&=\sum_{i,j}\tr\!\Bigl(D_\ell^{-1}C_\ell D_\ell^{-1}\,T^\ell(g_i)^\top D_\ell\,T^\ell(g_j)\Bigr)\\
&=\tr\!\Bigl(\bigl(D_\ell^{-1}C_\ell D_\ell^{-1}\bigr)\bigl(S_\ell^\top D_\ell S_\ell\bigr)\Bigr)
\;\ge\;0,
\end{align*}
since $D_\ell^{-1}C_\ell D_\ell^{-1}$ and $S_\ell^\top D_\ell S_\ell$ are congruent to $C_\ell$
and $D_\ell$, hence symmetric positive semidefinite, and for such $P,Q$ we have
$\tr(PQ)=\tr(\sqrt Q\,P\sqrt Q\,)\ge0$.

Summing over $\ell$ and separating the $N$ diagonal pairs $i=j$ (where
$T^\ell(e)=I$ gives the value $\sum_\ell\tr(D_\ell^{-1}C_\ell)$) from the $N(N-1)$ off-diagonal
pairs (where $g_i^{-1}g_j\in\Dom$ by Definition~\ref{def:nonOverlappingX}, so each term is at most
$\Ssup(\overrightarrow C)\le-1$),
\[
0\;\le\;N\sum_{\ell=1}^{L}\tr\bigl(D_\ell^{-1}C_\ell\bigr)\;-\;N(N-1) .
\]
Dividing by $N$ gives $N-1\le\sum_\ell\tr(D_\ell^{-1}C_\ell)$, which implies \eqref{eq:dbBound}.
\end{proof}

\begin{remark}\label{rem:complete}
In the deposited certificates $T^\ell_{\mathrm h}(q)$ is the matrix of the action of $R(q)$ on the
harmonic polynomials of degree $\ell$ on $\R^3$, in an integer orthogonal basis. The tables are
computed as in \cite{MRsix}. These actions are the irreducible representations of $\SO(3)$, one in
each dimension $2\ell+1$ \cite{BtD}. By Bochner's theorem \cite{Bochner}, in the form
\cite{SOV}, every continuous function of positive type on $\SO(3)$ has the form
$\sum_{\ell\ge0}\tr\bigl(F_\ell T^\ell(g)\bigr)$ with positive semidefinite matrices $F_\ell$. The
constant $-1$ in Theorem~\ref{thm:dbBound} plays the role of the term $\ell=0$. So the bound uses
all such functions of degree at most $L$. It is the Lov\'asz theta number \cite{Lovasz} of the
packing graph on $\SO(3)$, truncated at degree $L$.
\end{remark}

\section{Reduction to a polynomial problem}\label{sec:quat}
We will now find an equivalent problem in which $\SO(3)$ does not appear, and whose only objects
are polynomials in the four variables $q_0$, $q_1$, $q_2$, and $q_3$.

\subsection{Quaternions and rotations}\label{subs:quaternionsBasic}
For $q=(q_0,q_1,q_2,q_3)\in\R^4$ define
\begin{align*}
R(q)=\left[\begin{array}{ccc}
q_0^2+q_1^2-q_2^2-q_3^2 & 2(q_1q_2-q_0q_3) & 2(q_1q_3+q_0q_2)\\
2(q_1q_2+q_0q_3) & q_0^2-q_1^2+q_2^2-q_3^2 & 2(q_2q_3-q_0q_1)\\
2(q_1q_3-q_0q_2) & 2(q_2q_3+q_0q_1) & q_0^2-q_1^2-q_2^2+q_3^2
\end{array}\right],
\end{align*}
whose entries are integer quadratic forms in $q$, and write
$|q|^2=q_0^2+q_1^2+q_2^2+q_3^2$. Let $p*q$ denote the Hamilton product,
\begin{align*}
(p*q)_0&= p_0q_0-p_1q_1-p_2q_2-p_3q_3,\\
(p*q)_1&=p_0q_1+q_0p_1+p_2q_3-p_3q_2,\\
(p*q)_2&=p_0q_2+q_0p_2+p_3q_1-p_1q_3,\\
(p*q)_3&=p_0q_3+q_0p_3+p_1q_2-p_2q_1.
\end{align*}

\begin{lemma}\label{lem:quat}
The following hold as polynomial identities in $\Z[p,q]$:
\begin{enumerate}[leftmargin=2.2em,label=(\alph*)]
\item $R(q)^\top R(q)=|q|^4 I$ and $\det R(q)=|q|^6$. In particular, $R(q)\in\SO(3)$ whenever
$|q|=1$;
\item $R(p*q)=R(p)\,R(q)$ and $|p*q|^2=|p|^2|q|^2$;
\item the ten products $q_aq_b$ $(0\le a\le b\le3)$ are $\Z$-linear combinations of $|q|^2$ and the
entries $R_{ij}$ of $R(q)$. Explicitly,
$4q_0^2=|q|^2+R_{11}+R_{22}+R_{33}$, $4q_1^2=|q|^2+R_{11}-R_{22}-R_{33}$,
$4q_2^2=|q|^2-R_{11}+R_{22}-R_{33}$, $4q_3^2=|q|^2-R_{11}-R_{22}+R_{33}$,
$4q_0q_1=R_{32}-R_{23}$, $4q_0q_2=R_{13}-R_{31}$, $4q_0q_3=R_{21}-R_{12}$,
$4q_1q_2=R_{12}+R_{21}$, $4q_1q_3=R_{13}+R_{31}$, $4q_2q_3=R_{23}+R_{32}$.
\end{enumerate}
\end{lemma}
\begin{proof}
Each identity is an equality of polynomials with integer coefficients, of degree at most $6$ in
$q$ in (a) and (c) and at most $4$ in $p,q$ in (b), and is verified by expanding both sides.
\end{proof}

\begin{lemma}\label{lem:cover2}
\leavevmode
\begin{enumerate}[leftmargin=2.2em,label=(\alph*)]
\item If $q,q'$ are unit quaternions with $R(q)=R(q')$, then $q'=\pm q$.
\item Every $g\in\SO(3)$ equals $R(q)$ for some unit quaternion $q$.
\end{enumerate}
\end{lemma}
\begin{proof} (a) By Lemma~\ref{lem:quat}(c) the entries of $R(q)$ together with $|q|^2=1$ determine
every product $q_aq_b$, i.e.\ the symmetric rank-one matrix $qq^\top$. If $qq^\top=q'q'^\top$ then,
taking any index $a$ with $q_a\ne0$ (such an index exists because $|q|=1$), the $a$-th columns give
$q_a\,q=q'_a\,q'$ with $q_a^2=(q'_a)^2$, so $q'=\pm q$.

\noindent (b) First, every $g\in\SO(3)$ has an eigenvector corresponding to eigenvalue $1$. Indeed,
since $\det g=1$ and $g^\top g=I$,
\[
\det(g-I)=\det(g)\det(I-g^\top)=\det\bigl((I-g)^\top\bigr)=\det(I-g)=(-1)^3\det(g-I).
\]
Therefore, $\det(g-I)=0$ and there is a unit vector $\overrightarrow{n}$ with $g\overrightarrow{n}=\overrightarrow{n}$. Since $g$ is orthogonal and
fixes $\overrightarrow{n}$, it maps the plane $\overrightarrow{n}^\perp$ to itself, and its restriction there is an
orientation-preserving linear isometry of a plane, hence a rotation by some angle $\theta$. Let
$\overrightarrow{n}=(n_x,n_y,n_z)$ and set
\[q=\left(\cos\tfrac\theta2,\;n_x\sin\frac\theta2,\;n_y\sin\frac\theta2,\;n_z\sin\frac\theta2\right).\]
Clearly, $q$ is a unit quaternion. A direct computation from the definition of $R(q)$ shows
$R(q)\overrightarrow{n}=\overrightarrow{n}$ and that $R(q)$ rotates $\overrightarrow{n}^\perp$ by $\theta$ with the same orientation. Therefore,
$R(q)=g$.
\end{proof}

\subsection{The safe zone as a pair of polynomial inequalities}\label{subsec:safezoneQ}

By Lemma~\ref{lem:cover2} the map $q\mapsto R(q)$ carries the unit sphere of $\R^4$ onto $\SO(3)$,
two-to-one. We now describe the preimage of the safe zone $\Dom(\tau)$ of Definition~\ref{def:D} by
two polynomial inequalities in $q$.

\begin{lemma}\label{lem:nhNhHh}
Let $q\ne0$ and $g=|q|^{-2}R(q)$. Then
\begin{align}
\bigl\|\overrightarrow{e_1}\times g\overrightarrow{e_1}\bigr\|^{2}
&=\frac{4\bigl(q_0^{2}+q_1^{2}\bigr)\bigl(q_2^{2}+q_3^{2}\bigr)}{|q|^{4}}, \label{eq:NhExplicit}\\
\bigl(g\overrightarrow{e_3}-\overrightarrow{e_3}\bigr)\cdot\bigl(\overrightarrow{e_1}\times g\overrightarrow{e_1}\bigr)
&=\frac{-4\,q_1q_2}{|q|^{2}}, \label{eq:nhExplicit}\\
\bigl\|\bigl(g\overrightarrow{e_3}-\overrightarrow{e_3}\bigr)\times\overrightarrow{e_1}\bigr\|^{2}
&=\frac{4\Bigl[\bigl(q_2q_3-q_0q_1\bigr)^{2}+\bigl(q_1^{2}+q_2^{2}\bigr)^{2}\Bigr]}{|q|^{4}}.
\label{eq:HhExplicit}
\end{align}
\end{lemma}
\begin{proof}
Write $R_{ij}$ for the entries of $R(q)$.
The vectors $g\overrightarrow{e_1}$ and $g\overrightarrow{e_3}$ are the first and the third column of $g$.
Hence
\begin{align*}
\overrightarrow{e_1}\times g\overrightarrow{e_1}&=|q|^{-2}\bigl(0,\,-R_{31},\,R_{21}\bigr),\\
g\overrightarrow{e_3}-\overrightarrow{e_3}&=|q|^{-2}\bigl(R_{13},\,R_{23},\,R_{33}-|q|^{2}\bigr),
\end{align*}
and $\bigl(g\overrightarrow{e_3}-\overrightarrow{e_3}\bigr)\times\overrightarrow{e_1}=|q|^{-2}\bigl(0,\,R_{33}-|q|^{2},\,-R_{23}\bigr)$.
From the definition of $R(q)$ in Section~\ref{subs:quaternionsBasic},
\begin{align*}
R_{21}&=2(q_1q_2+q_0q_3), & R_{23}&=2(q_2q_3-q_0q_1),\\
R_{31}&=2(q_1q_3-q_0q_2), & R_{33}-|q|^{2}&=-2\bigl(q_1^{2}+q_2^{2}\bigr).
\end{align*}
Substituting these gives \eqref{eq:NhExplicit} and \eqref{eq:HhExplicit} at once.
In \eqref{eq:NhExplicit} the cross terms $\pm8q_0q_1q_2q_3$ in $R_{21}^{2}+R_{31}^{2}$ cancel.
For \eqref{eq:nhExplicit}, the left-hand side is $|q|^{-4}\bigl(-R_{23}R_{31}+(R_{33}-|q|^{2})R_{21}\bigr)$, and
\begin{align*}
&-R_{23}R_{31}+\bigl(R_{33}-|q|^{2}\bigr)R_{21}\\
&\quad=-4\bigl(q_2q_3-q_0q_1\bigr)\bigl(q_1q_3-q_0q_2\bigr)-4\bigl(q_1^{2}+q_2^{2}\bigr)\bigl(q_1q_2+q_0q_3\bigr)\\
&\quad=-4\bigl(q_1q_2q_3^{2}-q_0q_2^{2}q_3-q_0q_1^{2}q_3+q_0^{2}q_1q_2\bigr)\\
&\qquad\ -4\bigl(q_1^{3}q_2+q_0q_1^{2}q_3+q_1q_2^{3}+q_0q_2^{2}q_3\bigr)\\
&\quad=-4q_1q_2\bigl(q_0^{2}+q_1^{2}+q_2^{2}+q_3^{2}\bigr).
\end{align*}
The four terms carrying $q_0q_3$ cancel in pairs.
\end{proof}

Write $\tau^{2}=A/B$ with $A,B$ positive coprime integers and set
\begin{align}
Q_{\mathrm h}(q)&=4B\,q_1^{2}q_2^{2}-A\bigl(q_0^{2}+q_1^{2}\bigr)\bigl(q_2^{2}+q_3^{2}\bigr),
\label{eq:Qh}\\
U_{\mathrm h}(q)&=4B\Bigl[\bigl(q_2q_3-q_0q_1\bigr)^{2}+\bigl(q_1^{2}+q_2^{2}\bigr)^{2}\Bigr]
 -A\,\bigl(|q|^{2}\bigr)^{2}. \label{eq:Uh}
\end{align}
These are homogeneous of degree $4$ with integer coefficients, and both are invariant under
$q\mapsto-q$.

\begin{theorem}\label{thm:domainQ}
For every $q\ne0$,
\begin{align}\label{eq:Kdef}
|q|^{-2}R(q)\in\Dom(\tau)
\quad\Longleftrightarrow\quad
Q_{\mathrm h}(q)\ge0\ \text{ and }\ U_{\mathrm h}(q)\ge0 .
\end{align}
\end{theorem}
\begin{proof}
Let $g=|q|^{-2}R(q)$ and $\delta=\dist(\ell_0,g\cdot\ell_0)$.
The line $g\cdot\ell_0$ passes through $g\overrightarrow{e_3}$ with direction $g\overrightarrow{e_1}$.
Let $d$ be the distance from the point $g\overrightarrow{e_3}$ to the line $\ell_0$.
Since $\overrightarrow{e_1}$ is a unit vector, $d=\bigl\|\bigl(g\overrightarrow{e_3}-\overrightarrow{e_3}\bigr)\times\overrightarrow{e_1}\bigr\|$.
By \eqref{eq:HhExplicit} and \eqref{eq:Uh}, $U_{\mathrm h}(q)=B|q|^{4}\bigl(d^{2}-\tau^{2}\bigr)$.
So $U_{\mathrm h}(q)\ge0$ holds exactly when $d\ge\tau$.
Also $\delta\le d$, because $g\overrightarrow{e_3}$ lies on $g\cdot\ell_0$.

\emph{Case $\bigl(q_0^{2}+q_1^{2}\bigr)\bigl(q_2^{2}+q_3^{2}\bigr)>0$.}
By \eqref{eq:NhExplicit}, $\overrightarrow{e_1}\times g\overrightarrow{e_1}\ne0$, so the lines are not parallel.
Lemma~\ref{lem:linedist} with $\overrightarrow{p_1}=\overrightarrow{e_3}$, $\overrightarrow{v_1}=\overrightarrow{e_1}$,
$\overrightarrow{p_2}=g\overrightarrow{e_3}$ and $\overrightarrow{v_2}=g\overrightarrow{e_1}$, together with
\eqref{eq:NhExplicit} and \eqref{eq:nhExplicit}, gives
\[
\delta^{2}=\frac{4\,q_1^{2}q_2^{2}}{\bigl(q_0^{2}+q_1^{2}\bigr)\bigl(q_2^{2}+q_3^{2}\bigr)} .
\]
Hence $Q_{\mathrm h}(q)=B\bigl(q_0^{2}+q_1^{2}\bigr)\bigl(q_2^{2}+q_3^{2}\bigr)\bigl(\delta^{2}-\tau^{2}\bigr)$.
So $Q_{\mathrm h}(q)\ge0$ holds exactly when $\delta\ge\tau$.
In that case $d\ge\delta\ge\tau$, so $U_{\mathrm h}(q)\ge0$ automatically.

\emph{Case $\bigl(q_0^{2}+q_1^{2}\bigr)\bigl(q_2^{2}+q_3^{2}\bigr)=0$.}
Then $q_0=q_1=0$ or $q_2=q_3=0$. In both cases $q_1q_2=0$, so $Q_{\mathrm h}(q)=0$.
By \eqref{eq:NhExplicit} the lines are parallel.
So $\delta=d$ by Lemma~\ref{lem:linedist}, and $\delta\ge\tau$ holds exactly when $U_{\mathrm h}(q)\ge0$.
\end{proof}

We write
\[
\Kdom(\tau)=\bigl\{q\in\R^{4}\setminus\{0\}:\ Q_{\mathrm h}(q)\ge0,\ U_{\mathrm h}(q)\ge0\bigr\},
\]
a cone that is invariant under $q\mapsto\lambda q$ for every $\lambda\ne0$. By
Lemma~\ref{lem:cover2}, $\Kdom(\tau)$ is the full preimage of $\Dom(\tau)$ under
$q\mapsto|q|^{-2}R(q)$.

\section{The certificate theorem}\label{sec:certthm}

We can now state the general result. Hypotheses \ref{H:hom}--\ref{H:trace} below are finite
assertions about rational numbers and polynomials with rational coefficients, verified by
inspection, expansion, or an explicit positivity witness. Hypothesis \ref{H:cover} is an
inequality on a semialgebraic set. Section~\ref{sec:cover} proves it by a sum-of-squares
identity.

\begin{definition}\label{def:certificate}
Let $\tau\in(0,2)$ with $\tau^{2}\in\Q$, let $L\in\mathbb N$, and let $T^{*}\in\Q$. A
\emph{certificate at $(\tau,L,T^{*})$} consists of, for each $\ell=1,\dots,L$:
\begin{itemize}
\item a diagonal matrix $D_\ell=\diag(w^\ell_1,\dots,w^\ell_{2\ell+1})$ with $w^\ell_i\in\Q_{>0}$,
\item a symmetric matrix $C_\ell\in\Q^{(2\ell+1)\times(2\ell+1)}$, and
\item a matrix $T^{\ell}_{\mathrm h}$ of polynomials in $\Q[q_0,q_1,q_2,q_3]$ of size
$(2\ell+1)\times(2\ell+1)$,
\end{itemize}
subject to the following hypotheses. Set
\begin{align}
P(q)&=\sum_{\ell=1}^{L}\bigl(|q|^{2}\bigr)^{L-\ell}
\sum_{i,j=1}^{2\ell+1}\frac{(C_\ell)_{ij}\,(T^{\ell}_{\mathrm h})_{ij}(q)}{w^{\ell}_j},
\label{eq:P}\\
G(q)&=-\bigl(|q|^{2}\bigr)^{L}-P(q). \label{eq:G}
\end{align}
\begin{enumerate}[leftmargin=3.2em,label=\textup{(H\arabic*)}]
\item\label{H:hom} \emph{(Homogeneity.)} Every entry of $T^{\ell}_{\mathrm h}$ is a homogeneous
polynomial of degree $2\ell$.
\item\label{H:mult} \emph{(Multiplicativity.)}
$T^{\ell}_{\mathrm h}(p*q)=T^{\ell}_{\mathrm h}(p)\,T^{\ell}_{\mathrm h}(q)$ as an identity in
$\Q[p,q]$.
\item\label{H:orth} \emph{(Orthogonality.)}
$T^{\ell}_{\mathrm h}(q)^{\top}D_\ell\,T^{\ell}_{\mathrm h}(q)=\bigl(|q|^{2}\bigr)^{2\ell}D_\ell$ as
an identity in $\Q[q]$.
\item\label{H:psd} \emph{(Positive semidefiniteness.)} $C_\ell$ is positive semidefinite.
\item\label{H:trace} \emph{(Trace.)} $\sum_{\ell=1}^{L}\tr\bigl(D_\ell^{-1}C_\ell\bigr)=T^{*}$.
\item\label{H:cover} \emph{(Cover.)} $G(q)\ge0$ for every $q\in\Kdom(\tau)$.
\end{enumerate}
\end{definition}

\begin{theorem}\label{thm:certificate}
If a certificate at $(\tau,L,T^{*})$ exists, then
\[
N_{\max}(\tau)\;\le\;1+\left\lfloor T^{*}\right\rfloor .
\]
\end{theorem}
\begin{proof}
\emph{Step 1: $T^{\ell}_{\mathrm h}$ descends to a map on $\SO(3)$.} For $g\in\SO(3)$ choose, by
Lemma~\ref{lem:cover2}(b), a unit quaternion $q$ with $R(q)=g$, and set
$T^{\ell}(g)=T^{\ell}_{\mathrm h}(q)$. This is well defined: by Lemma~\ref{lem:cover2}(a) the only
other choice is $-q$, and by \ref{H:hom} the entries of $T^{\ell}_{\mathrm h}$ are homogeneous of the
even degree $2\ell$, so $T^{\ell}_{\mathrm h}(-q)=T^{\ell}_{\mathrm h}(q)$.

\emph{Step 2: the hypotheses of Section~\ref{sec:DBSequence}.} Each $D_\ell$ is diagonal with
positive entries, hence symmetric positive definite. Let $g_1,g_2\in\SO(3)$ and pick unit
quaternions $q^{(1)},q^{(2)}$ with $R(q^{(i)})=g_i$. By Lemma~\ref{lem:quat}(b),
$R(q^{(1)}*q^{(2)})=g_1g_2$ and $|q^{(1)}*q^{(2)}|=1$, so \ref{H:mult} gives
\begin{align*}
T^{\ell}(g_1g_2)=T^{\ell}_{\mathrm h}\bigl(q^{(1)}*q^{(2)}\bigr)
=T^{\ell}_{\mathrm h}(q^{(1)})\,T^{\ell}_{\mathrm h}(q^{(2)})=T^{\ell}(g_1)T^{\ell}(g_2),
\end{align*}
and \ref{H:orth} at a unit quaternion gives
$T^{\ell}(g)^{\top}D_\ell T^{\ell}(g)=D_\ell$. Thus the data $(D_\ell,T^{\ell})_{\ell=1}^{L}$ satisfy
(i)--(ii) of Section~\ref{sec:DBSequence}, and by \ref{H:psd} the sequence
$\overrightarrow C=(C_\ell)_{\ell=1}^{L}$ is admissible in Definition~\ref{def:DB}.

\emph{Step 3: $P$ computes the two-point sum.} For a unit quaternion $q$ we have
$|q|^{2}=1$, so by \eqref{eq:P} and Remark~\ref{rem:entrywise},
\begin{align*}
P(q)=\sum_{\ell=1}^{L}\sum_{i,j}\frac{(C_\ell)_{ij}(T^{\ell}_{\mathrm h})_{ij}(q)}{w^{\ell}_j}
=\sum_{\ell=1}^{L}\tr\!\Bigl(D_\ell^{-1}C_\ell\,T^{\ell}\bigl(R(q)\bigr)\Bigr).
\end{align*}
By \ref{H:hom} the $\ell$-th summand of \eqref{eq:P} is homogeneous of degree
$2(L-\ell)+2\ell=2L$, so $P$ is homogeneous of degree $2L$. Hence for arbitrary $q\ne0$,
\begin{align}\label{eq:PisF}
\frac{P(q)}{\bigl(|q|^{2}\bigr)^{L}}
=\sum_{\ell=1}^{L}\tr\!\Bigl(D_\ell^{-1}C_\ell\,T^{\ell}\bigl(|q|^{-2}R(q)\bigr)\Bigr).
\end{align}

\emph{Step 4: \ref{H:cover} is exactly $\Ssup(\overrightarrow C)\le-1$.} Let $g\in\Dom(\tau)$ and
pick $q\ne0$ with $g=|q|^{-2}R(q)$. By \eqref{eq:Kdef} we have $q\in\Kdom(\tau)$, so
\ref{H:cover} gives $G(q)\ge0$, i.e.\ $P(q)\le-(|q|^{2})^{L}$. Dividing by $(|q|^{2})^{L}>0$ and
using \eqref{eq:PisF},
\[
\sum_{\ell=1}^{L}\tr\!\Bigl(D_\ell^{-1}C_\ell\,T^{\ell}(g)\Bigr)\;\le\;-1 .
\]
Taking the supremum over $g\in\Dom(\tau)$ yields $\Ssup(\overrightarrow C)\le-1$.

\emph{Step 5: conclusion.} By \ref{H:trace} and Theorem~\ref{thm:dbBound},
\begin{align*}
N_{\max}(\tau)\le1+\left\lfloor\sum_{\ell}\tr(D_\ell^{-1}C_\ell)\right\rfloor
=1+\left\lfloor T^{*}\right\rfloor. \qedhere
\end{align*}
\end{proof}

In the language of Definition~\ref{def:Pcyl}, the conclusion of Theorem~\ref{thm:certificate}
reads $\bigl(1+\lfloor T^{*}\rfloor,\ \tau/(2-\tau)\bigr)\in\Pcyl$.

\section{No insertion for $3\le n\le7$: proof of Theorem~\ref{thm:A}}\label{sec:applications}

The four certificates of Table~\ref{tab:certs} prove the cases $4\le n\le7$ of
Theorem~\ref{thm:A}. They share the matrices $D_\ell$ and $T^\ell_{\mathrm h}$ of
Section~\ref{sec:cert}. They differ in $\tau^{2}$, $L$, the matrices $C_\ell$, the number $T^{*}$,
and the sum-of-squares data of Section~\ref{sec:cover}.

\begin{table}[t]
\begin{center}
\begin{tabular}{lcccc}
certificate & $\tau$ & $\omega=\frac{\tau}{2-\tau}$ & $L$ & proves\\
\hline
\texttt{cyl\_N4\_omega39\_25} & $\frac{39}{32}$ & $\frac{39}{25}$ & $4$ &  $(4,\frac{39}{25})\in\Pcyl$\\[2pt]
\texttt{cyl\_N5\_omega1203\_845} & $\frac{1203}{1024}$ & $\frac{1203}{845}$ & $4$ & $(5,\frac{1203}{845})\in\Pcyl$\\[2pt]
\texttt{cyl\_N6\_omega1\_1} & $1$ & $1$ & $4$ &  $(6,1)\in\Pcyl$\\[2pt]
\texttt{cyl\_N7\_omega14558721\_18995711} & $\frac{14558721}{2^{24}}$ & $\frac{14558721}{18995711}$ & $6$ & $(7,\frac{14558721}{18995711})\in\Pcyl$\\[2pt]
\end{tabular}
\end{center}
\caption{The deposited certificates. Each file is named after the statement it proves.}
\label{tab:certs}
\end{table}

By Lemma~\ref{lem:Pcylmono} it suffices to prove $(n,\omega')\in\Pcyl$ for some rational
$\omega'\le\omega_n$. The values of $\omega'$ in Table~\ref{tab:certs} are simply the ones our search
produced.

\begin{proof}[Proof of Theorem \ref{thm:A}]
\emph{$n=3$.} We have $\omega_3=3+2\sqrt3>1+\sqrt2$. Theorem~\ref{thm:OS} gives
$(3,\omega_3)\in\Pcyl$.

\emph{$n=4$.} The certificate \texttt{cyl\_N4\_omega39\_25}, at $\tau=\frac{39}{32}$, gives
$N_{\max}\bigl(\frac{39}{32}\bigr)\le4$ by Theorem~\ref{thm:certificate}. So
$(4,\frac{39}{25})\in\Pcyl$. Also $\frac{39}{25}\le1+\sqrt2=\omega_4$, since
$14^2=196\le1250=2\cdot25^2$.

\emph{$n=5$.} The certificate \texttt{cyl\_N5\_omega1203\_845}, at $\tau=\frac{1203}{1024}$, gives
$N_{\max}\bigl(\frac{1203}{1024}\bigr)\le5$. So $(5,\frac{1203}{845})\in\Pcyl$. It remains to check
$\omega_5\ge\frac{1203}{845}$, that is, $\sin\frac\pi5\ge\frac{1203}{2048}$. Squaring and using
$16\sin^2\frac\pi5=10-2\sqrt5$, this is equivalent in turn to
\[
10-2\sqrt5\ge\frac{1447209}{262144}
\iff 1174231\ge 524288\sqrt5
\iff 1174231^2\ge5\cdot524288^2.
\]
Indeed $1378818441361\ge1374389534720$.

\emph{$n=6$.} The certificate \texttt{cyl\_N6\_omega1\_1} gives $N_{\max}(1)\le6$. This was first
proved in \cite{MRsix}. Since $2\sin\frac\pi6=1$ corresponds to $\omega=\omega_6=1$, we get
$(6,\omega_6)\in\Pcyl$.

\emph{$n=7$.} The certificate \texttt{cyl\_N7\_omega14558721\_18995711}, at
$\tau=\frac{14558721}{2^{24}}$, gives $N_{\max}(\tau)\le7$. So $(7,\frac{\tau}{2-\tau})\in\Pcyl$. It
remains to check $\omega_7\ge\frac{\tau}{2-\tau}$, that is, $\sin\frac\pi7\ge\frac\tau2$. We have
$\sin^2\frac\pi7=\frac12(1-c)$, where $c=\cos\frac{2\pi}7$ is the largest root of
$p(x)=8x^3+4x^2-4x-1$. So the inequality is equivalent to $c\le r$, where $r=1-\frac{\tau^2}2$. The
polynomial $p$ is increasing to the right of its largest root, and $r>0.62$ exceeds the other two
roots. Hence the inequality is equivalent to $p(r)\ge0$. With
$\tau^2=\frac{211956357155841}{2^{48}}$, the inequality $p(r)>0$ is a comparison of two integers,
which holds.
\end{proof}

\begin{remark}\label{rem:firsching}
The case $n=4$ of Theorem~\ref{thm:A} was proved earlier. Ogievetsky and Shlosman
\cite[Section~6.3]{OSsix} proved it with an argument about inscribed balls. Firsching \cite{F},
building on Bra\ss{} and Wenk \cite{BW}, gave a stronger bound.
\end{remark}

\section{The certificate data}\label{sec:cert}
This section describes the deposited data and the checks that establish
\ref{H:hom}--\ref{H:trace}. How the data was found plays no role here. The \emph{universal} part, shared by all certificates, consists, for
each $\ell=1,\dots,6$, of:
\begin{itemize}[leftmargin=1.4em]
\item positive rationals $w^\ell_1,\dots,w^\ell_{2\ell+1}$, forming the diagonal matrix
$D_\ell=\diag(w^\ell_1,\dots,w^\ell_{2\ell+1})$;
\item an integer table defining a $(2\ell+1)\times(2\ell+1)$ matrix $T^\ell_{\mathrm h}(q)$ of
polynomials in $q=(q_0,q_1,q_2,q_3)$, with coefficients in
$\tfrac1{\mathrm{Tden}_\ell}\Z$ for a deposited scale $\mathrm{Tden}_\ell\in\{1,2,4\}$
(the deposited values are $1,1,2,1,4,2$ for $\ell=1,\dots,6$).
\end{itemize}
A certificate with truncation $L$ uses the first $L$ universal blocks. The deposited certificates
have $L=4$ or $L=6$.
Each certificate additionally deposits the rational $\tau^2$, the number $T^*$ with a scale $M$
(Appendix~\ref{app:data}), and
symmetric rational matrices $C_\ell\in\Q^{(2\ell+1)\times(2\ell+1)}$ together with integer
witnesses for their positive semidefiniteness.

Positivity of each $w^\ell_i$, required by Definition~\ref{def:certificate}, is read off from the
deposited numerators and denominators. Hypotheses \ref{H:hom}, \ref{H:mult}, and \ref{H:orth}
are finite computations on the universal tables, performed once for all certificates.
Hypothesis \ref{H:psd} follows from the deposited witnesses by Proposition~\ref{prop:psd}, and
\ref{H:trace} is a finite computation per certificate. Hypothesis \ref{H:cover} is the subject
of Section~\ref{sec:cover}.

\subsection{The identities \ref{H:hom}--\ref{H:orth}}\label{subsec:maps}

Hypothesis \ref{H:hom} is an inspection of the exponents in the deposited table: every row of
$T^\ell_{\mathrm h}$ carries exponents summing to $2\ell$.

Both \ref{H:mult} and \ref{H:orth} are identities between polynomials with integer coefficients
once the deposited scales are cleared. Writing
$A^\ell=\mathrm{Tden}_\ell\,T^\ell_{\mathrm h}$, which has integer entries, and
$\widehat D_\ell=m_\ell D_\ell$, where $m_\ell$ is a common denominator of
$w^\ell_1,\dots,w^\ell_{2\ell+1}$, they read
\begin{align*}
\mathrm{Tden}_\ell\,A^\ell(p*q)=A^\ell(p)\,A^\ell(q),
\quad
A^\ell(q)^\top \widehat D_\ell\,A^\ell(q)=\mathrm{Tden}_\ell^{2}\,|q|^{4\ell}\,\widehat D_\ell ,
\end{align*}
so each amounts to comparing two polynomial expansions with integer coefficients, coefficient by
coefficient.

\subsection{Positive semidefiniteness \ref{H:psd}}\label{subsec:C}

The positive semidefiniteness of the matrices $C_1,\dots,C_L$ of each certificate follows from
the deposited integer witnesses by the following two elementary facts. The first is a special case of Gershgorin's circle theorem \cite{Gershgorin}, which states that
every eigenvalue $\lambda$ of a square matrix $E$ satisfies
$|\lambda-E_{ii}|\le\sum_{j\ne i}|E_{ij}|$ for some $i$.

\begin{lemma}\label{lem:domPSD}
Let $E\in\R^{n\times n}$ be symmetric with
$E_{ii}\ge\sum_{j\ne i}|E_{ij}|$ for every $i$. Then $E\succeq0$.
\end{lemma}
\begin{proof}
Let $v\in\R^{n}$. Using $|v_iv_j|\le\tfrac12(v_i^{2}+v_j^{2})$ and then the symmetry of $E$,
\begin{align*}
\Bigl|\sum_{i\ne j}E_{ij}v_iv_j\Bigr|
\;\le\;\sum_{i\ne j}|E_{ij}|\,\tfrac12\bigl(v_i^{2}+v_j^{2}\bigr)
\;=\;\sum_{i}v_i^{2}\sum_{j\ne i}|E_{ij}|
\;\le\;\sum_{i}E_{ii}v_i^{2},
\end{align*}
so $v^{\top}Ev=\sum_i E_{ii}v_i^{2}+\sum_{i\ne j}E_{ij}v_iv_j\ge0$.
\end{proof}

\begin{proposition}\label{prop:psd}
Write $\widehat C_\ell=2^{\kappa}C_\ell$ for the deposited integer matrix, where $\kappa$ is the
even integer deposited as \texttt{cDenLog2}, and let $S_\ell$ and
$W_\ell$ be the deposited positive integer scalar and integer matrix. If
\[
E_\ell\;=\;S_\ell\,\widehat C_\ell\;-\;W_\ell^{\top}W_\ell
\]
is symmetric and diagonally dominant, in the sense that
$(E_\ell)_{ii}\ge\sum_{j\ne i}|(E_\ell)_{ij}|$ for every $i$, then $C_\ell$ is symmetric positive
semidefinite.
\end{proposition}
\begin{proof}
By Lemma~\ref{lem:domPSD}, $E_\ell\succeq0$. Also $W_\ell^{\top}W_\ell\succeq0$, since
$v^{\top}W_\ell^{\top}W_\ell v=\|W_\ell v\|^{2}\ge0$ for every $v$. Hence
\[
2^{\kappa}S_\ell\,C_\ell\;=\;S_\ell\,\widehat C_\ell\;=\;E_\ell+W_\ell^{\top}W_\ell\;\succeq\;0,
\]
and dividing by the positive scalar $2^{\kappa}S_\ell$ gives $C_\ell\succeq0$. Symmetry of $C_\ell$ is
symmetry of the deposited $\widehat C_\ell$.
\end{proof}

\subsection{The trace sum \ref{H:trace}}\label{subsec:trace}

The verifier evaluates $\sum_{\ell}\sum_i(C_{\ell})_{ii}/w^{\ell}_i$ exactly and checks that it equals the deposited $T^*$. This is \ref{H:trace}.

\subsection{The certificate polynomial}\label{subsec:P}

The polynomial $P$ of Definition~\ref{def:certificate}, formed from the deposited $C_\ell$, the
tables $T^\ell_{\mathrm h}$ and the weights, has rational coefficients and is homogeneous of degree
$2L$: by \ref{H:hom} each $T^\ell_{\mathrm h}(q)_{ij}$ is homogeneous of degree $2\ell$, and
$2(L-\ell)+2\ell=2L$ for every $\ell$. The same holds for $G=-(|q|^{2})^{L}-P$.

\section{A sum-of-squares proof of hypothesis \ref{H:cover}}\label{sec:cover}

By Theorem~\ref{thm:certificate} and Section~\ref{sec:cert}, each certificate proves its bound
once \ref{H:cover} is established for it, that is, once we show
\begin{align}\label{eq:gcone}
G(q)\ \ge\ 0\quad\text{on}\quad
\Kdom(\tau)=\{q\ne0:\ Q_{\mathrm h}(q)\ge0\ \text{and}\ U_{\mathrm h}(q)\ge0\},
\end{align}
with $Q_{\mathrm h}$, $U_{\mathrm h}$ formed from the certificate's $\tau^2=A/B$ by
\eqref{eq:Qh}--\eqref{eq:Uh}.

The statement asks that one polynomial be non-negative wherever two others are. The certificate
proves it by an identity that writes a positive multiple of $G$ as
$\sigma_0+\sigma_1U_{\mathrm h}+\sigma_2Q_{\mathrm h}+\sigma_3U_{\mathrm h}Q_{\mathrm h}$, where
each $\sigma_a$ is a sum of squares of polynomials. Such a $\sigma_a$ is non-negative
everywhere, so the right-hand side is non-negative wherever $U_{\mathrm h}$ and $Q_{\mathrm h}$
are. Given such an identity, \eqref{eq:gcone} follows in two lines
(Theorem~\ref{thm:sos}), and verifying the identity is a computation of the kind the verifier
already performs for \ref{H:mult} and \ref{H:orth}: expand both sides and compare integer
coefficients. The method of certifying polynomial inequalities on semialgebraic sets by sums of
squares goes back to the Positivstellens\"atze of real algebraic geometry and to its
computational development by Lasserre and Parrilo \cite{Lasserre,Parrilo}.

\subsection{Sums of squares, and how a matrix certifies one}\label{subsec:sosbasics}

Fix an integer $d\ge0$ and let $\mathfrak m_d(q)$ denote the column vector of all
$n_d=\binom{d+3}{3}$ monomials $q_0^{e_0}q_1^{e_1}q_2^{e_2}q_3^{e_3}$ of total degree $d$, listed
in a fixed order. The verifier uses the lexicographic order of the exponent vectors.

\begin{lemma}\label{lem:gram}
Let $\Sigma\in\R^{n_d\times n_d}$ be symmetric positive semidefinite, and define
\[
\sigma(q)\;=\;\mathfrak m_d(q)^{\top}\,\Sigma\,\mathfrak m_d(q).
\]
Then $\sigma$ is a homogeneous polynomial of degree $2d$ and $\sigma(q)\ge0$ for every
$q\in\R^{4}$. Moreover $\sigma$ is a sum of squares of forms of degree $d$.
\end{lemma}
\begin{proof}
Each entry of $\mathfrak m_d(q)\mathfrak m_d(q)^{\top}$ is a monomial of degree $2d$, so $\sigma$
is a form of degree $2d$. For fixed $q$ the number $\sigma(q)=v^{\top}\Sigma v$ with
$v=\mathfrak m_d(q)$ is non-negative because $\Sigma\succeq0$. For the last statement, write
$\Sigma=W^{\top}W$, for instance from the spectral decomposition. Then
$\sigma(q)=\|W\mathfrak m_d(q)\|^{2}$ is the sum of the squares of the forms $\bigl(W\mathfrak
m_d(q)\bigr)_i$.
\end{proof}

Only the inequality $\sigma\ge0$ is used below, and it needs only $\Sigma\succeq0$. Positive
semidefiniteness of a rational $\Sigma$ is established from a deposited integer witness by
Proposition~\ref{prop:psd}, exactly as for the matrices $C_\ell$.

\subsection{The certificate identity}\label{subsec:sosthm}

\begin{theorem}\label{thm:sos}
Fix a certificate at $(\tau,L,T^{*})$, let $G$ be its cover polynomial \eqref{eq:G}, and let
$Q_{\mathrm h}$, $U_{\mathrm h}$ be as in \eqref{eq:Qh}--\eqref{eq:Uh}. Suppose we are given an
integer $k\ge0$ with $L+k\ge4$ and four symmetric positive semidefinite matrices with rational
entries,
\begin{align*}
\Sigma_0\in\Q^{n_{L+k}\times n_{L+k}},\quad
\Sigma_1,\Sigma_2\in\Q^{n_{L+k-2}\times n_{L+k-2}},\quad
\Sigma_3\in\Q^{n_{L+k-4}\times n_{L+k-4}},
\end{align*}
such that, with $\sigma_a=\mathfrak m^{\top}\Sigma_a\,\mathfrak m$ formed from the monomial vector
of the matching degree, the identity
\begin{align}\label{eq:sos}
\bigl(|q|^{2}\bigr)^{k}\,G
\;=\;\sigma_0\;+\;\sigma_1\,U_{\mathrm h}\;+\;\sigma_2\,Q_{\mathrm h}
\;+\;\sigma_3\,U_{\mathrm h}\,Q_{\mathrm h}
\end{align}
holds in $\Q[q_0,q_1,q_2,q_3]$. Then hypothesis \ref{H:cover} holds.
\end{theorem}
\begin{proof}
Let $q\in\Kdom(\tau)$, so that $U_{\mathrm h}(q)\ge0$ and $Q_{\mathrm h}(q)\ge0$. By
Lemma~\ref{lem:gram} each $\sigma_a(q)\ge0$, so every one of the four summands on the right-hand
side of \eqref{eq:sos} is non-negative at $q$, and therefore $\bigl(|q|^{2}\bigr)^{k}G(q)\ge0$.
Since $q\ne0$, dividing by $\bigl(|q|^{2}\bigr)^{k}>0$ gives $G(q)\ge0$.
\end{proof}
 
\subsection{What is deposited, and how it is verified}\label{subsec:sosverify}

For each certificate the following records are deposited, in the format of
Appendix~\ref{app:data}:
\begin{itemize}[leftmargin=1.4em]
\item the integer $k$ of \eqref{eq:sos}, as \texttt{sosExponentK},
\item an even integer $\kappa_{\mathrm s}$, as \texttt{sosDenLog2},
\item the integer symmetric matrices \texttt{sosGram000$a$}${}=2^{\kappa_{\mathrm s}}\Sigma_a$ for
$a=0,\dots,3$,
\item for each $a$, the integer witnesses \texttt{sosSwit000$a$} and \texttt{sosWwit000$a$} of
Proposition~\ref{prop:psd}.
\end{itemize} The verifier then performs the following steps, all in exact
arithmetic.

\begin{enumerate}[leftmargin=2.6em,label=\textup{(S\arabic*)}]
\item\label{S:shape} \emph{Shape.} Each deposited $2^{\kappa_{\mathrm s}}\Sigma_a$ is symmetric
and of the size prescribed by Theorem~\ref{thm:sos}, namely $n_{L+k}$, $n_{L+k-2}$, $n_{L+k-2}$
and $n_{L+k-4}$, computed from the deposited $L$ and $k$.
\item\label{S:psd} \emph{Positivity.} For each $a$, the matrix
$E_a=S_a\bigl(2^{\kappa_{\mathrm s}}\Sigma_a\bigr)-W_a^{\top}W_a$ is symmetric and diagonally
dominant, whence $\Sigma_a\succeq0$ by Proposition~\ref{prop:psd}.
\item\label{S:assemble} \emph{Assembly.} The polynomial $G$ is built from the deposited $C_\ell$,
the universal tables and the weights, as in Section~\ref{subsec:P}. The generators
$Q_{\mathrm h}$ and $U_{\mathrm h}$ are read from the certificate records and validated against
\eqref{eq:Qh}--\eqref{eq:Uh} at the deposited $A$ and $B$ (check \ref{V:gens} of
Appendix~\ref{app:data}). The left-hand side of \eqref{eq:sos} is $G$ multiplied $k$ times by
$|q|^{2}$.
\item\label{S:identity} \emph{The identity.} The right-hand side of \eqref{eq:sos} is expanded:
for each $a$ and each pair $i\le j$ of indices, the entry $(\Sigma_a)_{ij}$ contributes
$(2-\delta_{ij})(\Sigma_a)_{ij}$ times the monomial $(\mathfrak m)_i(\mathfrak m)_j$, and the
blocks $a=1,2,3$ are then multiplied by $U_{\mathrm h}$, $Q_{\mathrm h}$ and
$U_{\mathrm h}Q_{\mathrm h}$ respectively. Both sides are scaled by a common denominator
$2^{\kappa_{\mathrm s}}$. The identity becomes an equality between two lists of $n_{2L+2k}$
integers, compared entry by entry.
\end{enumerate}

If \ref{S:shape}--\ref{S:identity} succeed, then Theorem~\ref{thm:sos} applies and \ref{H:cover}
holds. The number of coefficient equations in \ref{S:identity} is
$n_{2L+2k}=\binom{2L+2k+3}{3}$. All intermediate integers are bounded by an explicit function of
the deposited scales and are checked against the machine word width as in
Appendix~\ref{app:data}. A reader who wishes to reimplement the check needs only a parser for
the format of Appendix~\ref{app:data} and exact polynomial multiplication in four variables.

Hypotheses \ref{H:hom}--\ref{H:cover} are established for each of the deposited certificates, so
Theorem~\ref{thm:certificate} yields the bounds of Table~\ref{tab:certs}, and with them the cases
$4\le n\le7$ of Theorem~\ref{thm:A}.

\begin{remark}\label{rem:found}
The window configurations come from numerical
optimization of the minimum pairwise distance, followed by rounding to rational tangent lines. The
certificates come from semidefinite programs. Their numerical solutions were rounded to rational
data and then repaired so that all identities hold exactly.
\end{remark}

\section{Open problems}\label{sec:open}

The smallest open case of Question~\ref{Q:insertion} is $n=8$. We believe that the answer is ``no''.

\begin{conjecture}\label{conj:eight}
$(8,\omega_8)\in\Pcyl$.
\end{conjecture}

Question~\ref{Q:expansion} is open for every $n\ge5$.

\begin{problem}
Determine $R_n$ for $n\ge5$ \cite[Section~2.4, Problem~17]{BMP}.
\end{problem}

Brass, Moser, and Pach \cite[Section~2.4]{BMP} also asked for the asymptotic behavior of $R_n$. The
ring gives $\liminf_{n\to\infty}nR_n\ge\pi$.

\begin{problem}
Determine $\lim_{n\to\infty}nR_n$, or show that it does not exist.
\end{problem}

\appendix

\section{Elementary estimates}\label{app:decimal}

This appendix proves the estimates \ref{E:sin}--\ref{E:mono} of Section~\ref{sec:proofB}, used
in Lemmas~\ref{lem:tailstar2}, \ref{lem:tailarc} and~\ref{lem:sinbound}.

\begin{lemma}\label{lem:elementary}
Statements \textup{\ref{E:sin}--\ref{E:mono}} hold.
\end{lemma}

\begin{proof}
\ref{E:sin}: $f(x)=\sin x-x+\frac{x^3}6$ has $f(0)=0$ and $f'(x)=\cos x-1+\frac{x^2}2\ge0$,
since $g(x)=\cos x-1+\frac{x^2}2$ has $g(0)=0$ and $g'(x)=x-\sin x\ge0$. The upper bound is
standard.

\ref{E:tan}: $\frac{\tan y}y$ is increasing on $(0,\frac\pi2)$ because
$(y\sec^2y-\tan y)=\frac{y-\sin y\cos y}{\cos^2y}>0$. Hence for $y\in(0,1]$,
$\tan y\le y\tan1<y\sqrt3$, using $1<\frac\pi3$ and the monotonicity of $\tan$. Then
$h(y)=y+y^3-\tan y$ has $h(0)=0$ and $h'(y)=3y^2-\tan^2y>3y^2-3y^2=0$.

\ref{E:atan}: apply $\tan$ to both sides for the identity. The condition $pq>-1$ keeps the
difference in $(-\frac\pi2,\frac\pi2)$. The inequality $\arctan z\le z$ holds since $\tan$ is
convex on $[0,\frac\pi2)$.

\ref{E:mono}: the derivative has numerator $1+S^2T^2+z^2>0$.
 
\end{proof}

\section{Deposited data and standalone verifier}\label{app:data}

The data are in plain-text files in a self-describing format, designed to be read by a short
standalone parser. Each record consists of a key, written \texttt{[key]Name[/key]}, followed by a
value block \texttt{[value]}\,$\dots$\,\texttt{[/value]}. A value block holds one or more rows.
Each row, and each entry within a row, is delimited by \texttt{\{} and \texttt{\}}. Every value is
therefore a nested list of integers of depth two, including those that carry a single number, which
appear as one row with one entry. Keys that depend on an index $\ell$ or $a$ end in it as a
four-digit number.

The \emph{universal} records, identical in every certificate file, are, for each
$\ell=1,\dots,L$:
\begin{itemize}[leftmargin=1.4em]
\item \texttt{WNum000$\ell$} and \texttt{WDen000$\ell$}: the weights $w^\ell_i$ as
numerator/denominator pairs, forming \[D_\ell=\diag(w^\ell_1,\dots,w^\ell_{2\ell+1});\]
\item \texttt{Tpoly000$\ell$} together with the scale \texttt{Tden000$\ell$}. Each row has 7 numbers
$i$, $j$, $e_0$, $e_1$, $e_2$, $e_3$, $c$. These seven numbers mean that the polynomial
$T^\ell_{\mathrm h}(q)_{ij}$ contains the term
$(c/\mathrm{Tden}_\ell)\,q_0^{e_0}q_1^{e_1}q_2^{e_2}q_3^{e_3}$. The polynomial
$T^\ell_{\mathrm h}(q)_{ij}$ is the sum over the rows carrying that $i,j$.
\end{itemize}
Each \emph{certificate} file records:
\begin{itemize}[leftmargin=1.4em]
\item \texttt{Lmax}: the truncation $L$. All degree-dependent constants of the verifier are
derived from it;
\item \texttt{targetBound}: the positive integer $N$ of the claim. It enters the verification
only through the counting check \ref{V:count};
\item \texttt{tauSqNum} and \texttt{tauSqDen}: the positive integers $A$ and $B$ with
$\tau^2=A/B$;
\item \texttt{UhPoly} and \texttt{QhPoly}: the integer coefficient lists of the generators
$U_{\mathrm h}$ and $Q_{\mathrm h}$, in the row format of \texttt{Tpoly}. The verifier also
constructs these polynomials from the closed forms \eqref{eq:Qh}--\eqref{eq:Uh} at the deposited
$A$ and $B$, and checks that the deposited lists agree with them (check \ref{V:gens} below); 
\item \texttt{MNum}, \texttt{MDen}, \texttt{TstarNum}, \texttt{TstarDen}: a rational number $M<0$
and the trace sum $T^*$. The files are not normalized. Dividing the matrices $C_\ell$, the number
$T^*$, the polynomial $G$ and the matrices $\Sigma_a$ by $-M$ gives a certificate in the sense of
Definition~\ref{def:certificate}. All checks below are invariant under this scaling;
\item \texttt{cDenLog2}: the even integer $\kappa\in[0,60]$ scaling the next record;
\item \texttt{Cscaled000$\ell$} $=2^{\kappa}C_\ell$, as the rows of an integer matrix;
\item \texttt{Swit000$\ell$} and \texttt{Wwit000$\ell$}: the integer scalar $S_\ell$ and the
integer matrix $W_\ell$ of Proposition~\ref{prop:psd};
\item \texttt{sosExponentK}: the integer $k$ of the identity \eqref{eq:sos};
\item \texttt{sosDenLog2}: the even integer $\kappa_{\mathrm s}$ scaling the next record;
\item \texttt{sosGram000$a$} $=2^{\kappa_{\mathrm s}}\Sigma_a$ for $a=0,\dots,3$, as the rows of
integer symmetric matrices of the sizes prescribed by Theorem~\ref{thm:sos};
\item \texttt{sosSwit000$a$} and \texttt{sosWwit000$a$}: the integer scalar $S_a$ and the integer
matrix $W_a$ of Proposition~\ref{prop:psd} for $\Sigma_a$.
\end{itemize}

Every other quantity in the proof is derived from these records: the reciprocals $1/w^\ell_i$
from \texttt{WNum}/\texttt{WDen}, and the polynomials $P$ and $M(|q|^{2})^{L}-P$ from the
$C_\ell$, the tables and the weights. Apart from $U_{\mathrm h}$, $Q_{\mathrm h}$ and the powers
of $|q|^2$, the verifier constructs no polynomial in $q_0,\dots,q_3$ other than by expanding
deposited records.

On these records the standalone verifier confirms, for each certificate in turn:
\begin{enumerate}[leftmargin=3.2em,label=\textup{(V\arabic*)}]
\item\label{V:dom} the claim is well formed: \texttt{targetBound}${}\ge1$, $M<0$, and
$0<\tau^{2}<4$, so that $\omega=\tau/(2-\tau)$ is defined and positive;
\item\label{V:gens} the deposited generator records \texttt{UhPoly} and \texttt{QhPoly}
coincide, coefficient by coefficient, with the closed forms \eqref{eq:Qh}--\eqref{eq:Uh}
evaluated at the deposited $A$ and $B$;
\item\label{V:tab} \ref{H:hom} and the positivity of the weights, by inspection of the exponents in
\texttt{Tpoly000$\ell$} and of the signs in \texttt{WNum000$\ell$}, \texttt{WDen000$\ell$};
\ref{H:mult} and \ref{H:orth}, by expanding both sides of the two identities of
Section~\ref{subsec:maps} and comparing them coefficient by coefficient (once, since the tables are
universal);
\item\label{V:psd} the hypothesis of Proposition~\ref{prop:psd}, forming
$E_\ell=S_\ell\bigl(2^{\kappa}C_\ell\bigr)-W_\ell^{\top}W_\ell$ and testing symmetry and diagonal
dominance;
\item\label{V:tr} \ref{H:trace}, by evaluating $\sum_{\ell}\sum_i (C_\ell)_{ii}/w^\ell_i$ and
comparing it with the deposited $T^*$;
\item\label{V:count} the \emph{counting inequality}
\begin{align*}
T^{*}\;<\;(-M)\cdot\texttt{targetBound},
\quad\text{equivalently}\quad
\texttt{targetBound}\;\ge\;1+\left\lfloor\frac{T^{*}}{-M}\right\rfloor,
\end{align*}
in exact rational arithmetic, which by Theorem~\ref{thm:certificate} turns
hypotheses \ref{H:hom}--\ref{H:cover} into the bound $N_{\max}(\tau)\le\texttt{targetBound}$;
\item\label{V:sospsd} the hypothesis of Proposition~\ref{prop:psd} for each multiplier matrix
$\Sigma_a$, from the records \texttt{sosSwit000$a$} and \texttt{sosWwit000$a$}, establishing
$\Sigma_a\succeq0$ (step \ref{S:psd} of Section~\ref{sec:cover});
\item\label{V:cover} \ref{H:cover}, by expanding both sides of the identity \eqref{eq:sos} and
comparing them coefficient by coefficient (steps \ref{S:shape}, \ref{S:assemble} and
\ref{S:identity} of Section~\ref{sec:cover}). The conclusion follows by Theorem~\ref{thm:sos}.
\end{enumerate}

Every one of these checks is exact. The verifier uses no floating-point arithmetic. Integers are
held in fixed-width machine words of $128$, $256$ or $512$ bits, selected from the deposited $L$.
Each operation is checked for overflow. For rational numbers, we use two
representations. If the number is not dyadic, then we store the numerator and denominator. In the
dyadic case the denominator is a power of two, and we keep only its exponent. Addition,
multiplication and comparison are implemented so that no rounding can occur, and every operation is
guarded by an overflow test. Any guard that fires sets a permanent flag. A flagged run reports the
fact and its output is discarded.

\subsection*{The verifier's contract}
The verifier is a single program with a one-line verdict. Invoked without arguments it prints its
own contract, which we reproduce in substance:

\begin{quote}
The input file contains, among other data, a positive integer \texttt{targetBound} and positive
integers \texttt{tauSqNum}, \texttt{tauSqDen} with $\texttt{tauSqNum}<4\cdot\texttt{tauSqDen}$.
Define $\tau=\sqrt{\texttt{tauSqNum}/\texttt{tauSqDen}}$ and $\omega=\tau/(2-\tau)$. If the last
line printed by the program is \texttt{THE THEOREM WAS PROVED}, then the following statement is a
theorem: \emph{no more than \texttt{targetBound} pairwise non-overlapping cylinders of radius
$\omega$ can simultaneously touch the unit ball.} Everything else in the input file is a
machine-generated witness. No modification of the input, however adversarial, can make the program
print the verdict for a false statement. The only realistic effect of tampering is that the proof
is not found.
\end{quote}
In the language of Definition~\ref{def:Pcyl}: the verdict certifies
$(\texttt{targetBound},\,\omega)\in\Pcyl$.

The final line is printed only if all of \ref{V:dom}--\ref{V:cover} succeed \emph{and} every
arithmetic flag is still clear at the end of the run. Any failure, at any stage, suppresses it.
The soundness of the verdict therefore rests on the verifier source, the compiler and hardware
executing it, and Theorem~\ref{thm:certificate}. 
Two input switches, \texttt{skipMultVerification} and \texttt{skipOrthVerification},
allow the universal checks in \ref{V:tab} to be skipped when the tables have been verified
separately. If either switch is active, the verdict is conditional on those identities. Both
switches are off in all four deposited certificates.

The window configurations of Section~\ref{sec:window} are deposited in the format described
there, and are checked together with the inequalities of Table~\ref{tab:checks} by the run
\texttt{cylie -i constructions} of Section~\ref{subsec:checks}.

The rational configuration of Remark~\ref{rem:five} is deposited in \texttt{cyl\_window\_E5.txt}.
The file records five rows and a rational number $\tau=\frac{13600}{11513}$. The same run checks that
the rows are rational tangent lines with pairwise distances at least $\tau$. Moreover
$5-2\tau^2>0$ and $(5-2\tau^2)^2<5$, so $\tau^2>\frac{5-\sqrt5}2=4\sin^2\frac\pi5$. Hence the five lines
give five non-overlapping cylinders of radius $\frac{\tau}{2-\tau}>\omega_5$.

\section*{Acknowledgements}
We thank Frank Vallentin and Fernando M\'ario de Oliveira Filho for many discussions of the
problem and the paper. Frank Vallentin suggested replacing the subdivision of the safe zone by a
sum-of-squares certificate. This made the verification of Theorem~\ref{thm:certificate} fast and
algebraic. The semidefinite programs that produced the certificates were solved with CSDP
\cite{Borchers}.

\section*{On the use of AI}
Ideas were discussed and the literature was reviewed with the help of Anthropic's Claude
models. The programs that searched for certificates and for the configurations of
Section~\ref{sec:window} were written with the assistance of these models. The verifier
\cite{MRVer} was written by the authors and reviewed and tested both by the authors and by
the models. The proofs use the certificates and configurations only through the verifier, and
do not depend on how they were found. The text was written by the authors and then proofread
and edited with the same models.


\begin{thebibliography}{99} 
\bibitem{Bochner} S.~Bochner, \emph{Hilbert distances and positive definite functions}, Ann.\ of
Math.\ (2) \textbf{42} (1941), 647--656.

\bibitem{BtD} T.~Br\"ocker and T.~tom Dieck, \emph{Representations of Compact Lie Groups},
Graduate Texts in Mathematics~98, Springer, 1985.

\bibitem{Borchers} B.~Borchers, \emph{CSDP, a C library for semidefinite programming}, Optim.\
Methods Softw.\ \textbf{11} (1999), 613--623.

\bibitem{BMP} P.~Bra\ss, W.~Moser, and J.~Pach, \emph{Research Problems in Discrete Geometry},
Springer, New York, 2005.

\bibitem{BW} P.~Bra\ss{} and C.~Wenk, \emph{On the number of cylinders touching a ball}, Geom.\
Dedicata \textbf{81} (2000), 281--284.

\bibitem{DGS} P.~Delsarte, J.~M.~Goethals, and J.~J.~Seidel, \emph{Spherical codes and designs},
Geom.\ Dedicata \textbf{6} (1977), 363--388.

\bibitem{F} M.~Firsching, \emph{Optimization methods in discrete geometry}, PhD thesis, Freie
Universit\"at Berlin (2016).

\bibitem{Gershgorin} S.~Gerschgorin, \emph{\"Uber die Abgrenzung der Eigenwerte einer Matrix},
Izv.\ Akad.\ Nauk SSSR Ser.\ Mat.\ (1931), no.~6, 749--754.

\bibitem{HS} A.~Heppes and L.~Szab\'o, \emph{On the number of cylinders touching a ball},
Geom.\ Dedicata \textbf{40} (1991), 111--116.

\bibitem{K} W.~Kuperberg, \emph{How many unit cylinders can touch a unit ball?}, Problem 3.3,
DIMACS Workshop on Polytopes and Convex Sets, Rutgers University, January 10, 1990.

\bibitem{dLV} D.~de Laat and F.~Vallentin, \emph{A semidefinite programming hierarchy for
 packing problems in discrete geometry}, Math.\ Program.\ \textbf{151} (2015), 529--553. 

\bibitem{Lasserre} J.~B.~Lasserre, \emph{Global optimization with polynomials and the problem
of moments}, SIAM J.\ Optim.\ \textbf{11} (2001), 796--817.


\bibitem{Lovasz} L.~Lov\'asz, \emph{On the Shannon capacity of a graph}, IEEE Trans.\ Inform.\
Theory \textbf{25} (1979), 1--7.

\bibitem{MRVer} I.~Mati\'c, R.~Radoi\v ci\'c, and A.~Spomer, \emph{cylie: a standalone exact-arithmetic
verifier for cylinder insertions and expansions} (computer program),
\verb+https://github.com/maticivan/cylie+ (2026).   

\bibitem{MRCode2026} I.~Mati\'c and R.~Radoi\v ci\'c, \emph{Source code and certificates for
computer-assisted bounds on cylinders touching a ball},
\verb+https://github.com/maticivan/cyl+ (2026),
\url{https://doi.org/10.5281/zenodo.21706193}.


\bibitem{MRsix} I.~Mati\'c and R.~Radoi\v ci\'c, \emph{A computer-assisted proof of Kuperberg's
six-cylinder conjecture}, preprint (2026). 


\bibitem{OSMorse} O.~Ogievetsky and S.~Shlosman, \emph{Critical configurations of solid bodies and the
Morse theory of MIN functions}, Russian Math.\ Surveys \textbf{74} (2019), no.~4, 631--657;
arXiv:1907.01896.

\bibitem{OSsix} O.~Ogievetsky and S.~Shlosman, \emph{The six cylinders problem:
$\mathbb{D}_3$-symmetry approach}, Discrete Comput.\ Geom.\ \textbf{65} (2021), no.~2,
385--404; arXiv:1805.09833.

\bibitem{OSCm} O.~Ogievetsky and S.~Shlosman, \emph{Extremal cylinder configurations I:
configuration $C_{\mathfrak m}$}, Discrete Comput.\ Geom.\ \textbf{66} (2021), no.~1, 140--164;
arXiv:1812.09543.

\bibitem{OSO6} O.~Ogievetsky and S.~Shlosman, \emph{Extremal cylinder configurations II:
configuration $O_6$}, Exp.\ Math.\ \textbf{32} (2023), no.~2, 486--496 (published online
2019); arXiv:1902.08995.


\bibitem{OSPlatonic} O.~Ogievetsky and S.~Shlosman, \emph{Platonic compounds of cylinders},
in: Integrability, Quantization, and Geometry II, Proc.\ Sympos.\ Pure Math.\ \textbf{103.2},
Amer.\ Math.\ Soc.\ (2021), 447--480; arXiv:1904.02043.

\bibitem{OSunlocking} O.~Ogievetsky and S.~Shlosman, \emph{The art of unlocking}, 
Math.\ Intelligencer \textbf{44} (2022), 320--325, arXiv:2202.12013.

\bibitem{OV} F.~M.\ de Oliveira Filho and F.~Vallentin, \emph{Mathematical optimization for packing
problems}, arXiv:1403.1166 (2014).

\bibitem{Parrilo} P.~A.~Parrilo, \emph{Semidefinite programming relaxations for semialgebraic
problems}, Math.\ Program.\ \textbf{96} (2003), 293--320.


\bibitem{SOV} F.~M.\ de Oliveira Filho, A.~Spomer, and F.~Vallentin, \emph{Bounding the density of spherical polygon packings}, arXiv:2604.21451 (2026).

\bibitem{Spomer} A.~Spomer, \emph{Packing regular spherical polygons}, PhD thesis, Universit\"at zu
K\"oln, 2026.

\bibitem{VICM} F.~Vallentin, \emph{Conic optimization for extremal geometry}, arXiv:2510.06960
(2025), to appear in Proc. ICM 2026.

\bibitem{Yar} O.~Yardimci, \emph{On the number of cylinders touching a sphere}, PhD
dissertation, Auburn University (2019); 



\end{thebibliography}
\end{document}